\documentclass[smallextended,final]{svjour3}
\usepackage[letterpaper,top=2in, bottom=1.5in, left=1in, right=1in]{geometry}

\usepackage{epsfig,epsf,fancybox}
\usepackage{amsmath}
\usepackage{mathrsfs}
\usepackage{amssymb}
\usepackage{amsfonts}
\usepackage{graphicx}
\usepackage{color, colortbl}
\usepackage[linktocpage,pagebackref,colorlinks,linkcolor=blue,anchorcolor=blue,citecolor=blue,urlcolor=blue,hypertexnames=false]{hyperref}
\usepackage{boxedminipage}
\usepackage{stmaryrd}
\usepackage{multirow}
\usepackage{booktabs}
\usepackage{cite}
\usepackage{float}
\usepackage[ruled,vlined,linesnumbered]{algorithm2e}

\usepackage[shortlabels]{enumitem}

\definecolor{LightCyan}{rgb}{0.88,1,1}

\usepackage{mathtools}

\usepackage[normalem]{ulem} 

\newcommand{\bm}[1]{\boldsymbol{#1}}

\newcommand{\vd}{{\mathbf{d}}}
\newcommand{\ve}{{\mathbf{e}}}

\newcommand{\vg}{{\mathbf{g}}}

\newcommand{\vq}{{\mathbf{q}}}

\newcommand{\vu}{{\mathbf{u}}}
\newcommand{\vv}{{\mathbf{v}}}
\newcommand{\vw}{{\mathbf{w}}}
\newcommand{\vx}{{\mathbf{x}}}
\newcommand{\vy}{{\mathbf{y}}}
\newcommand{\vz}{{\mathbf{z}}}

\newcommand{\vI}{{\mathbf{I}}}

\newcommand{\vW}{{\mathbf{W}}}

\newcommand{\vtheta}{{\bm{\theta}}}

\newcommand{\cN}{{\mathcal{N}}}

\newcommand{\vareps}{\varepsilon}

\newcommand{\EE}{\mathbb{E}} 
\newcommand{\RR}{\mathbb{R}} 
\newcommand{\vzero}{\mathbf{0}} 
\newcommand{\vone}{{\mathbf{1}}} 

\newcommand{\Prob}{{\mathrm{Prob}}} 
\newcommand{\dom}{{\mathrm{dom}}} 

\newcommand{\st}{\mbox{ s.t. }}

\DeclareMathOperator*{\argmin}{arg\,min} 

\DeclareMathOperator*{\Min}{minimize}
\DeclareMathOperator*{\Max}{maximize}

\newcommand{\bc}{\begin{center}}
\newcommand{\ec}{\end{center}}

\newcommand{\bdm}{\begin{displaymath}}
\newcommand{\edm}{\end{displaymath}}

\newcommand{\beq}{\begin{equation}}
\newcommand{\eeq}{\end{equation}}

\newcommand{\bfl}{\begin{flushleft}}
\newcommand{\efl}{\end{flushleft}}

\newcommand{\bt}{\begin{tabbing}}
\newcommand{\et}{\end{tabbing}}

\newcommand{\beqn}{\begin{eqnarray}}
\newcommand{\eeqn}{\end{eqnarray}}

\newcommand{\beqs}{\begin{align*}} 
\newcommand{\eeqs}{\end{align*}}  

\newtheorem{assumption}{Assumption}

\numberwithin{equation}{section}

\begin{document}

\title{Momentum-based variance-reduced proximal stochastic gradient method for composite nonconvex stochastic optimization}

\author{Yangyang Xu \and Yibo Xu}

\institute{Yangyang Xu \at Department of Mathematical Sciences, Rensselaer Polytechnic Institute, Troy, NY 12180\\
\email{xuy21@rpi.edu}\\
Yibo Xu \at School of Mathematical and Statistical Sciences,  Clemson University, Clemson, SC 29634\\
\email{yibox@clemson.edu}
}

\date{\today}

\maketitle

\begin{abstract}
Stochastic gradient methods (SGMs) have been extensively used for solving stochastic problems or large-scale machine learning problems. Recent works employ various techniques to improve the convergence rate of SGMs for both convex and nonconvex cases. Most of them require a large number of samples in some or all iterations of the improved SGMs. In this paper, we propose a new SGM, named PStorm, for solving nonconvex nonsmooth stochastic problems. With a momentum-based variance reduction technique, PStorm can achieve the optimal complexity result $O(\vareps^{-3})$ to produce a stochastic $\vareps$-stationary solution, if a mean-squared smoothness condition holds. Different from existing optimal methods, PStorm can achieve the ${O}(\vareps^{-3})$ result by using only one or $O(1)$ samples in every update. With this property, PStorm can be applied to online learning problems that favor real-time decisions based on one or $O(1)$ new observations. In addition, for large-scale machine learning problems, PStorm can generalize better by small-batch training than other optimal methods that require large-batch training and the vanilla SGM, as we demonstrate on training a sparse fully-connected neural network and a sparse convolutional neural network.

\vspace{0.3cm}

\vspace{0.3cm}


\end{abstract}

\section{Introduction}
The stochastic approximation method first appears in \cite{robbins1951stochastic} for solving a root-finding problem. Nowadays, its first-order version, or the stochastic gradient method (SGM), has been extensively used to solve machine learning problems that involve huge amounts of given data and also to stochastic problems that involve uncertain streaming data. Complexity results of SGMs have been well established for convex problems. Many recent research papers on SGMs focus on nonconvex cases.

In this paper, we consider the regularized nonconvex stochastic programming
\begin{equation}\label{eq:stoc-prob}
\Phi^*= \Min_{\vx\in\RR^n} ~\Phi(\vx):= \big\{F(\vx)\equiv\EE_\xi [f(\vx;\xi)]\big\} + r(\vx), 
\end{equation}
where $f(\,\cdot\,; \xi)$ is a smooth nonconvex function almost surely for $\xi$, and $r$ is a closed convex function on $\RR^n$. Examples of \eqref{eq:stoc-prob} include the sparse online matrix factorization \cite{mairal2010online}, the online nonnegative matrix factorization \cite{zhao2016online}, and the streaming PCA (by a unit-ball constraint) \cite{mitliagkas2013memory}. In addition, as $\xi$ follows a uniform distribution on a finite set $\Xi=\{\xi_1,\ldots,\xi_N\}$, \eqref{eq:stoc-prob} recovers the so-called finite-sum structured problem. It includes most regularized machine learning problems such as the sparse bilinear logistic regression \cite{shi2014sparse}, the sparse convolutional neural network \cite{liu2015sparse}, and the group sparse regularized deep neural networks \cite{scardapane2017group}. 

\subsection{Background} When $r\equiv 0$, the recent work \cite{arjevani2019lower} gives an $O(\vareps^{-3})$ lower complexity bound of SGMs to produce a stochastic $\vareps$-stationary solution of \eqref{eq:stoc-prob} (see Definition~\ref{def:eps-sol} below), by assuming the so-called mean-squared smoothness condition (see Assumption~\ref{assump-smooth}). Several variance-reduced SGMs \cite{tran2021hybrid, wang2019spiderboost, fang2018spider, cutkosky2019momentum} have achieved an $O(\vareps^{-3})$ or $\tilde O(\vareps^{-3})$ complexity result\footnote{Throughout the paper, we use $\tilde O$ to suppress an additional polynomial term of $|\log\vareps|$}. Among them, \cite{fang2018spider, cutkosky2019momentum} only consider smooth cases, i.e., $r\equiv 0$ in \eqref{eq:stoc-prob}, and \cite{tran2021hybrid, wang2019spiderboost} study nonsmooth problems in the form of \eqref{eq:stoc-prob}. To reach an $O(\vareps^{-3})$ complexity result, the Hybrid-SGD method in \cite{tran2021hybrid} needs $O(\vareps^{-1})$ samples at the initial step and then at least two samples at each update, while \cite{wang2019spiderboost, fang2018spider} require $O(\vareps^{-2})$ samples after every fixed number of updates. The STORM method in \cite{cutkosky2019momentum} requires one single sample of $\xi$ at each update, but it only applies to smooth problems. Practically on training a (deep) machine learning model, small-batch training is often used to have better generalization \cite{masters2018revisiting, keskar2016large}. In addition, for certain applications such as reinforcement learning \cite{sutton2018reinforcement}, one single sample can usually be obtained, depending on the stochastic environment and the current decision. Furthermore, regularization terms can improve generalization of a machine learning model, even for training a neural network \cite{wei2019regularization}. We aim at designing a new SGM for solving the nonconvex nonsmooth problem \eqref{eq:stoc-prob} and achieving a (near)-optimal\footnote{By ``optimal'', we mean that the complexity result can reach the lower bound result; a result is ``near optimal'', if it has an additional logarithmic term or a polynomial of logarithmic term than the lower bound.} complexity result by using $O(1)$ (that can be \emph{one}) samples at each update.

\subsection{Mirror-prox Algorithm}
Our algorithm is a mirror-prox SGM, and we adopt the momentum technique to reduce variance of the stochastic gradient in order to achieve a (near)-optimal complexity result.

Let $w$ be a continuously differentiable and 1-strongly convex function on $\dom(r)$, i.e., 
$$w(\vy) \ge w(\vx) + \langle \nabla w(\vx), \vy-\vx\rangle + \frac{1}{2}\|\vy-\vx\|^2,\, \forall\, \vx,\vy\in \dom(r).$$
The Bregman divergence induced by $w$ is defined as
\begin{equation}
V(\vx,\vz)=w(\vx)-w(\vz) - \langle \nabla w(\vz), \vx-\vz\rangle.
\end{equation}
At each iteration of our algorithm, we obtain one or a few samples of $\xi$, compute stochastic gradients at the previous and current iterates using the same samples, and then perform a mirror-prox momentum stochastic gradient update. The pseudocode is shown in Algorithm~\ref{alg:prox-storm}. We name it as PStorm as it can be viewed as a proximal version of the Storm method in \cite{cutkosky2019momentum}. Notice that when $\beta_k=1,\forall\, k\ge0$, the algorithm reduces to the non-accelerated stochastic proximal gradient method. However, our analysis does not apply to this case, for which an innovative analysis can be found in \cite{davis2019stochastic}.

\begin{algorithm}[h]
\caption{Momentum-based variance-reduced proximal stochastic gradient method for \eqref{eq:stoc-prob}}\label{alg:prox-storm}
\DontPrintSemicolon
\textbf{Input:} max iteration numer $K$, minibatch size $m$, and positive sequences $\{\beta_k\}\subseteq(0,1)$ and $\{\eta_k\}$.\;
\textbf{Initialization:} choose $\vx^0\in\dom(r)$ and let $\vd^0 = \frac{1}{m_0}\sum_{\xi\in B_0}\nabla f(\vx^0;\xi)$ with $m_0$ i.i.d. samples $B_0=\{\xi_1^0,\ldots,\xi_{m_0}^0\}$\;
\For{$k=0, 1,\ldots, K-1$}{
Update $\vx$ by
\begin{equation}\label{eq:update-x}
\vx^{k+1}=\argmin_{\vx} \left\{\langle \vd^k, \vx\rangle + \frac{1}{\eta_k}V(\vx,\vx^k) + r(\vx)\right\}.\vspace{-0.3cm}
\end{equation}
\;
Obtain $m$ i.i.d. samples $B_{k+1}=\{\xi_1^{k+1},\ldots,\xi_m^{k+1}\}$ and let 
\begin{equation}\label{eq:def-v-u-k}
\textstyle \vv^{k+1} = \frac{1}{m}\sum_{\xi\in B_{k+1}}\nabla f(\vx^{k+1}; \xi),\quad \vu^{k+1} = \frac{1}{m}\sum_{\xi\in B_{k+1}}\nabla f(\vx^{k}; \xi).\vspace{-0.3cm}
\end{equation}\;
Let $\vd^{k+1} = \vv^{k+1} + (1-\beta_{k})(\vd^{k}-\vu^{k+1})$.\;
}
Return $\vx^\tau$ with $\tau$ selected from $\{0,1,\ldots,K-1\}$ uniformly at random or by the distribution
\begin{equation}\label{eq:select-tau}
\textstyle \Prob(\tau = k) = \frac{\frac{\eta_k}{4}(1-\eta_k L)-\frac{\eta_k^2}{5m\eta_{k+1}}(1-\beta_k)^2}{\sum_{j=0}^{K-1}\left(\frac{\eta_j}{4}(1-\eta_j L)-\frac{\eta_j^2}{5m\eta_{j+1}}(1-\beta_j)^2\right)},\, k = 0,1,\ldots,K-1.
\end{equation}
\vspace{-0.2cm}
\end{algorithm}

\subsection{Related Works}
Many efforts have been made on analyzing the convergence and complexity of SGMs for solving nonconvex stochastic problems, e.g., \cite{ghadimi2016accelerated, ghadimi2013stochastic, xu2015block-sg, davis2019stochastic, davis2020stochastic, wang2019spiderboost, cutkosky2019momentum, fang2018spider, allen2018natasha, tran2021hybrid}. We list comparison results on the complexity in Table~\ref{table:comparison}.

\begin{table}
\begin{center}
\resizebox{0.99\textwidth}{!}{
\begin{tabular}{|c|c|c|c|c|}
\hline
\multirow{2}{*}{\textbf{Method}} & \multirow{2}{*}{\textbf{problem}} & \multirow{2}{*}{\textbf{key assumption}} & \textbf{\#samples} & \multirow{2}{*}{\textbf{complexity}}\\
& & &\textbf{at $k$-th iteration} &\\\hline\hline
\multirow{2}{*}{accelerated prox-SGM \cite{ghadimi2016accelerated}} & \multirow{2}{*}{$\min_\vx\{\EE_\xi[f(\vx;\xi)]+r(\vx)\}$} & $\EE_\xi[f(\vx;\xi)]$ is smooth & \multirow{2}{*}{$\Theta(k)$} & \multirow{2}{*}{$O(\vareps^{-4})$}\\
& &  $r$ is convex & & \\\hline
\multirow{3}{*}{stochastic subgradient \cite{davis2019stochastic}} & \multirow{3}{*}{$\min_\vx\{\EE_\xi[f(\vx;\xi)]+r(\vx)\}$} & $\EE_\xi[f(\vx;\xi)]$ is weakly-convex & \multirow{3}{*}{$O(1)$} & \multirow{3}{*}{$O(\vareps^{-4})$}\\
& &  $r$ is convex & & \\
& &  bounded stochastic subgrad. & & \\\hline
\multirow{2}{*}{Spider \cite{fang2018spider}} & \multirow{2}{*}{$\min_\vx\{\EE_\xi[f(\vx;\xi)]\}$} & mean-squared smoothness & $\Theta(\vareps^{-2})$  & \multirow{2}{*}{$O(\vareps^{-3})$} \\
& &  see Assumption~\ref{assump-smooth} & or $\Theta(\vareps^{-1})$ & \\\hline
\multirow{2}{*}{Storm \cite{cutkosky2019momentum} } & \multirow{2}{*}{$\min_\vx\{\EE_\xi[f(\vx;\xi)]\}$ } & $f(\,\cdot\,;\xi)$ is smooth a.s. & \multirow{2}{*}{1} & \multirow{2}{*}{$\tilde O(\vareps^{-3})$} \\
& &  bounded stochastic grad.$^*$ & & \\\hline
\multirow{2}{*}{Spiderboost \cite{wang2019spiderboost} }& \multirow{2}{*}{$\min_\vx\{\EE_\xi[f(\vx;\xi)]+r(\vx)\}$ }& mean-squared smoothness & $\Theta(\vareps^{-2})$ & \multirow{2}{*}{ $O(\vareps^{-3})$}\\
& & $r$ is convex & or $\Theta(\vareps^{-1})$  & \\\hline
\multirow{2}{*}{Hybrid-SGD \cite{tran2021hybrid} } & \multirow{2}{*}{$\min_\vx\{\EE_\xi[f(\vx;\xi)]+r(\vx)\}$} & mean-squared smoothness  & $\Theta(\vareps^{-1})$ if $k=0$ & \multirow{2}{*}{$O(\vareps^{-3})$}\\
& & $r$ is convex & $O(1)$ but at least 2 if $k>0$ & \\\hline
\multirow{4}{*}{PStorm (\textbf{This paper})} & \multirow{4}{*}{$\min_\vx\{\EE_\xi[f(\vx;\xi)]+r(\vx)\}$} &   & $O(1)$ and can be 1  & \multirow{2}{*}{$\tilde O(\vareps^{-3})$}\\
& & mean-squared smoothness & varying stepsize & \\\cline{4-5}
& & $r$ is convex & $O(1)$ and can be 1 & \multirow{2}{*}{$O(\vareps^{-3})$} \\
& & & constant stepsize & \\\hline
\end{tabular}
}
\end{center}
\caption{Comparison of the complexity results of several methods in the literature to our method to produce a stochastic $\vareps$-stationary solution of a nonconvex stochastic optimization problem. To obtain the listed results, all the compared methods assume unbiasedness and variance boundedness of the stochastic (sub)gradients. The results only show the dependence on $\vareps$. All other constants (e.g., the smoothness constant $L$ and the initial objective error) are hidden in the big-$O$. More complete results of the proposed method are given in Remarks~\ref{rm:dep-sigma-dynamic} and \ref{rm:dep-sigma-const}.
}
$^*$: the boundedness assumption on stochastic gradient made by Storm \cite{cutkosky2019momentum} can be lifted if a bound $\sigma$ on the variance of the stochastic gradient is known. 
\label{table:comparison}
\end{table}

The work \cite{ghadimi2013stochastic} appears to be the first one that conducts complexity analysis of SGM for nonconvex stochastic problems. It introduces a randomized SGM. For a smooth nonconvex problem, the randomized SGM can produce a stochastic $\vareps$-stationary solution within $O(\vareps^{-4})$ SG iterations. The same-order complexity result is then extended in \cite{ghadimi2016accelerated} to nonsmooth nonconvex stochastic problems in the form of \eqref{eq:stoc-prob}. To achieve an $O(\vareps^{-4})$ complexity result, the accelerated prox-SGM in \cite{ghadimi2016accelerated} needs to take $\Theta(k)$ samples at the $k$-th update for each $k$. Assuming a weak-convexity condition and using the tool of Moreau envelope, \cite{davis2019stochastic} establishes an $O(\vareps^{-4})$ complexity result of stochastic subgradient method for solving more general nonsmooth nonconvex problems to produce a near-$\vareps$ stochastic stationary solution (see \cite{davis2019stochastic} for the precise definition).  

In general, the $O(\vareps^{-4})$ complexity result cannot be improved for smooth nonconvex stochastic problems, as \cite{arjevani2019lower} shows that for the problem $\min_\vx F(\vx)$ where $F$ is smooth, any SGM that can access unbiased SG with bounded variance needs $\Omega(\vareps^{-4})$ SGs to produce a solution $\bar\vx$ such that $\EE\big[\|\nabla F(\bar\vx)\|\big] \le \vareps$. However, with one additional mean-squared smoothness condition on each unbiased SG, the complexity can be reduced to $O(\vareps^{-3})$, which has been reached by a few variance-reduced SGMs \cite{tran2021hybrid, wang2019spiderboost, fang2018spider, cutkosky2019momentum, pham2020proxsarah}. These methods are closely related to ours. Below we briefly review them. 

\noindent\textbf{Spider.} To find a stochastic $\vareps$-stationary solution of \eqref{eq:stoc-prob} with $r\equiv0$,  \cite{fang2018spider} proposes the Spider method with the update: 
$\vx^{k+1}=\vx^k - \eta_k \vv^k$ for each $k\ge0$.
Here, $\vv^k$ is set to
\begin{equation}\label{eq:vk-spider}
\vv^k = \left\{
\begin{array}{ll}
\frac{1}{|B_k|}\sum_{\xi\in B_k}\left(\nabla f(\vx^k;\xi)-\nabla f(\vx^{k-1};\xi)\right) + \vv^{k-1}, & \text{ if } \mathrm{mod}(k, q)\neq 0, \\[0.2cm]
\frac{1}{|C_k|}\sum_{\xi\in C_k}\nabla f(\vx^k;\xi), & \text{ otherwise},
\end{array}
\right.
\end{equation}
where $|B_k|=\Theta(\frac{1}{q\vareps^{2}})$, $|C_k|=\Theta(\vareps^{-2})$, and $q=\Theta(\vareps^{-1})$ or $q=\Theta(\vareps^{-2})$. Under the mean-squared smoothness condition (see Assumption~\ref{assump-smooth}),  the Spider method can produce a stochastic $\vareps$-stationary solution with $O(\vareps^{-3})$ sample gradients, by choosing appropriate learning rate $\eta_k$ (roughly in the order of $\frac{1}{q\|\vv^k\|}$).

\noindent\textbf{Storm.} \cite{cutkosky2019momentum} focuses on a smooth nonconvex stochastic problem, i.e., \eqref{eq:stoc-prob} with $r\equiv0$. It proposes the Storm method, which can be viewed as a special case of Algorithm~\ref{alg:prox-storm} with $m_0=m=1$ applied to the smooth problem. However, its analysis and also algorithm design rely on the knowledge of a uniform bound on $\{\|\nabla f(\vx;\xi)\|\}$ or on the bound of the variance of the stochastic gradient. In addition, because the learning rate of Storm is set dependent on the sampled stochastic gradient, its analysis needs almost-sure uniform smoothness of $f(\vx;\xi)$. This assumption is significantly stronger than the mean-squared smoothness condition, and also the uniform smoothness constant can be much larger than an averaged one. 

\noindent\textbf{Spiderboost.} \cite{wang2019spiderboost} extends Spider into solving a nonsmooth nonconvex stochastic problem in the form of \eqref{eq:stoc-prob} by proposing a so-called Spiderboost method. Spiderboost iteratively performs the update
\begin{equation}\label{eq:spiderboost}
\textstyle \vx^{k+1}=\argmin_\vx  \langle \vv^k, \vx\rangle + \frac{1}{\eta}V(\vx,\vx^k) + r(\vx),
\end{equation}
where $V$ denotes the Bregman divergence induced by a strongly-convex function, and $\vv^k$ is set by \eqref{eq:vk-spider} with $q=|B_k| = \Theta(\vareps^{-1})$ and $|C_k|=\Theta(\vareps^{-2})$. Under the mean-squared smoothness condition, Spiderboost reaches a complexity result of $O(\vareps^{-3})$ by choosing $\eta=\frac{1}{2L}$, where $L$ is the smoothness constant.

\noindent\textbf{Hybrid-SGD.} \cite{tran2021hybrid} considers a nonsmooth nonconvex stochastic problem in the form of \eqref{eq:stoc-prob}. It proposes a proximal stochastic method, called Hybrid-SGD, as a hybrid of SARAH \cite{nguyen2017sarah} and an unbiased SGD. The Hybrid-SGD performs the update $\vx^{k+1} = (1-\gamma_k)\vx^k + \gamma_k \hat\vx^{k+1}$ for each $k\ge 0$, where
$$\textstyle \hat\vx^{k+1}=\argmin_\vx \langle \vv^k, \vx\rangle + \frac{1}{2\eta_k}\|\vx-\vx^k\|^2 + r(\vx).$$
Here, the sequence $\{\vv^k\}$ is set by $\vv^0=\frac{1}{|B_0|}\sum_{\xi\in B_0}\nabla f(\vx^0;\xi)$ with $|B_0| = \Theta(\vareps^{-1})$ for a given $\vareps>0$ and
\begin{equation}\label{eq:hyb-sgd-v}
\textstyle \vv^k = \beta_{k-1}\vv^{k-1}+\beta_{k-1}\big(\nabla f(\vx^k;\xi_k)-\nabla f(\vx^{k-1};\xi_k)\big) + (1-\beta_{k-1})\nabla f(\vx^k;\zeta_k),
\end{equation}
where $\xi_k$ and $\zeta_k$ are two independent samples of $\xi$. A mini-batch version of Hybrid-SGD is also given in \cite{tran2021hybrid}. By choosing appropriate constant parameters $\{(\beta_k,\gamma_k,\eta_k)\}$, Hybrid-SGD can reach an $O(\vareps^{-3})$ complexity result. Although the update of $\vv^k$ requires only two or $O(1)$ samples, its initial setting needs $O(\vareps^{-1})$ samples. As explained in \cite[Remark 3]{tran2021hybrid}, if the initial minibatch size is $|B_0|=O(1)$, then the complexity result of Hybrid-SGD will be worsened to $O(\vareps^{-4})$. It is possible to reduce the $O(\vareps^{-4})$ complexity by using an adaptive $\beta_k$ as mentioned in \cite[Remark 3]{tran2021hybrid} to adopt the technique in \cite{tran2020hybrid-minimax}. This way, a near-optimal $\tilde O(\vareps^{-3})$ result may be shown for Hybrid-SGD without a large initial minibatch. Notice that with $\xi_k=\zeta_k, \forall\, k$, the stochastic gradient estimator by Hybrid-SGD will reduce to that by Storm, and further with $\gamma_k=1,\, \forall\, k$, the update of Hybrid-SGD will recover ours. However, the analysis in \cite{tran2020hybrid-minimax, tran2021hybrid} relies on the independence of $\xi_k$ and $\zeta_k$ and the condition $\gamma_k\in (0,1)$, and thus it does not apply to our algorithm.

\noindent\textbf{More.} There are many other works analyzing complexity results of SGMs on solving nonconvex finite-sum structured problems, e.g., \cite{allen2016variance, reddi2016stochastic, lei2017non, huo2016asynchronous}. These results often emphasize the dependence on the number of component functions and also the target error tolerance $\vareps$. In addition, several works have analyzed adaptive SGMs for nonconvex finite-sum or stochastic problems, e.g., \cite{chen2018convergence, zhou2018convergence, xu2020-APAM}. Moreover, along the direction of accelerating SGMs, some works (e.g., \cite{zhang2019stochastic, tran2020hybrid-minimax, xu2021katyusha, zhang2021stochastic}) have considered minimax structured or compositional optimization problems. An exhaustive review of all these works is impossible and also beyond the scope of this paper. We refer interested readers to those papers and the references therein.

\subsection{Contributions} Our main contributions are about the algorithm design and analysis. 
\begin{itemize}
\item We design a momentum-based variance-reduced mirror-prox stochastic gradient method for solving nonconvex nonsmooth stochastic problems. The proposed method generalizes Storm in \cite{cutkosky2019momentum} from smooth cases to nonsmooth cases. In addition, with one single data sample per iteration, it achieves, by taking varying stepsizes, the same near-optimal complexity result $\tilde O(\vareps^{-3})$ under a mean-squared smooth condition, which is weaker than the almost-sure uniform smoothness condition assumed in \cite{cutkosky2019momentum}. 

\item When constant stepsizes are adopted, the proposed method can achieve the optimal $O(\vareps^{-3})$ complexity result, by using one single or $O(1)$ data samples per iteration. While Spiderboost \cite{wang2019spiderboost} can also achieve the optimal $O(\vareps^{-3})$ complexity result for stochastic nonconvex nonsmooth problems, it needs $\Theta(\vareps^{-2})$ data samples every $\Theta(\vareps^{-1})$ iterations and $\Theta(\vareps^{-1})$ samples for every other iteration. To achieve the optimal $O(\vareps^{-3})$ complexity result, Hybrid-SGD \cite{tran2021hybrid} needs $\Theta(\vareps^{-1})$ data samples for the first iteration and at least two samples for all other iterations.  However, if only $O(1)$ samples can be obtained initially, the worst-case complexity result of Hybrid-SGD with constant stepsize will increase to $O(\vareps^{-4})$. Our proposed method is the first one that uses only one or $O(1)$ samples per iteration and can still reach the optimal complexity result, and thus it can be applied to online learning problems that need real-time decision based on possibly one or several new data samples.  

\item Furthermore, the proposed method only needs an estimate of the smoothness parameter and is easy to tune to have good performance. Empirically, we observe that it converges faster than a vanilla SGD and can give higher testing accuracy than Spiderboost and Hybrid-SGD  on training sparse neural networks. 

\end{itemize}

\subsection{Notation, Definitions, and Outline}
We use bold lowercase letters $\vx,\vy,\vg,\ldots$ for vectors. $\EE_{B_{k}}$ denotes the expectation about a mini-batch set $B_{k}$ conditionally on the all previous history, and $\EE$ denotes the full expectation. $|B_k|$ counts the number of elements in the set $B_k$. We use $\|\cdot\|$ for the Euclidean norm. A differentiable function $F$ is called $L$-smooth, if $\|\nabla F(\vx)-\nabla F(\vy)\|\le L\|\vx-\vy\|$ for all $\vx$ and $\vy$.

\begin{definition}[proximal gradient mapping]\label{def:prox-map}
Given $\vd$, $\vx\in\dom(r)$, and $\eta>0$, we define $P(\vx,\vd,\eta)=\frac{1}{\eta}(\vx-\vx^+)$, where 
$\vx^+ =\argmin_\vy \textstyle \left\{\langle \vd, \vy\rangle  +\frac{1}{\eta}V(\vy,\vx) + r(\vy)\right\}.$
\end{definition}

By the proximal gradient mapping, if a point $\bar\vx\in\dom(r)$ is an optimal solution of \eqref{eq:stoc-prob}, then it must satisfy $P(\bar\vx,\nabla F(\bar\vx),\eta)=\vzero$ for any $\eta>0$. Based on this observation, we define a near-stationary solution as follows. This definition is standard and has been adopted in other papers, e.g., \cite{wang2019spiderboost}.

\begin{definition}[stochastic $\vareps$-stationary solution]\label{def:eps-sol}
Given $\vareps>0$, a random vector $\vx\in \dom(r)$ is called a stochastic $\vareps$-stationary solution of \eqref{eq:stoc-prob} if for some $\eta>0$, it holds $\EE[\|P(\vx,\nabla F(\vx),\eta)\|^2]\le \vareps^2$.
\end{definition}

From \cite[Lemma 1]{ghadimi2016mini}, it holds
\begin{equation}\label{eq:lem1-lan}
\big\langle \vd, P(\vx,\vd,\eta)\big\rangle \ge \|P(\vx,\vd,\eta)\|^2 + \frac{1}{\eta}\big(r(\vx^+)-r(\vx)\big).
\end{equation}
In addition, the proximal gradient mapping is nonexpansive from \cite[Proposition 1]{ghadimi2016mini}, i.e.,
\begin{equation}\label{eq:nonexp-P}
\|P(\vx, \vd_1, \eta) - P(\vx, \vd_2, \eta)\| \le \|\vd_1-\vd_2\|, \ \forall\, \vd_1,\vd_2, \, \forall\, \vx\in \dom(r), \, \forall\, \eta > 0.
\end{equation}
For each $k\ge0$, we denote
\begin{equation}\label{eq:vg-vgbar}
\vg^k = P(\vx^k, \vd^k, \eta_k),\quad \bar\vg^k = P(\vx^k, \nabla F(\vx^k), \eta_k).
\end{equation}
Notice that $\|\bar\vg^k\|$ measures the violation of stationarity of $\vx^k$. The gradient error is represented by
\begin{equation}\label{eq:error-grad}
\ve^k = \vd^k - \nabla F(\vx^k).
\end{equation}

\noindent\textbf{Outline.} The rest of the paper is organized as follows. In section~\ref{sec:analysis}, we establish complexity results of Algorithm~\ref{alg:prox-storm}. Numerical experiments are conducted in section~\ref{sec:numerical}, and we conclude the paper in section~\ref{sec:conclusion}.

\section{Convergence Analysis}\label{sec:analysis}
In this section, we analyze the complexity result of Algorithm \ref{alg:prox-storm}. Part of our analysis is inspired from that in \cite{cutkosky2019momentum} and \cite{wang2019spiderboost}. In addition, we give a novel analysis that enables us to obtain the optimal $O(\vareps^{-3})$ complexity result by using $O(1)$ samples every iteration. 
Throughout our analysis, we make the following assumptions.

\begin{assumption}[finite optimal objective]\label{assump-obj}
The optimal objective value $\Phi^*$ of \eqref{eq:stoc-prob} is finite.
\end{assumption}

\begin{assumption}[mean-squared smoothness]\label{assump-smooth}
The function $f(\,\cdot\,;\xi)$ satisfies the mean-squared smoothness condition:  
$\EE_\xi \big[\|\nabla f(\vx;\xi) - \nabla f(\vy;\xi)\|^2\big] \le L^2 \|\vx-\vy\|^2,\, \forall\, \vx,\vy\in \dom(r).$
\end{assumption}

\begin{assumption}[unbiasedness and variance boundedness]\label{assump-sgd}
There is $\sigma>0$ such that 
\begin{align}
\EE_\xi [\nabla f(\vx; \xi)] = \nabla F(\vx),\quad \EE[\|\nabla f(\vx; \xi) - \nabla F(\vx)\|^2] \le \sigma^2, \ \forall\, \vx\in \dom(r).
\end{align}
\end{assumption}

It is easy to show that under Assumptions~\ref{assump-smooth} and \ref{assump-sgd}, the function $F(\vx)=\EE_\xi [f(\vx;\xi)]$ is $L$-smooth; see the arguments at the end of section~2.2 of \cite{tran2021hybrid}. 
We first show a few lemmas. The lemma below estimates one-iteration progress. Its proof follows from \cite{wang2019spiderboost}.

\begin{lemma}[one-iteration progress]\label{lem:obj-diff}
Let $\{\vx^k\}$ be generated from Algorithm~\ref{alg:prox-storm}. If $F$ is $L$-smooth, then
$$\Phi(\vx^{k+1})-\Phi(\vx^k) \le \frac{\eta_k}{2}(2-\eta_k L)\|\ve^k\|^2 - \frac{\eta_k}{4}(1-\eta_k L)\|\bar\vg^k\|^2,\, \forall\, k\ge 1,$$
where $\bar\vg^k$ is defined in \eqref{eq:vg-vgbar}.
\end{lemma}

\begin{proof}
By the $L$-smoothness of $F$ and the definition of $\vg^k$ in \eqref{eq:vg-vgbar}, we have
\begin{equation}\label{eq:lip-F-ineq}
F(\vx^{k+1})-F(\vx^k)\le \langle \nabla F(\vx^k), \vx^{k+1}-\vx^k\rangle + \frac{L}{2}\|\vx^{k+1}-\vx^k\|^2 = -\eta_k\langle \nabla F(\vx^k), \vg^k\rangle + \frac{\eta_k^2L}{2}\|\vg^k\|^2.
\end{equation}
Using the definition of $\ve^k$ in \eqref{eq:error-grad} and the inequality in \eqref{eq:lem1-lan}, we have 
$$-\langle \nabla F(\vx^k), \vg^k\rangle = \langle \ve^k, \vg^k\rangle - \langle \vd^k, \vg^k\rangle \le \langle \ve^k, \vg^k\rangle - \|\vg^k\|^2 + \frac{1}{\eta_k}\big(r(\vx^k)-r(\vx^{k+1})\big).$$
Plugging the above inequality into \eqref{eq:lip-F-ineq} and rearranging terms give
$$\Phi(\vx^{k+1})-\Phi(\vx^k) \le \eta_k\langle \ve^k, \vg^k\rangle - \eta_k \|\vg^k\|^2+ \frac{\eta_k^2L}{2}\|\vg^k\|^2.$$
By the Cauchy-Schwartz inequality, it holds $\eta_k\langle \ve^k, \vg^k\rangle \le \frac{\eta_k}{2}\|\ve^k\|^2 + \frac{\eta_k}{2}\|\vg^k\|^2$, which together with the above inequality implies
\begin{equation}\label{eq:chg-obj-0}
\Phi(\vx^{k+1})-\Phi(\vx^k) \le \frac{\eta_k}{2}\|\ve^k\|^2 - \frac{\eta_k}{2}(1-\eta_k L)\|\vg^k\|^2.
\end{equation}
From \eqref{eq:nonexp-P} and the definitions of $\vg^k$ and $\bar\vg^k$ in \eqref{eq:vg-vgbar}, it follows
\begin{equation}\label{eq:tri-g}
-\|\vg^k\|^2 \le -\frac{1}{2}\|\bar\vg^k\|^2 + \|\vg^k-\bar\vg^k\|^2\le -\frac{1}{2}\|\bar\vg^k\|^2 + \|\vd^k -\nabla F(\vx^k)\|^2 = -\frac{1}{2}\|\bar\vg^k\|^2 + \|\ve^k\|^2.
\end{equation}
Now plug the above inequality into \eqref{eq:chg-obj-0} to give the desired result.
\end{proof}

The next lemma gives a recursive bound on the gradient error vector sequence $\{\ve^k\}$. Its proof follows that of \cite[Lemma 2]{cutkosky2019momentum}.
\begin{lemma}[recursive bound on gradient error]\label{lem:bd-ek}
Under Assumptions~\ref{assump-smooth} and \ref{assump-sgd}, it holds 
$$\EE\big[\|\ve^{k+1}\|^2\big] \le \frac{2\beta_k^2\sigma^2}{m} + \frac{4(1-\beta_k)^2\eta_{k}^2 L^2}{m} \EE\big[\|\bar\vg^{k}\|^2\big] + (1-\beta_k)^2\left(\textstyle 1+\frac{4\eta_{k}^2 L^2}{m}\right)\EE\big[\|\ve^{k}\|^2\big], \forall\, k\ge0,$$
where $\bar\vg^k$ and $\ve^k$ are defined in \eqref{eq:vg-vgbar} and \eqref{eq:error-grad}.
\end{lemma}

\begin{proof}
First, notice $\EE_{B_{k+1}}[\langle \vv^{k+1}, \ve^{k}\rangle] = \langle\nabla F(\vx^{k+1}), \ve^{k}\rangle$ and $\EE_{B_{k+1}}[\langle \vu^{k+1}, \ve^{k}\rangle] = \langle\nabla F(\vx^{k}), \ve^{k}\rangle$, and thus
\begin{equation}\label{eq:crs-term0}
\EE_{B_{k+1}}[\langle \vv^{k+1} - \nabla F(\vx^{k+1}), \ve^{k}\rangle] = 0,\quad \EE_{B_{k+1}}[\langle \vu^{k+1} - \nabla F(\vx^{k}), \ve^{k}\rangle] = 0.
\end{equation}
Hence, by writing $\ve^{k+1} = \vv^{k+1} - \nabla F(\vx^{k+1}) + (1-\beta_k)(\nabla F(\vx^{k})- \vu^{k+1})+(1-\beta_k)\ve^{k}$, we have
\begin{equation}\label{eq:one-step-e}
\EE_{B_{k+1}}\big[\|\ve^{k+1}\|^2\big] = \EE_{B_{k+1}}\big[\|\vv^{k+1} - \nabla F(\vx^{k+1}) + (1-\beta_k)(\nabla F(\vx^{k})- \vu^{k+1})\|^2\big] + (1-\beta_k)^2\|\ve^{k}\|^2.
\end{equation}
By the Young's inequality, it holds
\begin{align}\label{eq:step2}
&~\|\vv^{k+1} - \nabla F(\vx^{k+1}) + (1-\beta_k)(\nabla F(\vx^{k})- \vu^{k+1})\|^2 \cr
= &~ \big\|\beta_k\big(\vv^{k+1} - \nabla F(\vx^{k+1})\big) + (1-\beta_k)\big(\vv^{k+1} - \nabla F(\vx^{k+1}) + \nabla F(\vx^{k})- \vu^{k+1}\big)\big\|^2\cr
\le &~ 2\beta_k^2\|\vv^{k+1} - \nabla F(\vx^{k+1})\|^2 + 2(1-\beta_k)^2\|\vv^{k+1} - \nabla F(\vx^{k+1}) + \nabla F(\vx^{k})- \vu^{k+1}\|^2.
\end{align}
From the definition of $\vv^{k+1}$ and $ \vu^{k+1}$ in \eqref{eq:def-v-u-k}, we have 
\begin{align}\label{eq:step23}
 	&~\EE_{B_{k+1}}\big[\|\vv^{k+1} - \nabla F(\vx^{k+1}) + \nabla F(\vx^{k})- \vu^{k+1}\|^2\big] \cr 
 	=&~\frac{1}{m^2}\EE_{B_{k+1}}\left\| \sum_{\xi\in B_{k+1}}\left(\nabla f(\vx^{k+1};\xi)-\nabla f(\vx^{k};\xi)  - \nabla F(\vx^{k+1})+\nabla F(\vx^{k})\right) \right\|^2 	 \cr
 	=&~\frac{1}{m^2}\sum_{j=1}^m\EE_{\xi_j^{k+1}}\left\| \nabla f(\vx^{k+1};\xi_j^{k+1})-\nabla f(\vx^{k};\xi_j^{k+1})  - \nabla F(\vx^{k+1})+\nabla F(\vx^{k})  \right\|^2 	 \cr
 	\le&~\frac{1}{m^2}\sum_{j=1}^m\EE_{\xi_j^{k+1}}\left\| \nabla f(\vx^{k+1};\xi_j^{k+1})-\nabla f(\vx^{k};\xi_j^{k+1}) \right\|^2 	 \cr	
 	\le&~\frac{L^2}{m}\|\vx^{k+1}-\vx^{k}\|^2,
\end{align}
where the second equality holds because of the i.i.d. samples in $B_{k+1}$ and the zero mean of the random vector $\nabla f(\vx^{k+1};\xi_j^{k+1})-\nabla f(\vx^{k};\xi_j^{k+1})  - \nabla F(\vx^{k+1})+\nabla F(\vx^{k}) $ resulted from unbiasedness in Assumption~\ref{assump-sgd}, 
the first inequality is due to the fact that the variance of a random vector is upper-bounded by its second moment, and the last inequality follows from Assumption~\ref{assump-smooth}.

Now, take conditional expectation on both sides of \eqref{eq:step2}, use \eqref{eq:step23}, and substitute it into \eqref{eq:one-step-e}. We have
$$\EE_{B_{k+1}}\big[\|\ve^{k+1}\|^2\big] \le (1-\beta_k)^2\|\ve^{k}\|^2 +2\beta_k^2\EE_{B_{k+1}} \big[\|\vv^{k+1} - \nabla F(\vx^{k+1})\|^2\big] + \frac{2(1-\beta_k)^2 L^2}{m} \EE_{B_{k+1}}\big[\|\vx^{k+1}-\vx^{k}\|^2\big] .$$
Taking a full expectation over the above inequality and using  Assumption~\ref{assump-sgd}, we have
	\begin{align}\label{eq:step32}
		\EE\big[\|\ve^{k+1}\|^2\big] \le &~ (1-\beta_k)^2\EE\big[\|\ve^{k}\|^2\big]+ \frac{2\beta_k^2\sigma^2}{m} + \frac{2(1-\beta_k)^2 L^2}{m} \EE\big[\|\vx^{k+1}-\vx^{k}\|^2\big]\cr
		= & ~ (1-\beta_k)^2\EE\big[\|\ve^{k}\|^2\big]+ \frac{2\beta_k^2\sigma^2}{m} + \frac{2(1-\beta_k)^2 \eta_k^2 L^2}{m} \EE\big[\|\vg^{k}\|^2\big],
	\end{align}
where we have used	$\vx^{k+1}-\vx^{k} = -\eta_{k} \vg^{k}$ in the equality.

By similar arguments as those in \eqref{eq:tri-g}, it holds
$$\|\vg^{k}\|^2 \le 2\|\bar\vg^{k}\|^2 + 2\|\vg^{k}-\bar\vg^{k}\|^2\le 2\|\bar\vg^{k}\|^2  + 2\|\ve^{k}\|^2.$$
Plugging the above inequality into \eqref{eq:step32}, we obtain the desired result.
\end{proof}

\subsection{Results with Varying Stepsize}
In this subsection, we show the convergence results of Algorithm~\ref{alg:prox-storm} by taking varying stepsizes. Using Lemmas~\ref{lem:obj-diff} and \ref{lem:bd-ek}, we first show a convergence rate result by choosing the parameters that satisfy a general condition. Then we specify the choice of the parameters.
\begin{theorem}\label{thm:generic-vary}
Under Assumptions~\ref{assump-obj} through \ref{assump-sgd}, let $\{\vx^k\}$ be the iterate sequence from Algorithm~\ref{alg:prox-storm}, with the parameters $\{\eta_k\}$ and $\{\beta_k\}$ satisfying the condition:
\begin{equation}\label{eq:cond-eta-beta}
\frac{1}{4}(1-\eta_k L)-\frac{\eta_k }{5m\eta_{k+1}}(1-\beta_k)^2 > 0,~\text{and}~ \frac{\eta_k}{2}(2-\eta_k L)-\frac{1}{20\eta_k L^2}+\frac{(1-\beta_k)^2(1+\frac{4\eta_k^2 L^2}{m})}{20\eta_{k+1} L^2} \le 0, \,\forall\, k\ge0.
\end{equation}
Let $\{\bar\vg^k\}$ be defined in \eqref{eq:vg-vgbar}. Then
\begin{equation}\label{eq:rate-grad}
\sum_{k=0}^{K-1}\left(\frac{\eta_k}{4}(1-\eta_k L)-\frac{\eta_k^2}{5m\eta_{k+1}}(1-\beta_k)^2\right)\EE[\|\bar\vg^k\|^2] \le \Phi(\vx^{0}) - \Phi^* + \frac{\sigma^2}{20 m_0\eta_0 L^2} + \sum_{k=0}^{K-1}\frac{\beta_k^2\sigma^2}{10m\eta_{k+1} L^2}.
\end{equation}
\end{theorem}

\begin{proof}
From Lemmas~\ref{lem:obj-diff} and \ref{lem:bd-ek}, it follows that
\begin{align}\label{eq:diff-merit}
&\EE\left[\Phi(\vx^{k+1})+\frac{\|\ve^{k+1}\|^2}{20\eta_{k+1} L^2}-\Phi(\vx^{k})-\frac{\|\ve^{k}\|^2}{20\eta_k L^2}\right]
\le \EE\left[\frac{\eta_k}{2}(2-\eta_k L)\|\ve^k\|^2 - \frac{\eta_k}{4}(1-\eta_k L)\|\bar\vg^k\|^2-\frac{\|\ve^{k}\|^2}{20\eta_k L^2}\right]\nonumber\\
&\hspace{2cm}+ \frac{1}{20\eta_{k+1} L^2}\EE\left[\frac{2\beta_k^2\sigma^2}{m} + \frac{4(1-\beta_k)^2\eta_k^2 L^2}{m} \|\bar\vg^{k}\|^2 + (1-\beta_k)^2\left(1+\frac{4\eta_k^2 L^2}{m}\right)\|\ve^{k}\|^2\right].
\end{align}
We have from the condition of $\{\beta_k\}$ that
the coefficient of the term $\|\ve^k\|^2$ on the right hand side of \eqref{eq:diff-merit} is nonpositive, and thus we obtain from \eqref{eq:diff-merit} that
$$\EE\left[\Phi(\vx^{k+1})+\frac{\|\ve^{k+1}\|^2}{20\eta_{k+1} L^2}-\Phi(\vx^{k})-\frac{\|\ve^{k}\|^2}{20\eta_k L^2}\right] \le \frac{\beta_k^2\sigma^2}{10m\eta_{k+1} L^2} -\left(\frac{\eta_k}{4}(1-\eta_k L)-\frac{\eta_k^2 }{5m\eta_{k+1}}(1-\beta_k)^2\right)\EE[\|\bar\vg^k\|^2].$$
Summing up the above inequality from $k=0$ through $K-1$ gives
\begin{align*}
&~\EE\left[\Phi(\vx^{K})+\frac{\|\ve^{K}\|^2}{20\eta_{K} L^2}-\Phi(\vx^{0})-\frac{\|\ve^{0}\|^2}{20\eta_0 L^2}\right] \\
\le &~ \sum_{k=0}^{K-1}\frac{\beta_k^2\sigma^2}{10m\eta_{k+1} L^2}-\sum_{k=0}^{K-1}\left(\frac{\eta_k}{4}(1-\eta_k L)-\frac{\eta_k^2}{5m\eta_{k+1}}(1-\beta_k)^2\right)\EE[\|\bar\vg^k\|^2],
\end{align*}
which implies the inequality in \eqref{eq:rate-grad} by $\EE[\|\ve^0\|^2]\le \frac{\sigma^2}{m_0}$.
\end{proof}

Below we specify the choice of parameters and establish complexity results of Algorithm~\ref{alg:prox-storm}. 

\begin{theorem}[convergence rate with varying stepsizes]\label{thm:rate-dynamic}
Under Assumptions~\ref{assump-obj} through \ref{assump-sgd}, let $\{\vx^k\}$ be the iterate sequence from Algorithm~\ref{alg:prox-storm}, with $m_0=m$ and the parameters $\{\eta_k\}$ and $\{\beta_k\}$ set to
\begin{equation}\label{eq:dynamic-para}
\eta_k = \frac{\eta}{L(k+4)^{\frac{1}{3}}},\quad \beta_k = \frac{1+24\eta_k^2 L^2-\frac{\eta_{k+1}}{\eta_k}}{1+4\eta_k^2 L^2}, \,\forall\, k\ge 0,
\end{equation}
where $\eta\le \frac{\sqrt[3]{4}}{8 }$ is a positive number. If $\tau$ is selected according to \eqref{eq:select-tau}, then
\begin{equation}\label{eq:rate-grad-dynamic}
\EE[\|\bar\vg^\tau\|^2] \le \frac{2 \left( L(\Phi(\vx^0)-\Phi^*)+\frac{\sqrt[3]{4}\sigma^2}{20 m \eta}+\frac{\sigma^2}{10m}\big(1152\eta^3 (\frac{5}{4})^{\frac{1}{3}}(\log(K+3)-\log 3) + \frac{1}{3\sqrt[3]{9}\eta}\big)\right)}{3\big(\frac{7}{32}-\frac{1}{5}(\frac{5}{4})^{\frac{1}{3}}\big)\eta\big((K+4)^{\frac{2}{3}}-4^{\frac{2}{3}}\big)}.
\end{equation}
\end{theorem}

\begin{proof}
Since $\eta\le \frac{\sqrt[3]{4}}{8}$, it holds $\eta_k\le \frac{1}{8L}$. Also, notice $\frac{\eta_k}{\eta_{k+1}}\le (\frac{5}{4})^{\frac{1}{3}}$ or equivalently $\frac{\eta_{k+1}}{\eta_{k}}\ge (\frac{4}{5})^{\frac{1}{3}}$ for all $k\ge0$. Hence, it is straightforward to have $\beta_k\in(0,1)$ and thus $(1-\beta_k)^2 \le 1-\beta_k$ for each $k\ge0$. Now notice $\frac{5m\eta_{k+1}}{4\eta_k}(1-\eta_k L)\ge \frac{5}{4}(\frac{4}{5})^{\frac{1}{3}}\frac{7}{8}>1 \ge (1-\beta_k)^2$, so the first inequality in \eqref{eq:cond-eta-beta} holds. In addition, to ensure the second inequality in \eqref{eq:cond-eta-beta}, it suffices to have $(1-\beta_k)(1+\frac{4\eta_k^2 L^2}{m})\le \frac{\eta_{k+1}}{\eta_{k}} - 10\eta_k\eta_{k+1}L^2(2-\eta_k L)$. Because $20\eta_k^2 L^2 \ge 10\eta_k\eta_{k+1}L^2(2-\eta_k L)$, this inequality is implied by $(1-\beta_k)(1+\frac{4\eta_k^2 L^2}{m})\le \frac{\eta_{k+1}}{\eta_{k}} -20\eta_k^2 L^2$, which is further implied by the choice of $\beta_k$ in \eqref{eq:dynamic-para}. Therefore, both conditions in \eqref{eq:cond-eta-beta} hold, and thus we have \eqref{eq:rate-grad}.

Next we bound the coefficients in \eqref{eq:rate-grad}. First, from $1-\eta_k L \ge\frac{7}{8}$ and $\frac{\eta_k}{\eta_{k+1}}\le (\frac{5}{4})^{\frac{1}{3}}$ for all $k$, we have
\begin{equation}\label{eq:bd-sum-eta-2}
\sum_{k=0}^{K-1}\left(\frac{\eta_k}{4}(1-\eta_k L)-\frac{\eta_k^2}{5m\eta_{k+1}}(1-\beta_k)^2\right)\ge c\sum_{k=0}^{K-1}\eta_k \ge \frac{c\eta}{L}\int_{0}^{K}(x+4)^{-\frac{1}{3}}dx = \frac{3c\eta}{2 L}\left((K+4)^{\frac{2}{3}}-4^{\frac{2}{3}}\right),
\end{equation}
where $c=\frac{7}{32}-\frac{1}{5}(\frac{5}{4})^{\frac{1}{3}} > 0$. Second,
\begin{align}\label{eq:sum-beta-term1}
\sum_{k=0}^{K-1}\frac{\beta_k^2}{\eta_{k+1}} \le \frac{L}{\eta}\sum_{k=0}^{K-1}  (k+5)^{\frac{1}{3}}\left(1+24\eta_k^2 L^2-\frac{\eta_{k+1}}{\eta_k}\right)^2 = \frac{L}{\eta}\sum_{k=0}^{K-1}  (k+5)^{\frac{1}{3}}\left(1+24\eta_k^2 L^2-\frac{(k+4)^{\frac{1}{3}}}{(k+5)^{\frac{1}{3}}} \right)^2.
\end{align}
Note that
\begin{align}\label{eq:k-eta-t1}
\sum_{k=0}^{K-1}  (k+5)^{\frac{1}{3}}\eta_k^4= \frac{\eta^4}{L^4}\sum_{k=0}^{K-1}  (k+5)^{\frac{1}{3}}(k+4)^{-\frac{4}{3}}\le \frac{\eta^4}{L^4}(\tfrac{5}{4})^{\frac{1}{3}}\sum_{k=0}^{K-1}(k+4)^{-1}\le \frac{\eta^4}{L^4}(\tfrac{5}{4})^{\frac{1}{3}}(\log(K+3)-\log 3).
\end{align}
Furthermore, by $a^3-b^3 = (a-b)(a^2+ab+b^2)$ for any $a,b\in\RR$, we have 
$$1-\frac{(k+4)^{\frac{1}{3}}}{(k+5)^{\frac{1}{3}}}=(k+5)^{-\frac{1}{3}}\left((k+5)^{\frac{1}{3}}-(k+4)^{\frac{1}{3}}\right)=\frac{(k+5)^{-\frac{1}{3}}}{(k+5)^{\frac{2}{3}}+(k+5)^{\frac{1}{3}}(k+4)^{\frac{1}{3}}+(k+4)^{\frac{2}{3}}},$$
and thus
\begin{align}\label{eq:k-eta-t2}
\sum_{k=0}^{K-1}  (k+5)^{\frac{1}{3}}\left({\textstyle 1-\frac{(k+4)^{\frac{1}{3}}}{(k+5)^{\frac{1}{3}}} }\right)^2= &~ \sum_{k=0}^{K-1}  \frac{(k+5)^{-\frac{1}{3}}}{\left((k+5)^{\frac{2}{3}}+(k+5)^{\frac{1}{3}}(k+4)^{\frac{1}{3}}+(k+4)^{\frac{2}{3}}\right)^2}\cr
\le &~ \frac{1}{9}\sum_{k=0}^{K-1} (k+4)^{-\frac{5}{3}} \le \frac{1}{6\sqrt[3]{9}}.
\end{align}
Now applying the inequality $(a+b)^2 \le 2a^2+2b^2$ to \eqref{eq:sum-beta-term1} and then using \eqref{eq:k-eta-t1} and \eqref{eq:k-eta-t2}, we obtain
\begin{align}\label{eq:sum-beta-term2}
\sum_{k=0}^{K-1}\frac{\beta_k^2}{\eta_{k+1}} \le 1152\eta^3 L(\tfrac{5}{4})^{\frac{1}{3}}(\log(K+3)-\log 3) + \frac{L}{3\sqrt[3]{9}\eta}.
\end{align}

Therefore, plugging \eqref{eq:bd-sum-eta-2} and \eqref{eq:sum-beta-term2} into  \eqref{eq:rate-grad} and by the selection of $\tau$ in \eqref{eq:select-tau}, we obtain the desired result. 
\end{proof}

\begin{remark}\label{rm:rate-dynamic}
The result in Theorem~\ref{thm:rate-dynamic} does not include the noiseless case, i.e., $\sigma=0$. Nevertheless, if in that case, we can simply choose $\eta_k=\Theta(\frac{1}{L})$ and $\beta_k=1$ for all $k\ge0$. This way, Algorithm~\ref{alg:prox-storm} reduces to the deterministic mirror-prox method, and we can easily obtain $\min_{0\le k < K}\|\bar\vg^k\|^2 = O(\frac{1}{K})$ from \eqref{eq:rate-grad}. 
\end{remark}

By Theorem~\ref{thm:rate-dynamic}, we below estimate the complexity result of Algorithm~\ref{alg:prox-storm} to produce a stochastic $\vareps$-stationary solution.

\begin{corollary}[complexity result with varying stepsizes]\label{cor:complexity-dynamic}
 Let $\vareps>0$ be given and suppose $\sigma>0$. Then under the same conditions of Theorem~\ref{thm:rate-dynamic},  Algorithm~\ref{alg:prox-storm} can produce a stochastic $\vareps$-stationary solution of \eqref{eq:stoc-prob} with a total complexity 
$$T_{\mathrm{total}} = mK=O\left(\max\left\{m\vareps^{-3}\big(L(\Phi(\vx^0)-\Phi^*)\big)^{\frac{3}{2}},\ \vareps^{-3}(|\log\vareps|+|\log\sigma|)^{\frac{3}{2}}\frac{\sigma^3}{\sqrt{m}}\right\}\right).$$ 
\end{corollary}

\begin{proof}
By Theorem~\ref{thm:rate-dynamic} with $\eta=\frac{\sqrt[3]{4}}{8 }$, we have 
\begin{equation}\label{eq:bd-vg-tau-dynamic}
\EE[\|\bar\vg^\tau\|^2] = O\left(K^{-\frac{2}{3}}\big(\textstyle L(\Phi(\vx^0)-\Phi^*) + \frac{\sigma^2\log K}{m}\big)\right).
\end{equation}
Hence,  it suffices to let 
$K=\Theta\left(\max\left\{\vareps^{-3}\big(L(\Phi(\vx^0)-\Phi^*)\big)^{\frac{3}{2}},\ \vareps^{-3}(|\log\vareps|+|\log\sigma|)^{\frac{3}{2}}\frac{\sigma^3}{m^{\frac{3}{2}}}\right\}\right)$, to have $\EE[\|\bar\vg^\tau\|^2]\le \vareps^2$.  This completes the proof.
\end{proof}

\begin{remark}\label{rm:dep-sigma-dynamic}
If $m=1$ or $m=O(1)$ independent of $\sigma$, then the total complexity will be $$T_{\mathrm{total}}=O\left(\max\left\{\vareps^{-3}\big(L(\Phi(\vx^0)-\Phi^*)\big)^{\frac{3}{2}}, \, \vareps^{-3}\sigma^3(|\log\vareps|+|\log\sigma|)^{\frac{3}{2}}\right\}\right).$$ 
If $\sigma \ge1$ is big and can be estimated, we can take $m=\Theta(\sigma^2)$. This way, we obtain the total complexity $O\left(\vareps^{-3} \sigma^2\big((|\log\vareps|+\log\sigma)^{\frac{3}{2}} +(L(\Phi(\vx^0)-\Phi^*))^{\frac{3}{2}}\big)\right)$. This result is near-optimal in the sense that its dependence on $\vareps$ has the additional logarithmic term $|\log\vareps|^{\frac{3}{2}}$ compared to the lower bound result in \cite{arjevani2019lower}. In the remaining part of this section, we show that with constant stepsizes, 
Algorithm~\ref{alg:prox-storm} can achieve the optimal complexity result $O(\vareps^{-3})$.
\end{remark}

\subsection{Results with Constant Stepsize}
In this subsection, we show convergence results of Algorithm~\ref{alg:prox-storm} by taking constant stepsizes, i.e., $\eta_k = \eta_0,\forall \, k\ge1$. In order to consider the dependence on the quantities $L$, $\Phi(\vx^0)-\Phi^*$ and $\sigma^2$, we give two settings that yield two different results, but each result has the same dependence on the target accuracy $\vareps$. The first result is obtained from Theorem~\ref{thm:generic-vary} by taking constant stepsizes.

\begin{theorem}[convergence rate I with constant stepsizes]\label{thm:rate-const}
Under Assumptions~\ref{assump-obj} through \ref{assump-sgd}, let $\{\vx^k\}$ be the iterate sequence from Algorithm~\ref{alg:prox-storm}, with the parameters $\{\eta_k\}$ and $\{\beta_k\}$ set to
\begin{equation}\label{eq:const-para}
\eta_k = \frac{\eta}{L\sqrt[3]{K}},\quad \beta_k = \beta=\frac{4\eta^2/m+10\eta^2(2-\eta /K^{\frac{1}{3}})}{K^{\frac{2}{3}}+4\eta^2/m}, \,\forall\, k\ge 0,
\end{equation}
where $\eta < \frac{\sqrt[3]{K}}{5 }$ is a positive number. If $\tau$ is selected from $\{0,1,\ldots,K-1\}$ uniformly at random, then
\begin{equation}\label{eq:rate-grad-const}
\EE[\|\bar\vg^\tau\|^2] \le \frac{1}{ K^{\frac{2}{3}}\left(\textstyle\frac{1}{4}\big(1- \frac{\eta}{\sqrt[3]{K}}\big) -  \frac{1}{5}\right)} \left(\frac{L\big(\Phi(\vx^{0}) - \Phi^*\big)}{\eta} + \frac{\sigma^2 \sqrt[3]{K}}{20 m_0\eta^2}+\frac{24^2\sigma^2\eta^2}{10m}\right).
\end{equation}
\end{theorem}

\begin{proof}
First note $\frac{\eta}{\sqrt[3]{K}} < \frac{1}{5}$ and thus $\beta \in (0,1)$. Now it is easy to verify by using $(1-\beta)^2 < 1-\beta$ that the conditions in \eqref{eq:cond-eta-beta} are satisfied. Hence, the result in \eqref{eq:rate-grad} holds.

Second, by the choice of $\eta_k$ and $\beta_k$, we have
\begin{equation}\label{eq:sum-coeff1}
\begin{aligned}
&~\sum_{k=0}^{K-1}\left(\textstyle\frac{\eta_k}{4}(1-\eta_k L)-\frac{\eta_k^2}{5m\eta_{k+1}}(1-\beta_k)^2\right)\cr
\ge&~ \sum_{k=0}^{K-1} \left(\textstyle\frac{\eta}{4L\sqrt[3]{K}}\big(1- \frac{\eta }{\sqrt[3]{K}}\big) -  \frac{\eta}{5L\sqrt[3]{K}}\right) =  \textstyle \frac{\eta}{L} K^{\frac{2}{3}} \left(\textstyle\frac{1}{4}\big(1- \frac{\eta }{\sqrt[3]{K}}\big) -  \frac{1}{5}\right),
\end{aligned}
\end{equation}
and 
\begin{equation}\label{eq:sum-coeff2}
\begin{aligned}
\sum_{k=0}^{K-1}\frac{\beta_k^2\sigma^2}{10m\eta_{k+1} L^2} \le \sum_{k=0}^{K-1} \frac{\sigma^2\sqrt[3]{K}}{10m\eta L}\left(\frac{4\eta^2+20\eta^2}{K^{\frac{2}{3}}}\right)^2 = \frac{24^2\sigma^2\eta^3}{10m L}.
\end{aligned}
\end{equation}
Plugging \eqref{eq:sum-coeff1} and \eqref{eq:sum-coeff2} into \eqref{eq:rate-grad}, we obtain the desired result by the selection of $\tau$ in \eqref{eq:select-tau}.
\end{proof}
From \eqref{eq:rate-grad-const}, we see that in order to have the $O(K^{-\frac{2}{3}})$ convergence rate, we need to set $m_0=\Theta(\sqrt[3]{K})$. Next we set $m_0$ in this way and estimate the complexity result of Algorithm~\ref{alg:prox-storm} with the constant stepsize.

\begin{corollary}[complexity result I with constant stepsizes]\label{cor:complexity-const}
 Let $\vareps>0$ be given. Under Assumptions~\ref{assump-obj} through \ref{assump-sgd}, let $\{\vx^k\}$ be the iterate sequence from Algorithm~\ref{alg:prox-storm} with $m_0\ge c_0 \sqrt[3]{K}$ and the parameters $\{\eta_k\}$ and $\{\beta_k\}$ set to  those in \eqref{eq:const-para} where $\eta \le \frac{\sqrt[3]{K}}{10 }$. Let $\tau$ be selected from $\{0,1,\ldots,K-1\}$ uniformly at random. Then $\vx^\tau$ is a stochastic $\vareps$-stationary solution of \eqref{eq:stoc-prob} if  
\begin{equation}\label{eq:K-const-1}
K = \left\lceil\frac{40^{\frac{3}{2}}\left(\frac{L\big(\Phi(\vx^{0}) - \Phi^*\big)}{\eta} + \frac{\sigma^2}{20 c_0\eta^2}+\frac{24^2\sigma^2\eta^2}{10m}\right)^{\frac{3}{2}}}{\vareps^3}\right\rceil.
\end{equation} 
\end{corollary}

\begin{proof}
When $\eta \le \frac{\sqrt[3]{K}}{10}$, it holds $\frac{1}{4}\big(1- \frac{\eta}{\sqrt[3]{K}}\big) -  \frac{1}{5}\ge \frac{1}{40}$. Hence, \eqref{eq:rate-grad-const} with $m_0\ge c_0 \sqrt[3]{K}$ implies
$$\EE[\|\bar\vg^\tau\|^2] \le \frac{40}{K^{\frac{2}{3}}} \left( \frac{L\big(\Phi(\vx^{0}) - \Phi^*\big)}{\eta} + \frac{\sigma^2}{20 c_0\eta^2}+\frac{24^2\sigma^2\eta^2}{10m}\right),$$
which together with the condition of $K$ in \eqref{eq:K-const-1} gives $\EE[\|\bar\vg^\tau\|^2] \le\vareps^2$. This completes the proof.
\end{proof}

\begin{remark}\label{rm:dep-sigma-const}
Suppose that $\sigma\ge1$ and can be estimated. Also, assume $L=\Omega(1)$ and $\Phi(\vx^{0}) - \Phi^*=\Omega(1)$. In this case, we let $\eta=\Theta\left(\sigma^{-\frac{2}{3}}\big( L(\Phi(\vx^{0}) - \Phi^*)\big)^{\frac{1}{3}}\right)$, $c_0=\Theta(\sigma^{\frac{8}{3}})$, and $m=O(1)$ independent of $\sigma$. Then from \eqref{eq:K-const-1}, we have
$K=O\left(\vareps^{-3}\sigma L(\Phi(\vx^{0}) - \Phi^*)\right)$. With this choice, the total number of sample gradients will be 
\begin{equation}\label{eq:dep-sigma-const}
T_{\mathrm{total}}=m_0 + m(K-1) = O\left(\vareps^{-1}\sigma^3\big(L(\Phi(\vx^{0}) - \Phi^*)\big)^{\frac{1}{3}}+\vareps^{-3}\sigma L(\Phi(\vx^{0}) - \Phi^*)\right).
\end{equation}
The dependence on the pair $(\vareps, \sigma)$ matches with the result in \cite{tran2021hybrid}.
\end{remark}

The complexity result given in \eqref{eq:dep-sigma-const} has a low dependence on $(\vareps, \sigma, L(\Phi(\vx^{0}) - \Phi^*) )$ in the sense that $\vareps^{-3}$ only multiplies with $\sigma L(\Phi(\vx^{0}) - \Phi^*)$ but not a higher order. However, the drawback is that the initial batch $m_0$ must be in the order of $\vareps^{-1}$ to obtain the complexity result $O(\vareps^{-3})$. Our second result with constant stepsizes will relax the requirement. 
We utilize the momentum accumulation in the parameter of \eqref{eq:step32} and give our novel convergence analysis, by introducing the following quantity
\begin{equation}\label{eq:gamma}
	\Gamma_k = \left\{
	\begin{array}{ll}
		\prod_{i=0}^{k-1}(1-\beta_i)^2, & \text{ if } k\ge 1, \\[0.2cm]
		1, & \text{ if } k=0.
	\end{array}
	\right.
\end{equation}

We first give a generic result below under certain conditions on the parameters. Then, we will specify the choice of parameters to satisfy the conditions. 
\begin{theorem}\label{thm:main}
	Under Assumptions~\ref{assump-obj} through \ref{assump-sgd}, let $\{\vx^k\}$ be the iterate sequence from Algorithm~\ref{alg:prox-storm}. Suppose there are constants $A$ and $B$ such that the parameters $\{\eta_k\}$ and $\{\beta_k\}$ satisfying the conditions:
	\begin{equation}\label{eq:cond-eta-beta1}
		2\eta_k L+ \frac{4 L^2}{m} \frac{\eta_{k}}{\Gamma_{k} } \sum_{j=k+1}^{K-1}\eta_j\Gamma_{j} \le 1, \,\forall\, k=0,\ldots,K-1,
	\end{equation}
	\begin{equation}\label{eq:cond-eta-beta2}
\sum_{k=1}^{K-1}\eta_k\Gamma_{k}\le A,~\text{and}~\sum_{k=1}^{K-1}\eta_k\Gamma_{k}\sum_{j=0}^{k-1}\frac{\beta_j^2}{\Gamma_{j+1} }\le B,
\end{equation}
where $K$ is the maximum number of iterations in Algorithm~\ref{alg:prox-storm}. 
	Let $\{\bar\vg^k\}$ be defined in \eqref{eq:vg-vgbar}. Then	
	\begin{equation}\label{eq:mainthm}
		\sum_{k=0}^{K-1} \eta_k \EE\big[\|\bar\vg^{k}\|^2\big] 
		\le	12\big[\Phi(\vx^0)-\Phi^*\big] +4(2A+3)\frac{\sigma^2}{ m_0}+ 16B\frac{\sigma^2}{m} .
	\end{equation}	
\end{theorem}
\begin{proof}
We begin by taking the total expectation and 
telescoping the inequality in \eqref{eq:chg-obj-0} over $k=0,\ldots,K-1$ to obtain
\begin{align*}
	\EE\big[\Phi(\vx^{K})\big]-\Phi(\vx^0) 
	\le&~ \sum_{k=0}^{K-1}\frac{\eta_k}{2} \EE\big[\|\ve^k\|^2\big] -\sum_{k=0}^{K-1} \frac{\eta_k}{2}(1-\eta_k L)\EE\big[\|\vg^k\|^2\big]\\
	\le&~ \frac{\sigma^2}{m_0}+ \sum_{k=1}^{K-1}\frac{\eta_k}{2} \EE\big[\|\ve^k\|^2\big] -\sum_{k=0}^{K-1} \frac{\eta_k}{2}(1-\eta_k L)\EE\big[\|\vg^k\|^2\big],
\end{align*}
where we have used $\EE\big[\|\ve^{0}\|^2\big]\le \frac{\sigma^2}{m_0}$ by Assumption~\ref{assump-sgd}. Since $\Phi(\vx^{K}) \ge \Phi^*$ from Assumption~\ref{assump-obj}, the above inequality implies 
\begin{equation}\label{eq:main1}
	\sum_{k=0}^{K-1} \frac{\eta_k}{2}(1-\eta_k L)\EE\big[\|\vg^k\|^2\big] 
	\le \Phi(\vx^0)-\Phi^* +\frac{\sigma^2}{m_0}+ \sum_{k=1}^{K-1}\frac{\eta_k}{2} \EE\big[\|\ve^k\|^2\big].
\end{equation}	

In addition, we divide both sides of \eqref{eq:step32} by $\Gamma_{k+1}$ and obtain from the definition of $\Gamma_{k+1}$ in \eqref{eq:gamma} that
	$$\frac{1}{\Gamma_{k+1} }\EE\big[\|\ve^{k+1}\|^2\big] \le \frac{1}{\Gamma_{k} }\EE\big[\|\ve^{k}\|^2\big] +\frac{1}{\Gamma_{k+1} }\frac{2\beta_k^2\sigma^2}{m} + \frac{1}{\Gamma_{k} }\frac{2\eta_{k}^2 L^2}{m} \EE\big[\|\vg^{k}\|^2\big],\, \forall\, k\ge 0.$$
Let $j=0,\ldots,k-1$ be another index on which the above inequality is telescoped. We obtain
	$$\frac{1}{\Gamma_{k} }\EE\big[\|\ve^{k}\|^2\big] \le \EE\big[\|\ve^{0}\|^2\big] +\sum_{j=0}^{k-1}\frac{1}{\Gamma_{j+1} }\frac{2\beta_j^2\sigma^2}{m} + \sum_{j=0}^{k-1}\frac{1}{\Gamma_{j} }\frac{2\eta_{j}^2 L^2}{m} \EE\big[\|\vg^{j}\|^2\big],\, \forall\, k\ge 1.$$
Multiplying $\Gamma_{k}$ to both sides of the above inequality and rearranging it gives
$$ 	\EE\big[\|\ve^{k}\|^2\big] \le \Gamma_{k}\bigg(\frac{\sigma^2}{m_0} +\frac{2\sigma^2}{m}\sum_{j=0}^{k-1}\frac{\beta_j^2}{\Gamma_{j+1} } \bigg)+ \frac{2 L^2}{m}\sum_{j=0}^{k-1}\frac{\Gamma_{k}}{\Gamma_{j} } \eta_{j}^2  \EE\big[\|\vg^{j}\|^2\big],\, \forall\, k\ge 1,$$
where we have used $\EE\big[\|\ve^{0}\|^2\big]\le \frac{\sigma^2}{m_0}$ again. 
 Now multiply $\eta_k$ to the above inequality and sum it up over $k=1,\ldots,K-1$ to have
\begin{align}\label{eq:main2}
 	\sum_{k=1}^{K-1}\eta_k\EE\big[\|\ve^{k}\|^2\big] 
 	\le&~ \sigma^2\sum_{k=1}^{K-1}\eta_k\Gamma_{k}\bigg(\frac{1}{m_0} +\frac{2}{m} \sum_{j=0}^{k-1}\frac{\beta_j^2}{\Gamma_{j+1} } \bigg)+\frac{2 L^2}{m}\sum_{k=1}^{K-1} \sum_{j=0}^{k-1}\frac{\eta_k\Gamma_{k}}{\Gamma_{j} } \eta_{j}^2  \EE\big[\|\vg^{j}\|^2\big] \cr
 	=&~ \sigma^2\sum_{k=1}^{K-1}\eta_k\Gamma_{k}\bigg(\frac{1}{m_0} +\frac{2}{m} \sum_{j=0}^{k-1}\frac{\beta_j^2}{\Gamma_{j+1} } \bigg)+\frac{2 L^2}{m}\sum_{j=0}^{K-2}\frac{\eta_{j}^2}{\Gamma_{j} }\bigg(\sum_{k=j+1}^{K-1}\eta_k\Gamma_{k}\bigg)    \EE\big[\|\vg^{j}\|^2\big] \cr
 	=&~ \sigma^2\sum_{k=1}^{K-1}\eta_k\Gamma_{k}\bigg(\frac{1}{m_0} +\frac{2}{m} \sum_{j=0}^{k-1}\frac{\beta_j^2}{\Gamma_{j+1} } \bigg)+\frac{2 L^2}{m}\sum_{k=0}^{K-1}\frac{\eta_{k}^2}{\Gamma_{k} }\bigg(\sum_{j=k+1}^{K-1}\eta_j\Gamma_{j}\bigg)    \EE\big[\|\vg^{k}\|^2\big],
\end{align}
where the first equality follows by swapping summation, and the second equality is obtained by swapping indices and realizing that the coefficient for $\EE\big[\|\vg^{K-1}\|^2\big]$ is null.

Now we have by substituting \eqref{eq:main2} into \eqref{eq:main1} and rearranging terms that
$$
	\sum_{k=0}^{K-1} \frac{\eta_k}{2}\bigg(1-\eta_k L- \frac{2 L^2}{m} \frac{\eta_{k}}{\Gamma_{k} } \sum_{j=k+1}^{K-1}\eta_j\Gamma_{j}\bigg)\EE\big[\|\vg^k\|^2\big] 
	\le \Phi(\vx^0)-\Phi^* +\frac{\sigma^2}{m_0}+ \frac{\sigma^2}{2}\sum_{k=1}^{K-1}\eta_k\Gamma_{k}\bigg(\frac{1}{m_0} +\frac{2}{m} \sum_{j=0}^{k-1}\frac{\beta_j^2}{\Gamma_{j+1} } \bigg),
$$
which together with the conditions in \eqref{eq:cond-eta-beta1} and \eqref{eq:cond-eta-beta2} gives the bound for $\vg^k$:
\begin{equation}\label{eq:main3}
	\sum_{k=0}^{K-1}  \eta_k \EE\big[\|\vg^k\|^2\big] 
	\le 4\big[\Phi(\vx^0)-\Phi^*\big]  + 2(A+2)\frac{\sigma^2}{ m_0}+ 4B\frac{\sigma^2}{m} .
\end{equation}	
Use \eqref{eq:cond-eta-beta1} again and substitute \eqref{eq:main3} into \eqref{eq:main2}. We obtain the bound for $\ve^k$:
\begin{equation}\label{eq:main4}
	\sum_{k=0}^{K-1} \eta_k \EE\big[\|\ve^k\|^2\big] 
	\le A\frac{\sigma^2}{ m_0}+ 2B\frac{\sigma^2}{m} +\sum_{k=0}^{K-1}\frac{\eta_{k}}{2} \EE\big[\|\vg^{k}\|^2\big]
	\le 2\big[\Phi(\vx^0)-\Phi^*\big] +2(A+1)\frac{\sigma^2}{ m_0}+ 4B\frac{\sigma^2}{m} .
\end{equation}	

Finally, we have from \eqref{eq:tri-g} that  
$\|\bar\vg^{k}\|^2 \le 2\|\vg^{k}\|^2  + 2\|\ve^{k}\|^2.$ Sum up this inequality over $k=0,\ldots,K-1$ and
substitute \eqref{eq:main3} and  \eqref{eq:main4} into the summation. We obtain the result in \eqref{eq:mainthm}.
\end{proof}

Below we specify the choice of parameters and establish complexity results of Algorithm~\ref{alg:prox-storm}. The following lemma will be used to show the conditions in \eqref{eq:cond-eta-beta1} and \eqref{eq:cond-eta-beta2}. 

\begin{lemma}\label{lem:keylem}
	Let 
	\begin{equation}\label{eq:beta-def}
		\beta_{k}=3\big[(k+3)^{1/3}-(k+2)^{1/3}\big],\quad k\ge 0,
	\end{equation}
Then we  have 
\begin{equation}\label{eq:keylem1}
	\sum_{j=k+1}^{K-1}\frac{\Gamma_{j}}{\Gamma_{k}} \le \frac{1}{2}(k+2)^{2/3}+\frac{1}{6}(k+2)^{1/3}+\frac{1}{36}. 
\end{equation}
\end{lemma}
\begin{proof}	
	By the fact $a^3-b^3 = (a-b)(a^2+ab+b^2)$, we have 
	\begin{equation}\label{eq:beta-bound}
		\beta_{k}=3\big[(k+3)^{1/3}-(k+2)^{1/3}\big]=
		\frac{3}{(k+3)^{2/3}+(k+3)^{1/3}(k+2)^{1/3}+(k+2)^{2/3}}.
	\end{equation}
	Hence, $\beta_{k}\in \big[(k+3)^{-2/3},(k+2)^{-2/3}\big]$ for all $k\ge 0$, and it is a decreasing sequence. In addition,
by the definition of $\Gamma_k$ and $\beta_k$, it holds for all $j>k\ge 0$ that
\begin{equation}\label{eq:keystring}
	\frac{\Gamma_{j}}{\Gamma_{k}} 
	= \frac{\prod_{l=0}^{j-1}(1-\beta_l)^2}{\prod_{l=0}^{k-1}(1-\beta_l)^2}  = \prod_{l=k}^{j-1}(1-\beta_l)^2
	\le  e^{-2\sum_{l=k}^{j-1}\beta_l}
	=  e^{-6\big[(j+2)^{1/3}-(k+2)^{1/3}\big]},
\end{equation}	
where the inequality holds because $0\le 1+x\le e^x, \forall\, x\ge -1$. Therefore we have that for any $k\ge 0,$
\begin{equation}\label{eq:key-beta-ineq1}
\sum_{j=k+1}^{K-1}\frac{\Gamma_{j}}{\Gamma_{k}} 
\le 
 \sum_{j=k+1}^{K-1} e^{-6\big[(j+2)^{1/3}-(k+2)^{1/3}\big]}
= e^{6(k+2)^{1/3}}\sum_{j=k+1}^{K-1} e^{-6 (j+2)^{1/3} }.
\end{equation}
	Since $e^{-6x^{1/3}}$ is a decreasing function and has an anti-derivative $-\frac{1}{36} e^{-6x^{1/3}} (18x^{2/3}+6x^{1/3}+1)$, we have 
	\begin{equation}\label{eq:key-beta-ineq2}
\sum_{j=k+1}^{K-1} e^{-6 (j+2)^{1/3} }
\le \int_{k+2}^{K+1} e^{-6x^{1/3}}dx
\le \frac{1}{36}  e^{-6(k+2)^{1/3}} (18(k+2)^{2/3}+6(k+2)^{1/3}+1) .
	\end{equation}
	Substituting \eqref{eq:key-beta-ineq2} into \eqref{eq:key-beta-ineq1} gives \eqref{eq:keylem1} and completes the proof.
\end{proof}

Now we are ready to show the second convergence rate result with constant stepsizes.
\begin{theorem}[convergence rate II with constant stepsizes]\label{thm:rate-const2}
	Under Assumptions~\ref{assump-obj} through \ref{assump-sgd}, let $\{\vx^k\}$ be the iterate sequence from Algorithm~\ref{alg:prox-storm} with $ \eta_k =\frac{\eta}{L \sqrt[3]{K}}$ and $\{\beta_k\}$ set by \eqref{eq:beta-def}, 
	where $\eta \le \frac{1}{4}$ is a positive number. If $\tau$ is selected from $\{0,1,\ldots,K-1\}$ uniformly at random, then
	\begin{equation}\label{eq:rate-grad-const2}
		\EE[\|\bar\vg^\tau\|^2] \le \frac{1}{K^{\frac{2}{3}}} 
		\left(	\frac{12 L}{\eta}\big[\Phi(\vx^0)-\Phi^*\big] + \Big( 2^{-1/3}+\frac{1}{6}2^{1/3}+\frac{1}{36}\Big) \frac{8}{\sqrt[3]{K}} \frac{\sigma^2}{ m_0} +  \frac{12\sigma^2 L}{ \eta m_0}
		+ \frac{ 32}{(1-2^{-2/3})^2} \frac{\sigma^2}{m}\right).
	\end{equation}
\end{theorem}
\begin{proof}
We show the desired result by verifying the conditions in Theorem~\ref{thm:main}. First, with $ \eta_k =\frac{\eta}{L \sqrt[3]{K}}$, the condition in
 \eqref{eq:cond-eta-beta1} becomes 
$$\frac{2 \eta }{\sqrt[3]{K}} + \frac{4 }{m} \frac{\eta^2}{K^{2/3}}  \sum_{j=k+1}^{K-1}\frac{\Gamma_{j}}{\Gamma_{k} }  \le 1, \quad k=0,\ldots,K-1.$$
Notice that when $k=K-1$ the summation above is null. Hence by \eqref{eq:keylem1}, it suffices to require 
	$$\frac{2\eta }{\sqrt[3]{K}} + \frac{4 }{m} \frac{\eta^2 }{K^{2/3}}  \bigg(\frac{1}{2}K^{2/3}+\frac{1}{6}K^{1/3}+\frac{1}{36}\bigg)  \le 1, $$ which is guaranteed when $\eta\le\frac{1}{4}$ and $K\ge1$. Therefore, the condition in
 \eqref{eq:cond-eta-beta1} holds.
	
Secondly, by letting $k=0$ in \eqref{eq:keylem1} and recalling $\Gamma_0=1$, we have 
	$\sum_{k=1}^{K-1}\eta_k\Gamma_{k}\le \big(\frac{1}{2}2^{2/3}+\frac{1}{6}2^{1/3}+\frac{1}{36}\big) \frac{ \eta}{L\sqrt[3]{K}}.$
 Hence, the first condition in \eqref{eq:cond-eta-beta2} holds with
$A=\big(2^{-1/3}+\frac{1}{6}2^{1/3}+\frac{1}{36}\big) \frac{ \eta}{L\sqrt[3]{K}}.$	 
Finally, notice 
	\begin{align}\label{eq:last1}
		 \sum_{k=1}^{K-1}\eta_k\Gamma_{k}\sum_{j=0}^{k-1}\frac{\beta_j^2}{\Gamma_{j+1} } 
		=&~\sum_{j=0}^{K-2}\frac{\beta_j^2}{(1-\beta_j)^2 }  \sum_{k=j+1}^{K-1}\eta_k\frac{\Gamma_{k}}{\Gamma_{j}} \cr
		\le&~ \frac{ \eta}{L\sqrt[3]{K}} \sum_{j=0}^{K-2}\frac{\beta_j^2}{(1-\beta_0)^2 }  
		\bigg(\frac{1}{2}(j+2)^{2/3}+\frac{1}{6}(j+2)^{1/3}+\frac{1}{36}\bigg) \cr
		\le&~ \frac{ \eta}{(1-\beta_0)^2 L \sqrt[3]{K}} \sum_{j=0}^{K-2} 
		\bigg(\frac{1}{2}(j+2)^{-2/3}+\frac{1}{6}(j+2)^{-1}+\frac{1}{36} (j+2)^{-4/3}\bigg) \cr
		\le&~\frac{ \eta}{(1-\beta_0)^2 L \sqrt[3]{K}}   
		\bigg(\frac{3}{2}(K^{1/3}-1)+\frac{1}{6}\log K+\frac{1}{12} (1-K^{-1/3})\bigg) 	\cr
		\le&~ \frac{ 2\eta}{(1-2^{-2/3})^2 L}, 
	\end{align}
where the first inequality follows from \eqref{eq:keylem1}, the decreasing monotonicity of $\beta_k$, and the setting of $\eta_k$, the second inequality holds by $\beta_j \le (j+2)^{-2/3}$, and the last inequality is obtained by $\beta_0\le 2^{-2/3}$ and using the fact $3x^{1/3} > \log x, \forall\, x>0$.	
Thus the second condition in \eqref{eq:cond-eta-beta2} holds with $B = \frac{ 2\eta}{(1-2^{-2/3})^2 L}$. Therefore, \eqref{eq:rate-grad-const2} follows from \eqref{eq:mainthm} and the choice of $\tau$ by uniformly random selection. 
\end{proof}

From Theorem~\ref{thm:rate-const2}, we can immediately obtain the next complexity result of Algorithm~\ref{alg:prox-storm} with the constant stepsize.

\begin{corollary}[complexity result II with constant stepsizes]\label{cor:complexity-const}
 Let $\vareps>0$ be given. Under Assumptions~\ref{assump-obj} through \ref{assump-sgd}, let $\{\vx^k\}$ be the iterate sequence from Algorithm~\ref{alg:prox-storm} with $ \eta_k =\frac{\eta}{L \sqrt[3]{K}}$ and $\{\beta_k\}$ set by \eqref{eq:beta-def}, 
	where $\eta \le \frac{1}{4}$ is a positive number. 
Let $\tau$ be selected from $\{0,1,\ldots,K-1\}$ uniformly at random. Then $\vx^\tau$ is a stochastic $\vareps$-stationary solution of \eqref{eq:stoc-prob} if  
\begin{equation}\label{eq:K-const-2}
K =\left\lceil\frac{ \left(\frac{12 L}{\eta}\big[\Phi(\vx^0)-\Phi^*\big] + \big( 2^{-1/3}+\frac{1}{6}2^{1/3}+\frac{1}{36}\big) \frac{8\sigma^2}{ m_0} +  \frac{12\sigma^2 L}{ \eta m_0}
		+ \frac{ 32}{(1-2^{-2/3})^2} \frac{\sigma^2}{m} \right)^{3/2} }{\vareps^3}\right\rceil.
\end{equation} 
\end{corollary}

\begin{remark}\label{rm:dep-sigma-const-2}
Let $m_0=O(1)$ and $m=O(1)$. Then we have from \eqref{eq:K-const-2} that $K= O(\vareps^{-3})$ by ignoring the dependence on other quantities and the total sample complexity is $m_0+m(K-1) = O(\vareps^{-3})$, which matches with the lower bound in \cite{arjevani2019lower}. However, as we need $\eta \le \frac{1}{4}$, the dependence on $L\big(\Phi(\vx^0)-\Phi^*\big)$ will be not as good as the result in \eqref{eq:dep-sigma-const}.
\end{remark}

\section{Numerical Experiments}\label{sec:numerical}
In this section, we test Algorithm~\ref{alg:prox-storm}, named as PStorm, on solving three problems. The first problem is the nonnegative principal component analysis (NPCA) \cite{reddi2016proximal}, and the other two are on training neural networks. We compare PStorm to the vanilla proximal SGD, Spiderboost \cite{wang2019spiderboost}, and Hybrid-SGD \cite{tran2021hybrid}. Spiderboost and Hybrid-SGD both achieve optimal complexity results, and the vanilla proximal SGD is used as a baseline for the comparison. For NPCA, all methods were implemented in MATLAB 2021a on a quad-core iMAC with 40 GB memory, and for training neural networks, all methods were implemented by using PyTorch on a Dell workstation with 32 CPU cores, 2 GPUs, and 64 GB memory.  

\subsection{Nonnegative Principal Component Analysis (NPCA)}\label{sec:npca}
In this subsection, we compare the four methods on solving the NPCA problem:
\begin{equation}\label{eq:npca}
\Max_{\vx\in\RR^n} \frac{1}{2}\EE_\vz[\vx^\top(\vz\vz^\top)\vx], \st \|\vx\|\le 1, \vx\ge\vzero,
\end{equation}
where $\vz\in\RR^n$ represents a random data point following a certain distribution, and $\EE_\vz$ takes expectation about $\vz$. The problem \eqref{eq:npca} can be formulated into the form of \eqref{eq:stoc-prob}, by negating the objective and adding an indicator function of the constraint. Two datasets were used in this test. The first one takes $\vz=\frac{\vw}{\|\vw\|}$ where $\vw\sim\cN(\vone, \vI)$, and we solved a stochastic problem; for the second one, we used the normalized training and testing datasets of \verb|realsim| from LIBSVM \cite{chang2011libsvm}, and we solved a deterministic finite-sum problem. For both datasets, each sample function in the objective of \eqref{eq:npca} is 1-smooth, and thus we used the Lipschitz constant $L=1$ for all methods.

\noindent\textbf{Random dataset:}~~For the randomly generated dataset, we set the dimension $n=100$ and the minibatch size to $m=10$ for PStorm, the vanilla proximal SGD, and the Hybrid-SGD. For the Spiderboost, we set $\vareps=5\times10^{-3}$, and for each iteration $k$, it accessed $q=\vareps^{-1}$ data samples if $\text{mod}(k,q)\neq 0$ and $\vareps^{-2}$ data samples otherwise.  Each method could access at most $10^6$ data samples. The stepsize of PStorm was set according to \eqref{eq:dynamic-para} with $\eta$ tuned from $\{0.1,0.2,0.5,1\}$, out of which $\eta=0.1$ turned out the best. The stepsize of the vanilla proximal SGD was set to $\frac{\eta}{\sqrt{k+1}}$ for each iteration $k\ge0$ with $\eta$ tuned from $\{0.1,0.2,0.5,1\}$, out of which $\eta=0.5$ turned out the best. The stepsize of Spiderboost was set to $\eta=0.5$. The Hybrid-SGD has a few more parameters to tune. As suggested by  \cite[Theorem 4]{tran2021hybrid} and also its numerical experiments, we set $\gamma_k$, $\beta_k$, $\eta_k$, and the initial batch size to
\begin{equation}\label{eq:para-rand-set}
\gamma_k\equiv\gamma = \frac{3c_0 m^{\frac{3}{4}}}{\sqrt{13}m_0 (K+1)^{\frac{1}{4}}},\ \beta_k\equiv\beta=1-\frac{\sqrt{m}}{\sqrt{m_0 K}}, \ \eta_k\equiv\eta = \frac{2}{L(3+\gamma)},\ m_0=\frac{c_1^2}{\lceil m (K+1)^{\frac{1}{3}}\rceil},
\end{equation}
where $K$ is the maximum number of iterations. We tuned $c_0$ to 10 and $c_1$ to 5. 

To evaluate the performance of the tested methods, we randomly generated $10^7$ data samples following the same distribution as we described above, and at the iterates of the methods, we computed their violation of stationarity of the sample-approximation problem. Since the compared methods have different learning rate, to make a fair comparison, we measured the \emph{violation of stationarity} at $\vx$ by $\|P(\vx, \nabla F, 1)\|$, where $P$ is the proximal mapping defined in Definition~\ref{def:prox-map}, and $F$ is the sample-approximated objective. Also, to obtain the ``optimal'' objective value, we ran the projected gradient method to 1,000 iterations on the deterministic sample-approximation problem. The results in terms of the number of samples are plotted in Figure~\ref{fig:snpca}, which clearly shows the superiority of PStorm over all the other three methods.

\begin{figure}[h]
\begin{center}
\begin{tabular}{cc}
\includegraphics[width=0.3\textwidth]{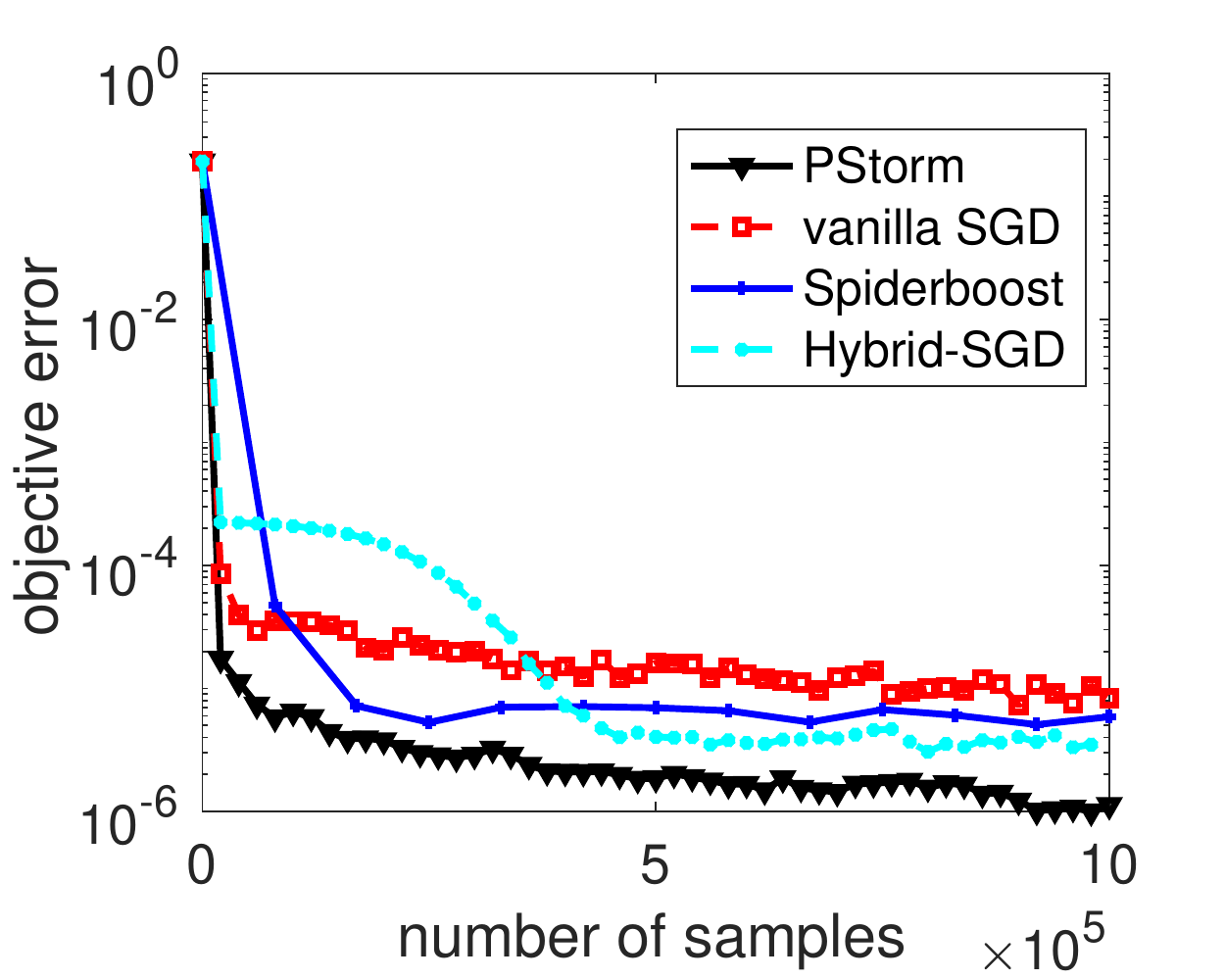} &
\includegraphics[width=0.3\textwidth]{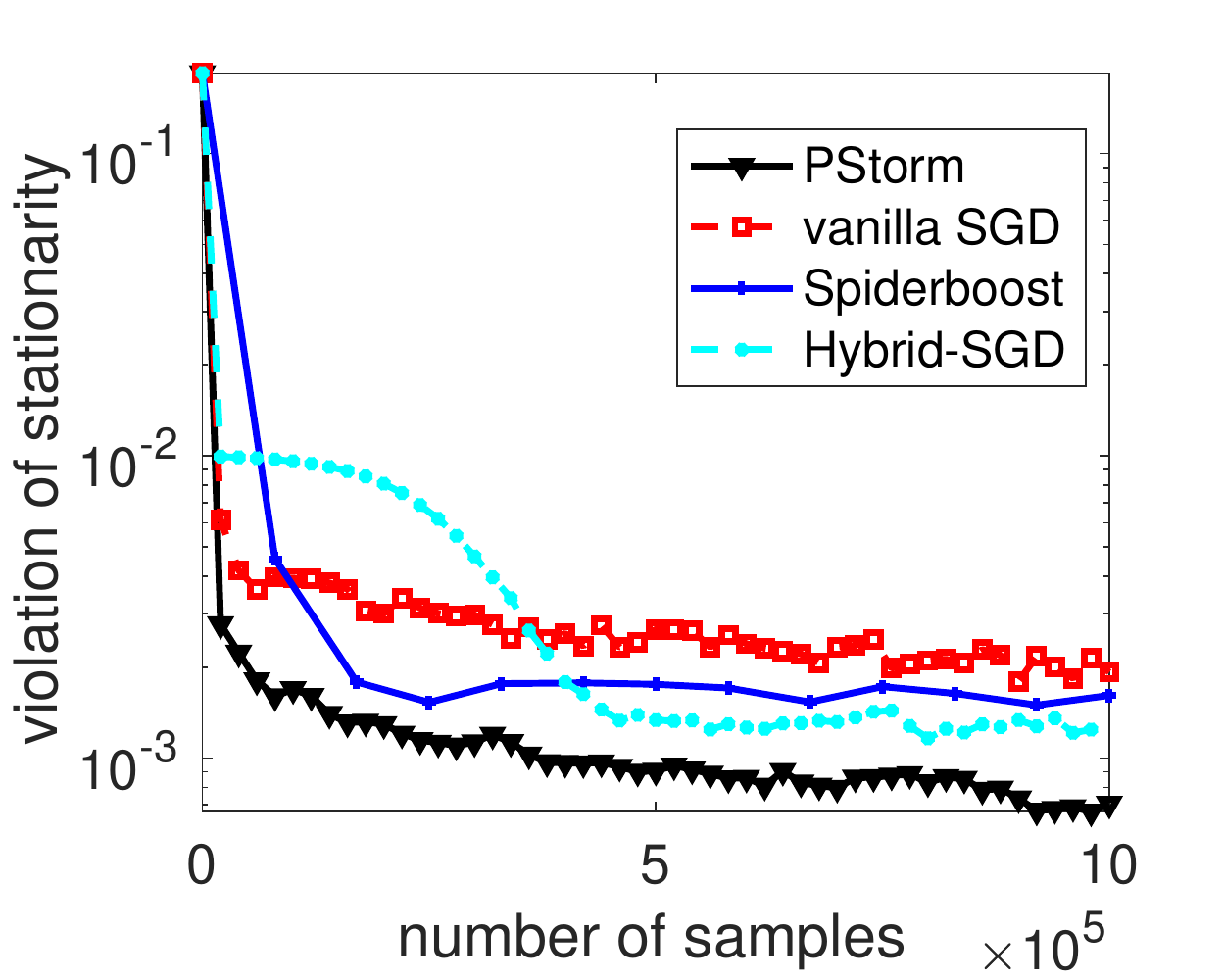}
\end{tabular}
\end{center}
\caption{Objective error and the violation of stationarity by PStorm, the vanilla SGD, Spiderboost, and Hybrid-SGD on solving \eqref{eq:npca} with randomly generated dataset.}\label{fig:snpca}
\end{figure}

\noindent\textbf{realsim dataset:}~~The \verb|realsim| dataset has $N=72,309$ samples in total. We set the minibatch size to $m=64$ for PStorm, the vanilla proximal SGD, and the Hybrid-SGD. For each iteration $k$ of the Spiderboost, we set $|B_k| = q=\lceil \sqrt{N} \rceil =269$ in \eqref{eq:vk-spider}, as suggested by \cite[Theorem 3]{wang2019spiderboost}. The stepsizes of PStorm and the vanilla proximal SGD were tuned in the same way as above, and the best $\eta$ was 0.2 for the former and 0.5 for the latter. The stepsize for Spiderboost was still set to $0.5$ as the smoothness constant is $L=1$. For Hybrid-SGD, we set its parameters to
$$\gamma_k\equiv\gamma = 0.95,\ \beta_k\equiv\beta=1-\frac{\sqrt{m}}{\sqrt{m_0 K}}, \ \eta_k\equiv\eta = \frac{2}{L(3+\gamma)},\ m_0=\max\left\{N, \frac{c_1^2}{\lceil m (K+1)^{\frac{1}{3}}\rceil}\right\},$$
where $K$ is the maximum number of iterations and $c_1$ was tuned to 15. Notice that different from \eqref{eq:para-rand-set}, here we simply fix $\gamma=0.95$. This choice of $\gamma$ was also adopted in \cite{tran2021hybrid}, and it turned out that this setting resulted in the best performance of Hybrid-SGD for this test.

We ran each method to 100 epochs, where one epoch is equivalent to one pass of all data samples. The results in terms of epoch number are shown in Figure~\ref{fig:npca}, where the violation of stationary was again measured by $\|P(\vx, \nabla F, 1)\|$ and the ``optimal'' objective value was given by running the projected gradient method to 1,000 iterations. For this test, we found that Spiderboost converges extremely fast and gave much smaller errors than those by other methods, and thus we plot the results by Spiderboost in separate figures. PStorm performed better than the vanilla proximal SGD and the Hybrid-SGD. We also tested the methods on the datasets \verb|w8a| and \verb|gisette| from LIBSVM. Their comparison performance was similar to that on \verb|realsim|.

\begin{figure}[h]
\begin{center}
\includegraphics[width=0.24\textwidth]{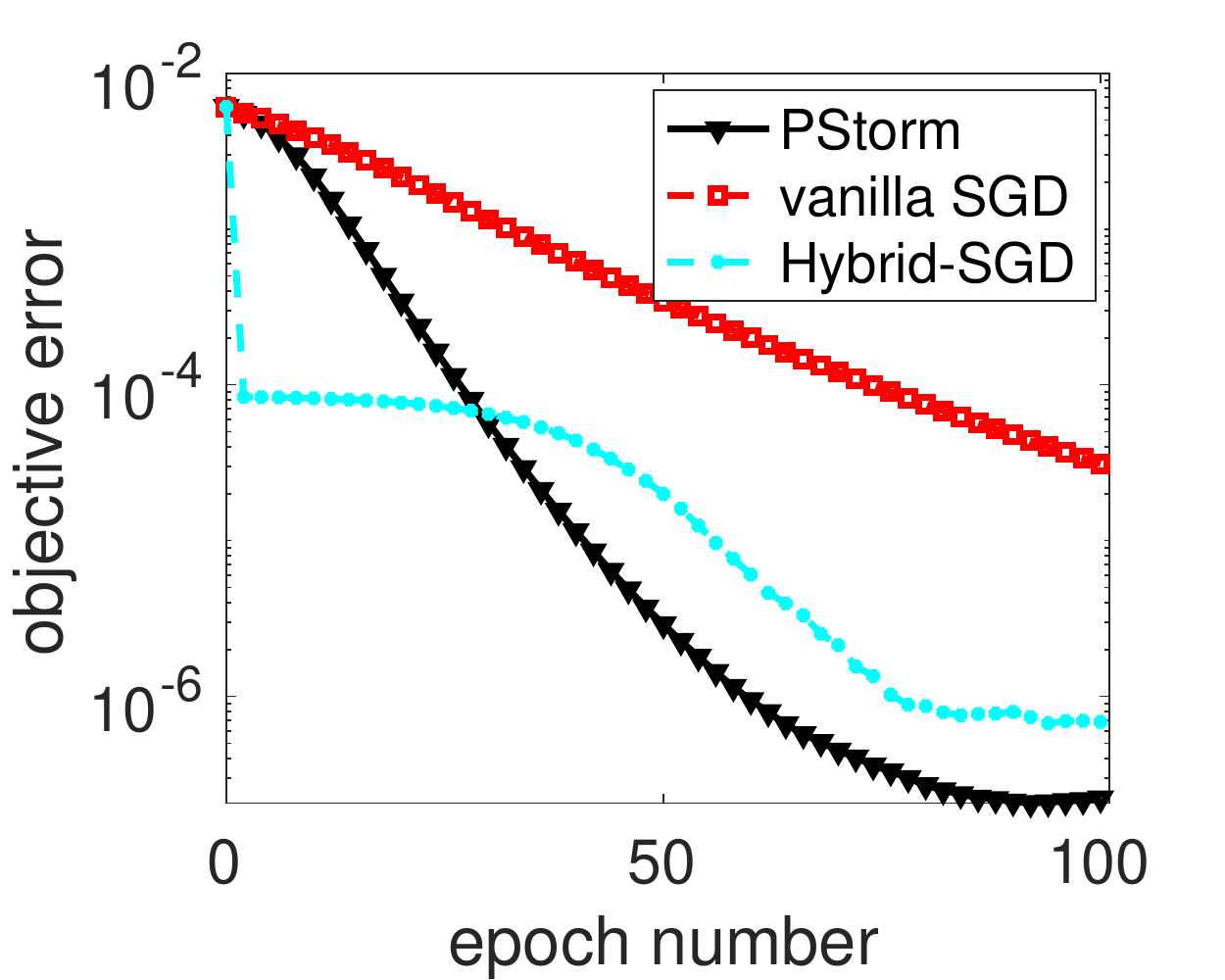} 
\includegraphics[width=0.24\textwidth]{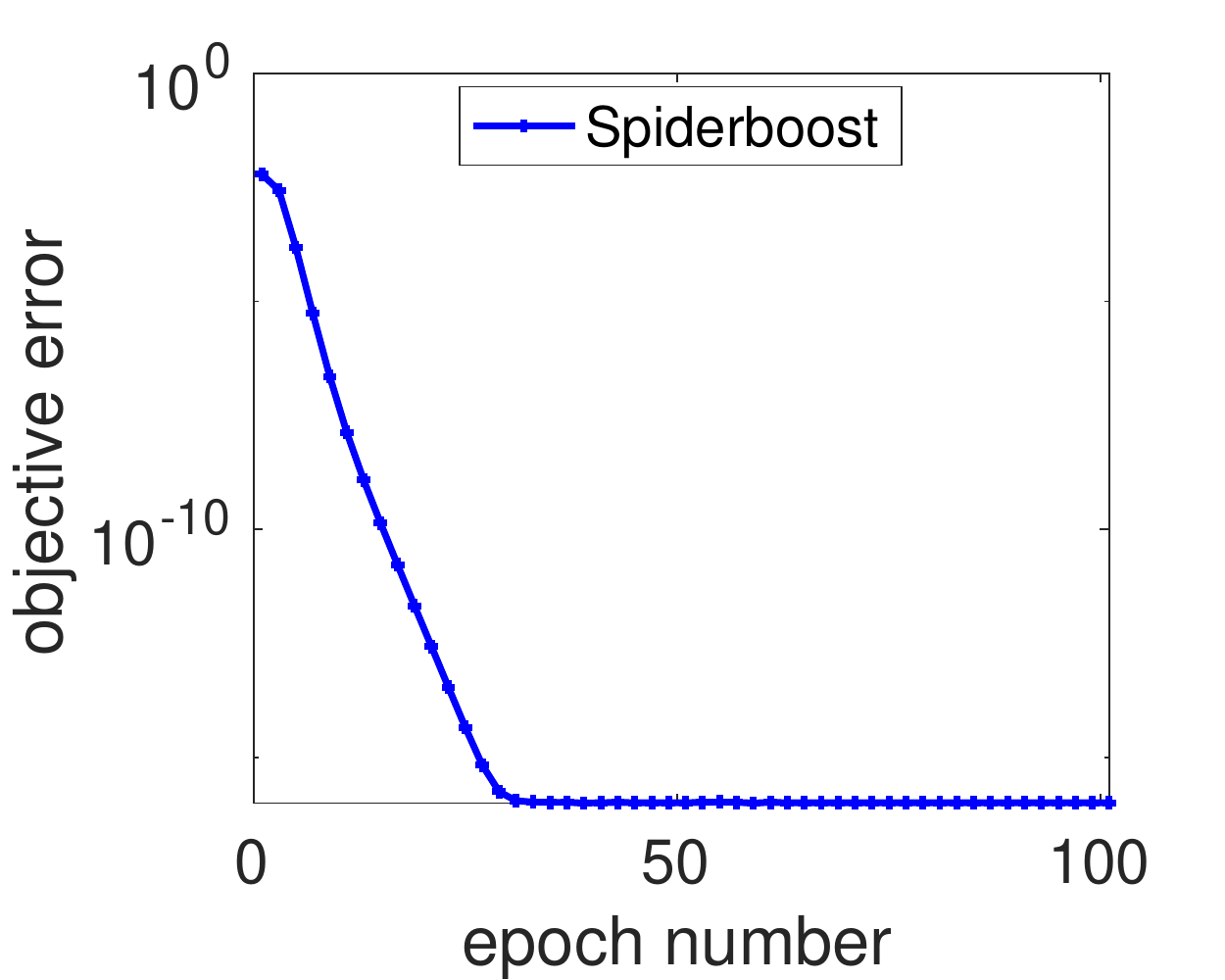} 
\includegraphics[width=0.24\textwidth]{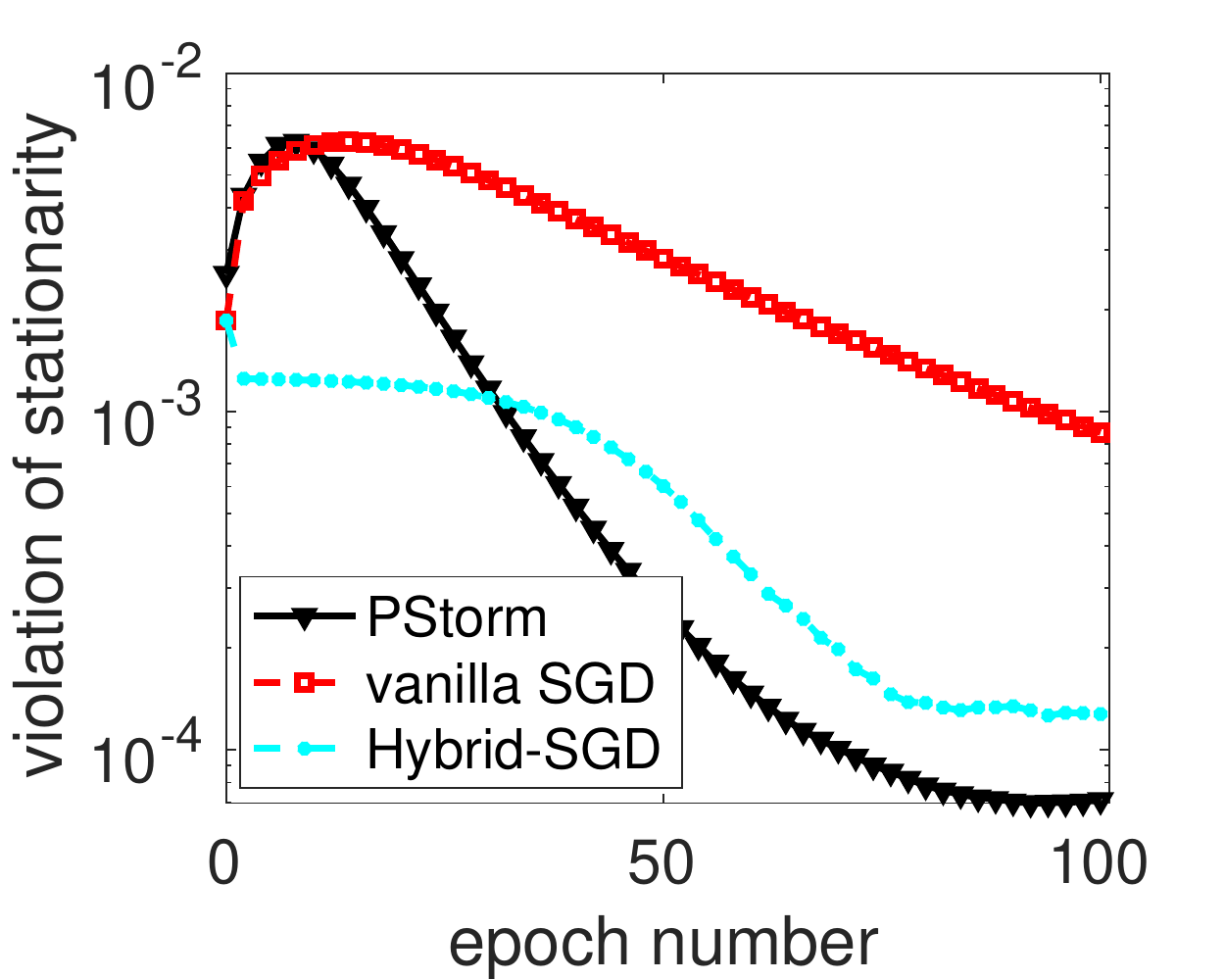} 
\includegraphics[width=0.24\textwidth]{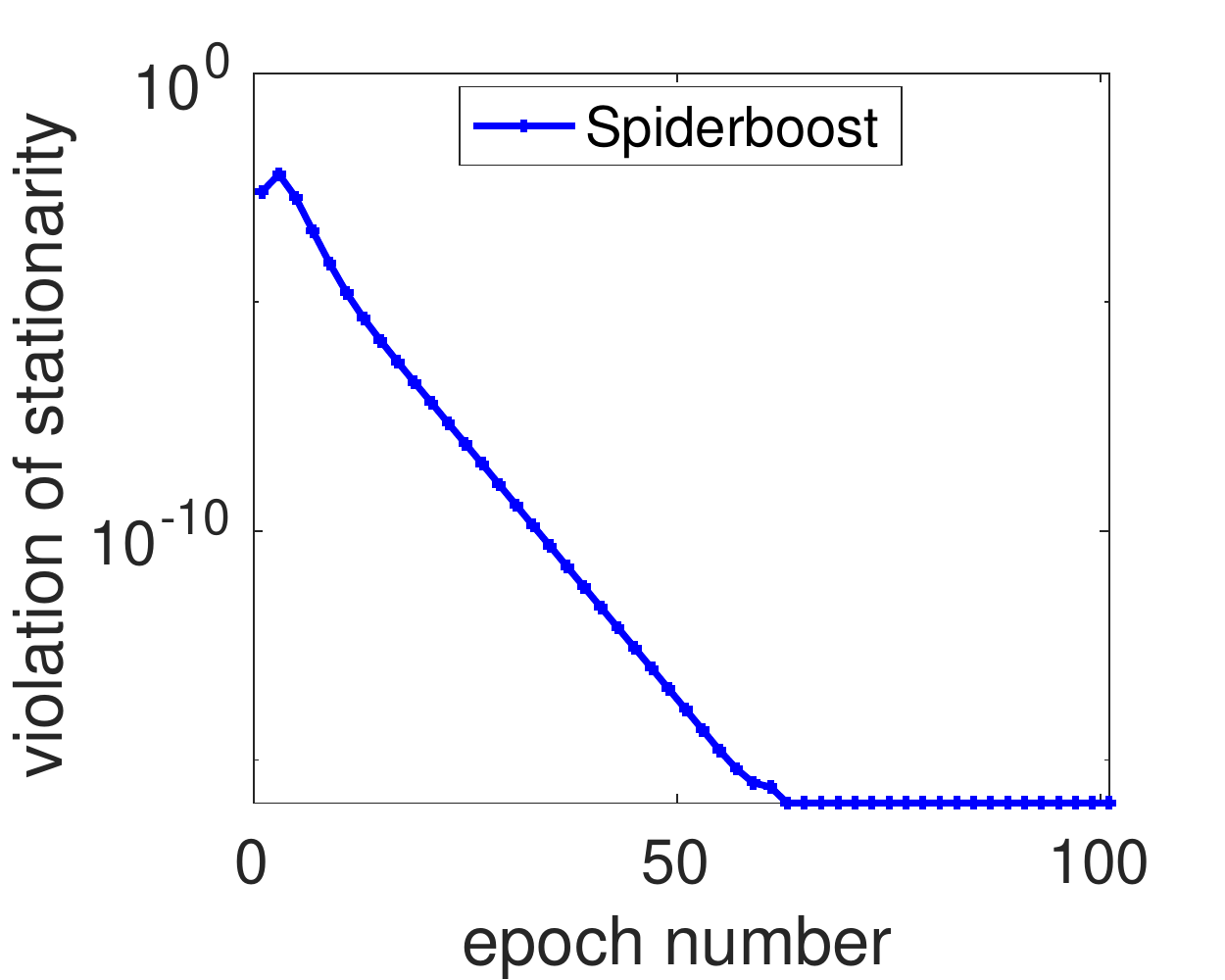} 
\end{center}
\caption{Objective error and the violation of stationarity by PStorm, the vanilla SGD, Spiderboost, and Hybrid-SGD on solving \eqref{eq:npca} with realsim dataset.}\label{fig:npca}
\end{figure}

\subsection{Regularized Feedforward Fully-connected Neural Network}\label{sec:fnn}
In this subsection, we compare different methods on solving an $\ell_1$-regularized 3-layer feedforward fully-connected neural network, formulated as
\begin{equation}\label{eq:sparseDNN}
\min_{\vtheta} \frac{1}{N}\sum_{i=1}^N \ell\Big(\mathrm{softmax}\big(\vW_3\sigma (\vW_2\sigma(\vW_1\vx_i))\big), y_i\Big) + \lambda\big(\|\vW_1\|_1+\|\vW_2\|_1+\|\vW_3\|_1\big).
\end{equation}
Here $\{(\vx_i,y_i)\}_{i=1}^N$ is a $c$-class training data set with $y_i\in \{1,\ldots,c\}$ for each $i$, $\vtheta:=(\vW_1,\vW_2,\vW_3)$ contains the parameters of the neural network, $\sigma(\cdot)$ is an activation function, $\ell$ denotes a loss function, $\mathrm{softmax}(\vz):= \frac{1}{\sum_{j=1}^c e^{z_j}}[e^{z_1}; \ldots; e^{z_c} ]\in\RR^c, \forall\, \vz\in\RR^c$, and $\lambda\ge0$ is a regularization parameter to trade off the loss and sparsity.

In the test, we used the MNIST dataset \cite{lecun1998gradient} of hand-written-digit images. The training set has 60,000 images, and the testing set has 10,000 images. Each image was originally $28\times 28$ and vectorized into a vector of dimension $784$. We set $\vW_1\in\RR^{784\times 120}, \vW_2\in\RR^{120\times 84}$, and $\vW_3\in\RR^{84\times 10}$, whose initial values were set to the default ones in \verb|libtorch|, a C++ distribution of PyTorch. We used the hyperbolic tangent activation function $\sigma(x)=\frac{e^x - e^{-x}}{e^x + e^{-x}}$ and the cross entropy $\ell(\vq, y_i)=-\log q_{y_i}$ for any distribution $\vq\in\RR^c$. 

The parameters of PStorm were set according to \eqref{eq:dynamic-para} with  $L=1$ and $\eta=\frac{\sqrt[3]{4}}{8}\approx 0.198$. Notice that the gradient of the loss function in \eqref{eq:sparseDNN} is not uniformly Lipschitz continuous, and its Lipschitz constant depends on $\vtheta$. More specifically, the gradient is Lipschitz continuous over any bounded set of $\vtheta$. However, PStorm with this parameter setting performed well. The learning rate of the vanilla SGD was set to $\eta_k = \frac{\eta}{\sqrt{k+1}}, \forall\, k\ge0$ with $\eta=\frac{\sqrt[3]{4}}{8}$. We also tried $\eta=0.5$, and it turned out that the performance of the vanilla SGD was not as well as that with $\eta=\frac{\sqrt[3]{4}}{8}$ when $\lambda>0$ in \eqref{eq:sparseDNN}. For Spiderboost, we set $q =\lceil \sqrt{60000}\rceil =245$ in \eqref{eq:vk-spider} as specified by \cite[Theorem 2]{wang2019spiderboost} and its learning rate $\eta=0.02$ in \eqref{eq:spiderboost}. We also tried $\eta=0.1$ and $\eta=0.01$. It turned out that Spiderboost could diverge  with $\eta=0.1$ and converged too slowly with $\eta=0.01$. For Hybrid-SGD, we fixed its parameter $\gamma=0.95$ as suggested in the numerical experiments of \cite{tran2021hybrid}, and we set $\beta_k=\beta=1-\frac{1}{\sqrt{K+1}},\forall k\ge0$ in \eqref{eq:hyb-sgd-v}, where $K$ is the maximum number of iterations. Its learning rate was set to $\eta=\frac{2}{4+L\gamma}$. Then we chose the initial mini-batch size $m_0$ from $\{256, 2560, 30000, 60000\}$ and $L$ from $\{5, 10, 50, 100\}$. The best results were reported. 

We ran each method to 100 epochs. Mini-batch size was set to 32 for PStorm, the vanilla SGD, and Hybrid-SGD. 
Again, to make a fair comparison, we measured the \emph{violation of stationarity} at $\vtheta$ by $\|P(\vtheta, \nabla F, 1)\|$, where $P$ is the proximal mapping defined in Definition~\ref{def:prox-map}, and $F$ is the smooth term in the objective of \eqref{eq:sparseDNN}. Table~\ref{tbl: sp-DNN-32} and Figure~\ref{fig:sp-DNN-32} show the results by the compared methods. Each result in the table is the average of those at the last five epochs. For Hybrid-SGD, the best results were obtained with $(m_0, L) = (60000, 50)$ when $\lambda=0$ and with $(m_0, L) = (60000, 100)$ when $\lambda>0$. From the results, we see that PStorm and Hybrid-SGD give similar training loss and testing accuracies while the vanilla SGD and Spiderboost yield higher loss and lower accuracies. The lower accuracies by Spiderboost may be caused by its larger batch size that is required in \cite{wang2019spiderboost}, and the lower accuracies by the vanilla SGD are because of its slower convergence. In addition, PStorm produced sparser solutions than those by other methods in all regularized cases. In terms of the violation of stationarity, the solutions by PStorm have better quality than those by other methods. Furthermore, we notice that the model \eqref{eq:sparseDNN} trained by PStorm with $\lambda = 5\times 10^{-4}$ is much sparser than that without the $\ell_1$ regularizer, but the sparse model gives just slightly lower testing accuracy than the dense one. This is important because a sparser model would reduce the inference time when the model is deployed to predict new data.

\begin{table}[h]
\begin{center}
\resizebox{0.99\textwidth}{!}{
\begin{tabular}{|c||cccc||cccc||cccc||cccc|}
\hline
Method & \multicolumn{4}{|c||}{PStorm} & \multicolumn{4}{|c||}{vanilla SGD} & \multicolumn{4}{|c||}{Spiderboost} & \multicolumn{4}{|c|}{Hybrid-SGD}\\\hline\hline
$\lambda$ &  train & test & grad & density & train & test & grad & density & train & test & grad & density & train & test & grad & density \\\hline
0.00 & 3.61e-3 & \textbf{98.01} & \textbf{3.45e-3} & 100 & 6.91e-2 & 97.09 & 3.42e-2 & 100 & 4.24e-2 & 97.41 & 1.57e-2 & 100& \textbf{1.50e-3} & 97.11 & 3.64e-3 & 100 \\
2e-4 & 4.38e-2 & 97.60 & \textbf{1.60e-2} & \textbf{14.06} & 1.08e-1 & 96.62 & 5.77e-2 & 99.47 & 8.70e-2 & 97.24 & 1.87e-2 & 27.17& \textbf{4.08e-2} & \textbf{97.78} & 9.53e-2 & 28.29\\
5e-4 & 8.86e-2 & \textbf{97.12} & \textbf{1.94e-2} & \textbf{6.16} & 1.69e-1 & 95.54 & 5.96e-2 & 92.86 & 1.41e-1 & 96.16 & 2.18e-2 & 10.62& \textbf{8.34e-2} & \textbf{97.12} & 1.11e-1 & 12.69 \\\hline
\end{tabular}
}
\end{center}
\caption{Results by the proposed method PStorm, the vanilla SGD, Hybrid-SGD, and Spiderboost on training the model \eqref{eq:sparseDNN}. The first three methods use mini-batch $m = 32$. Each method runs to 100 epochs. ``train'' is for training loss; ``test'' is for testing accuracy;  ``grad'' is for the violation of stationarity; ``density'' is for the percentage of nonzeros in the solution. The best results for ``test'', ``grad'', and ``density'' are highlighted in \textbf{bold}.}\label{tbl: sp-DNN-32}
\end{table}

\begin{figure}[h]
\begin{center}
\begin{tabular}{ccc}
$\lambda = 0$ & $\lambda = 2\times 10^{-4}$ & $\lambda = 5\times 10^{-4}$ \\
\includegraphics[width=0.25\textwidth]{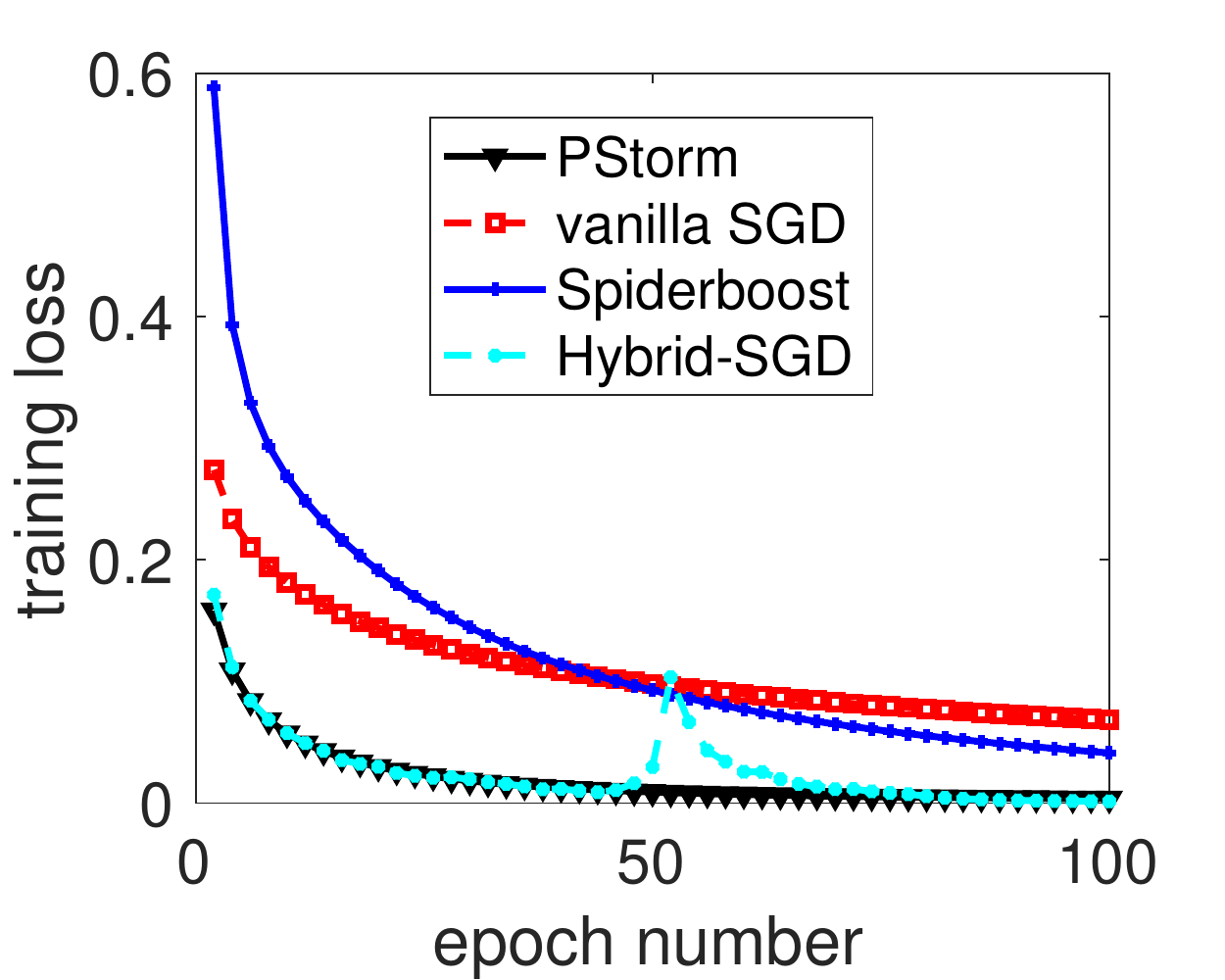} &
\includegraphics[width=0.25\textwidth]{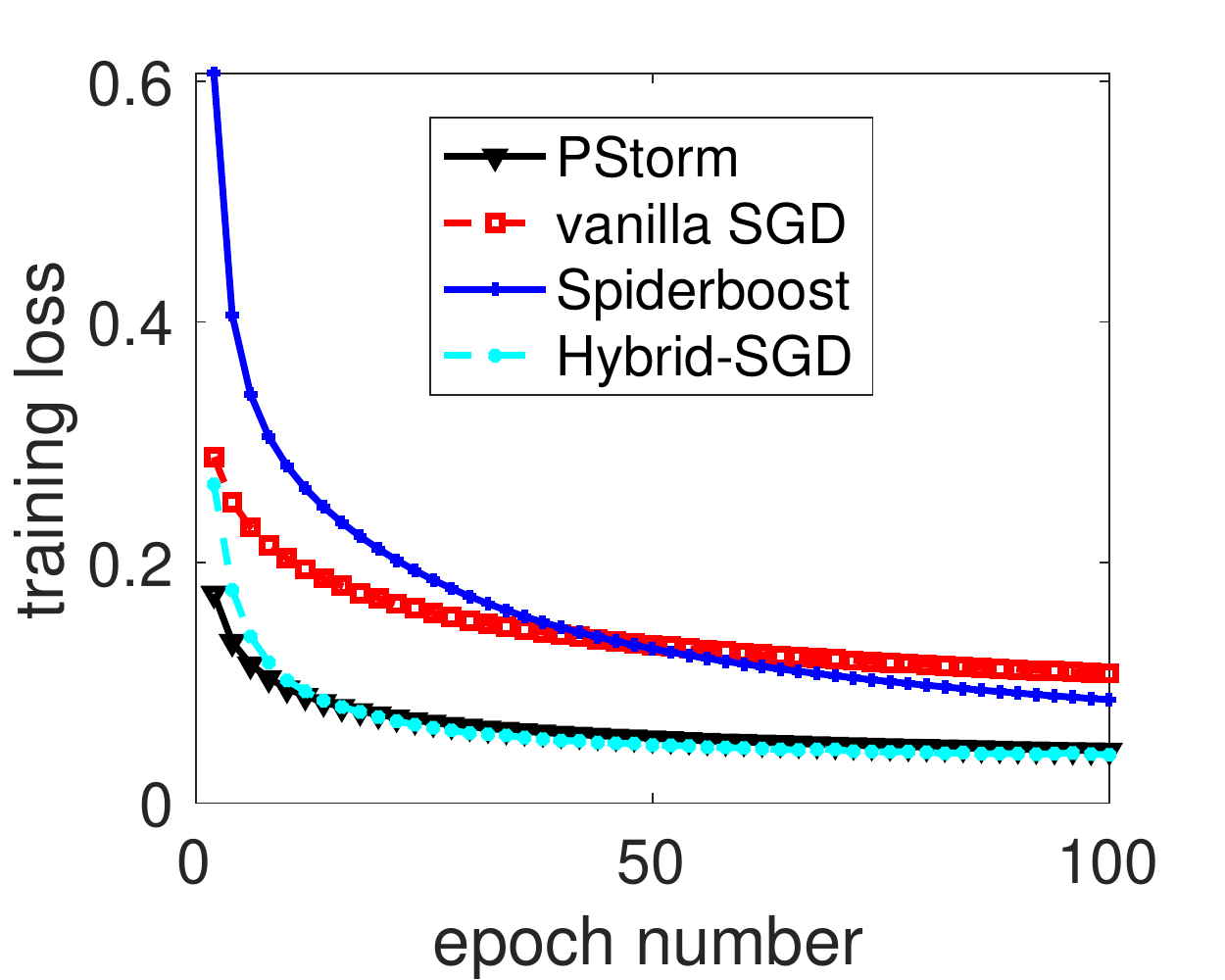} &
\includegraphics[width=0.25\textwidth]{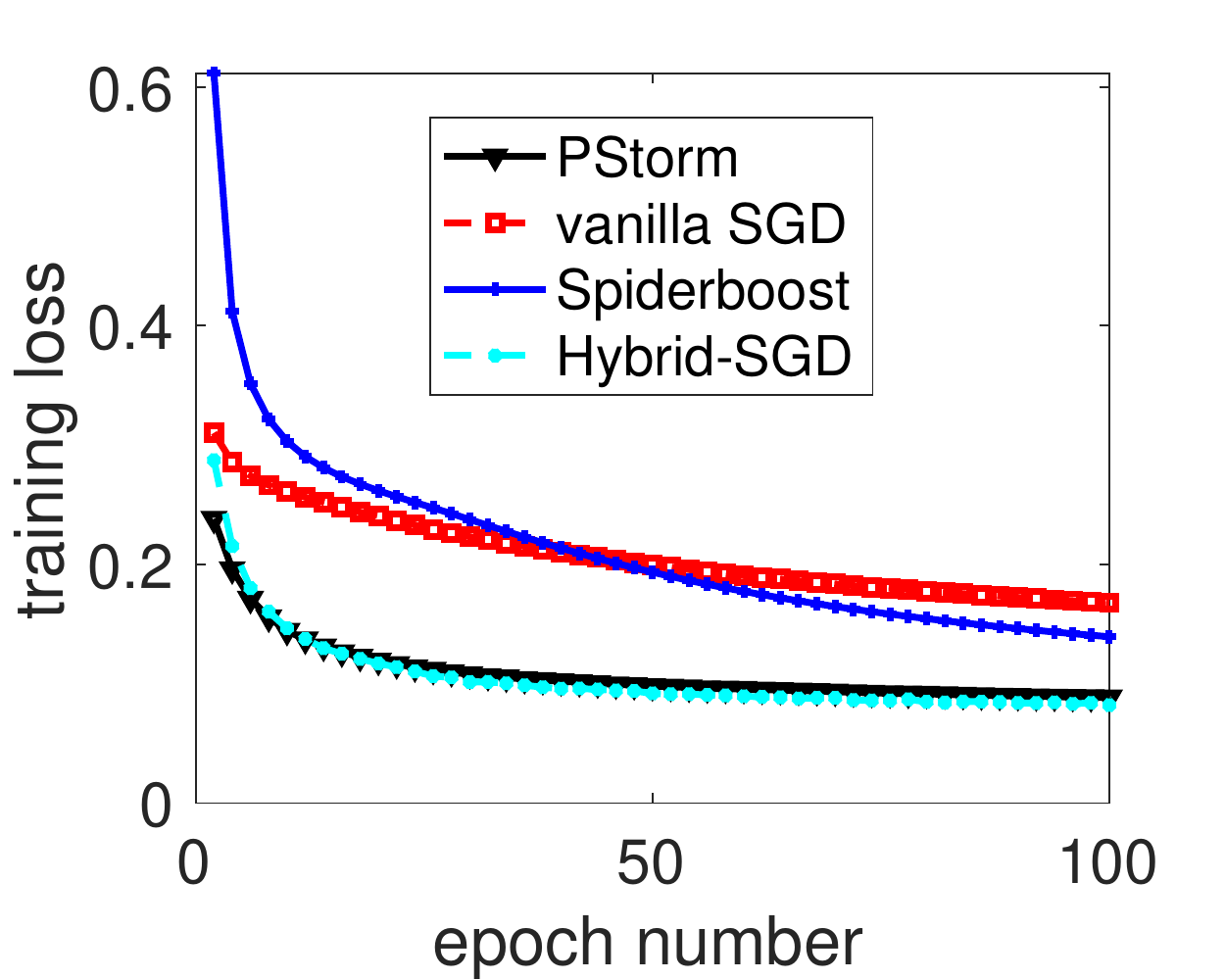} \\
\includegraphics[width=0.25\textwidth]{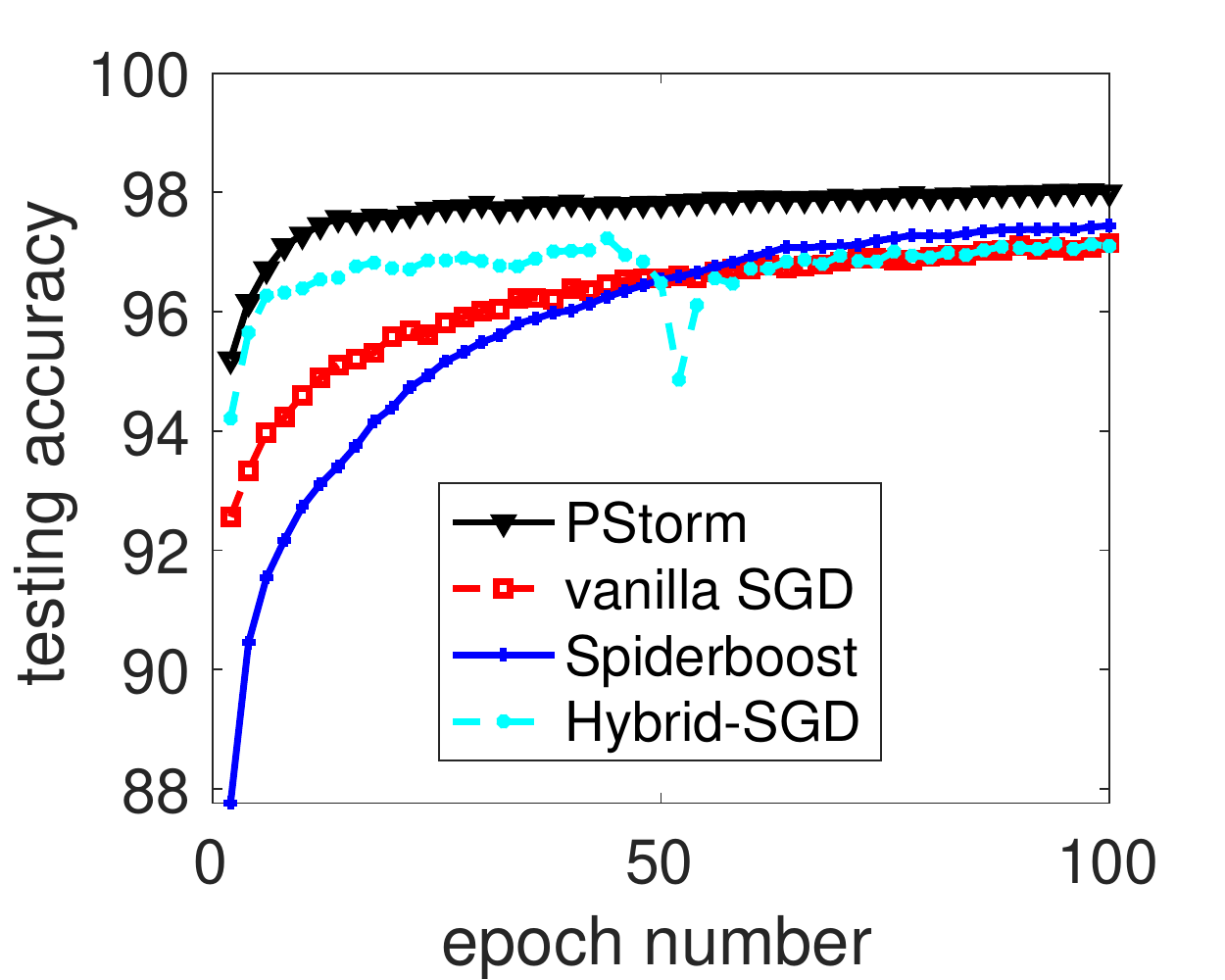} &
\includegraphics[width=0.25\textwidth]{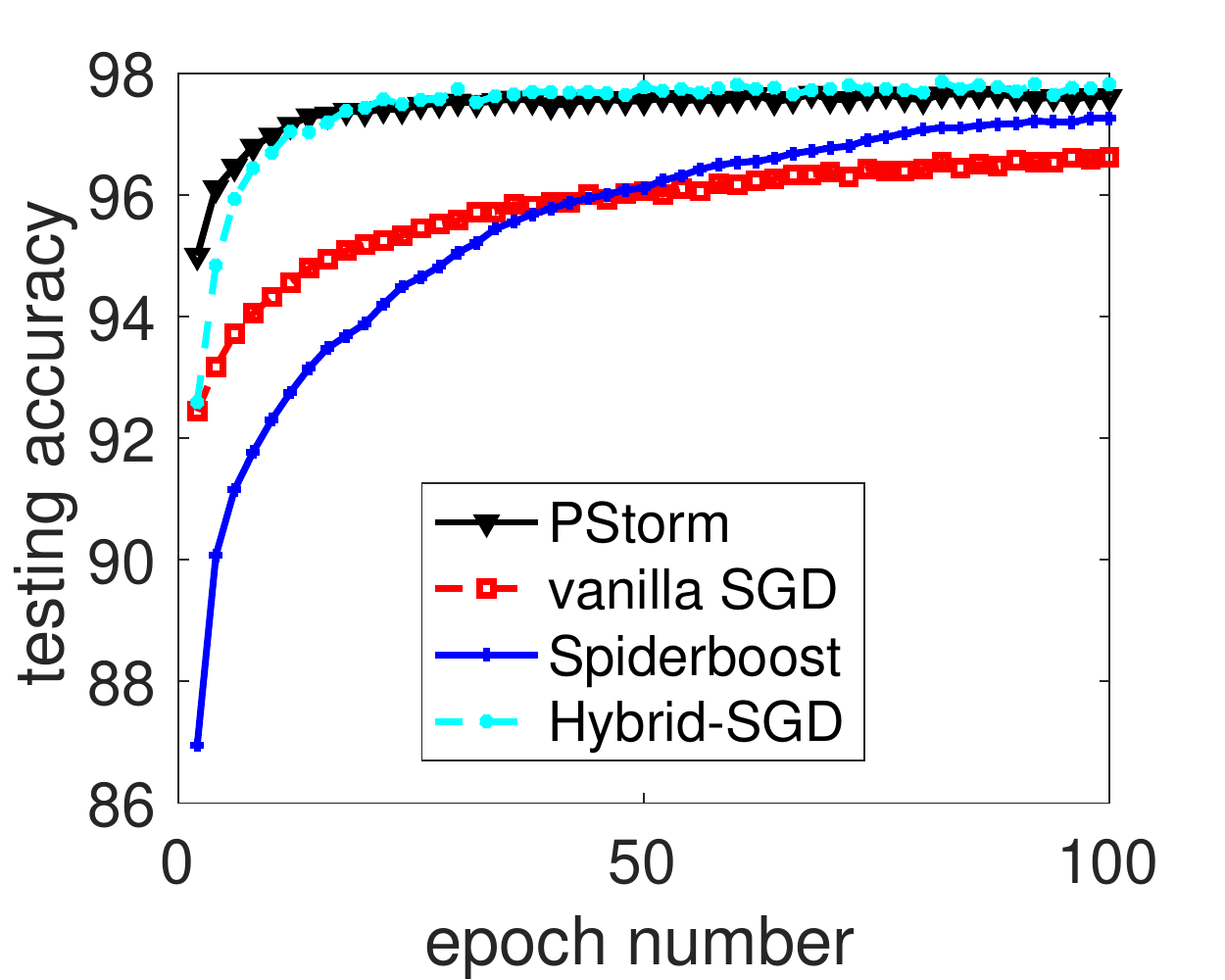} &
\includegraphics[width=0.25\textwidth]{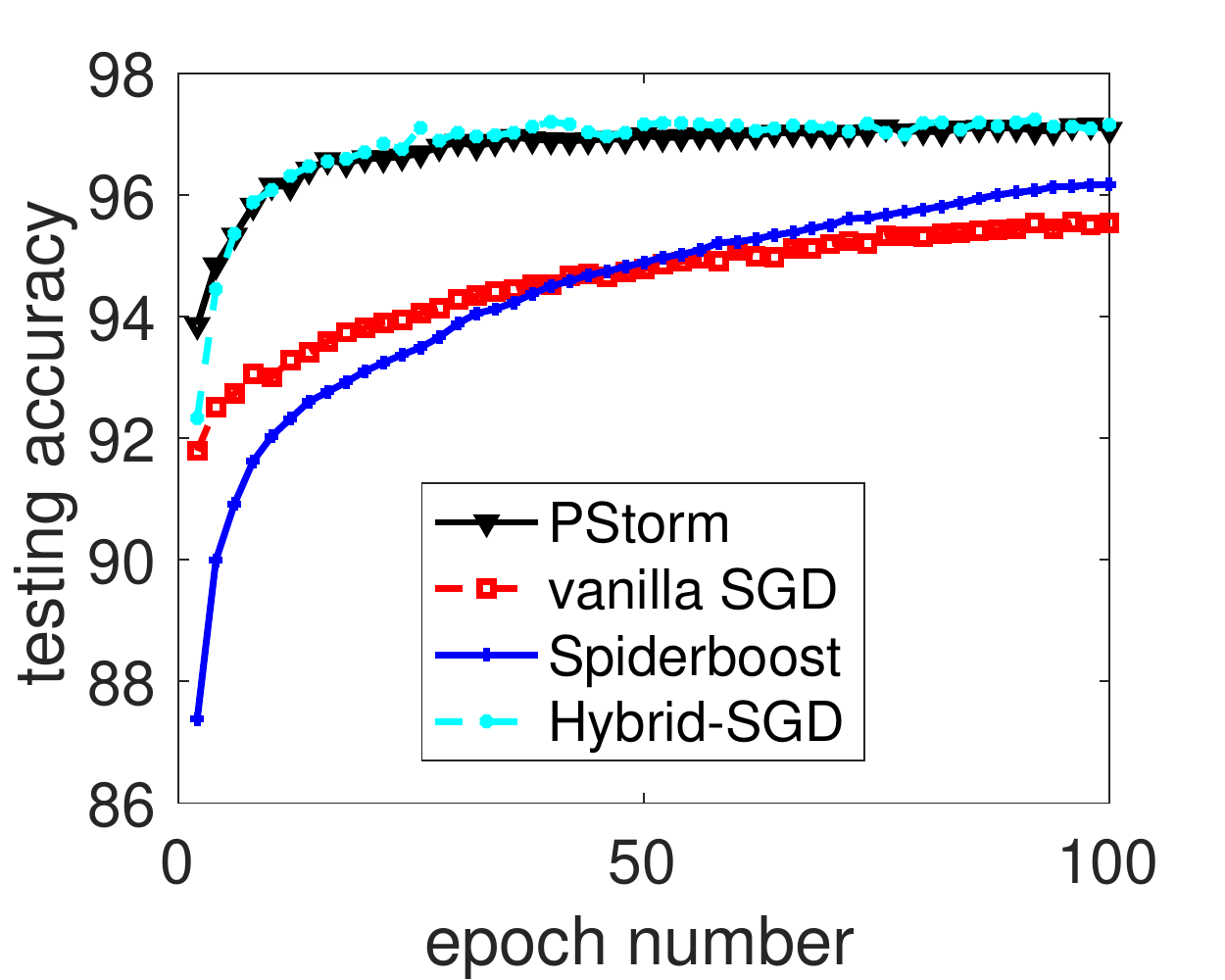} \\
\includegraphics[width=0.25\textwidth]{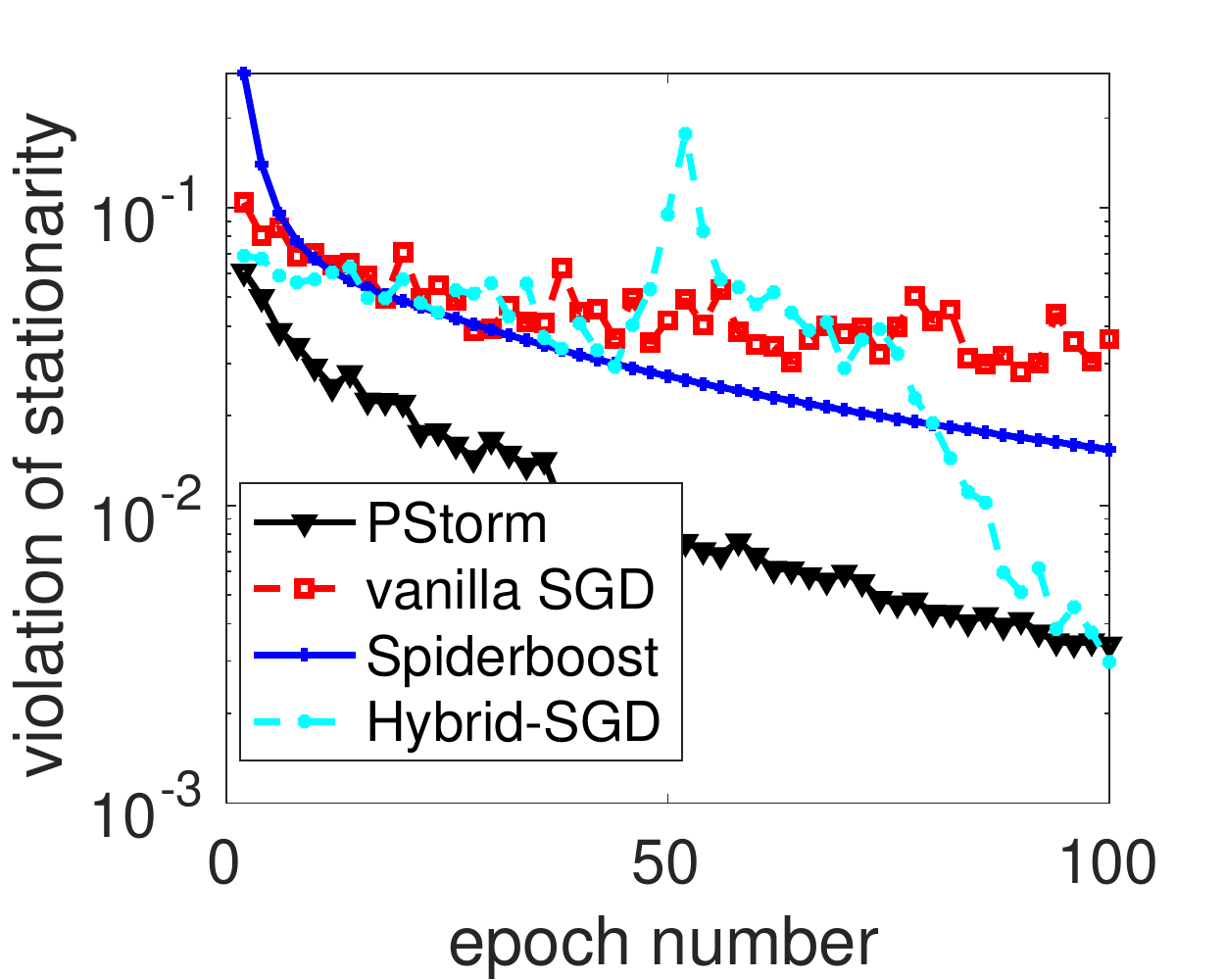} &
\includegraphics[width=0.25\textwidth]{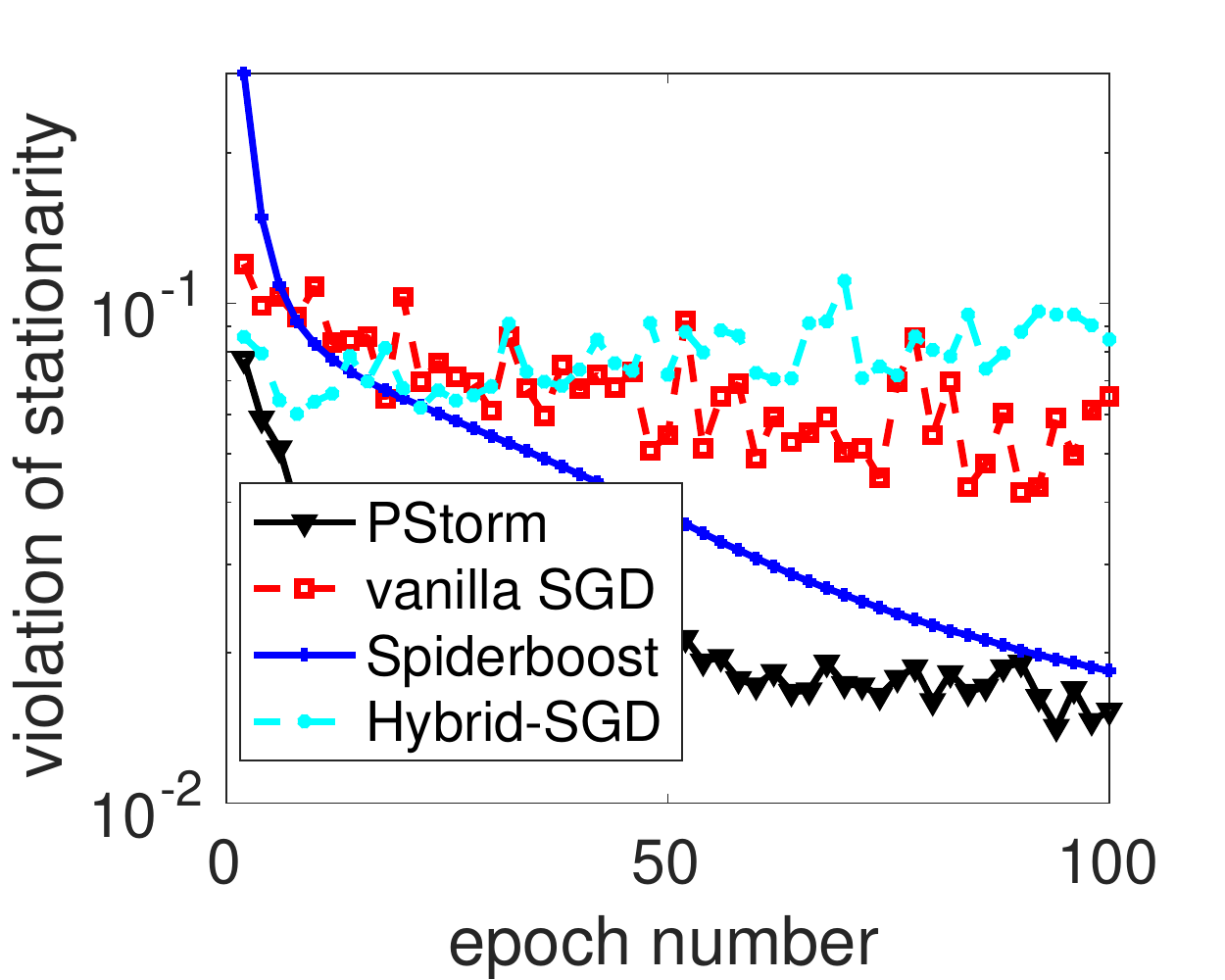} &
\includegraphics[width=0.25\textwidth]{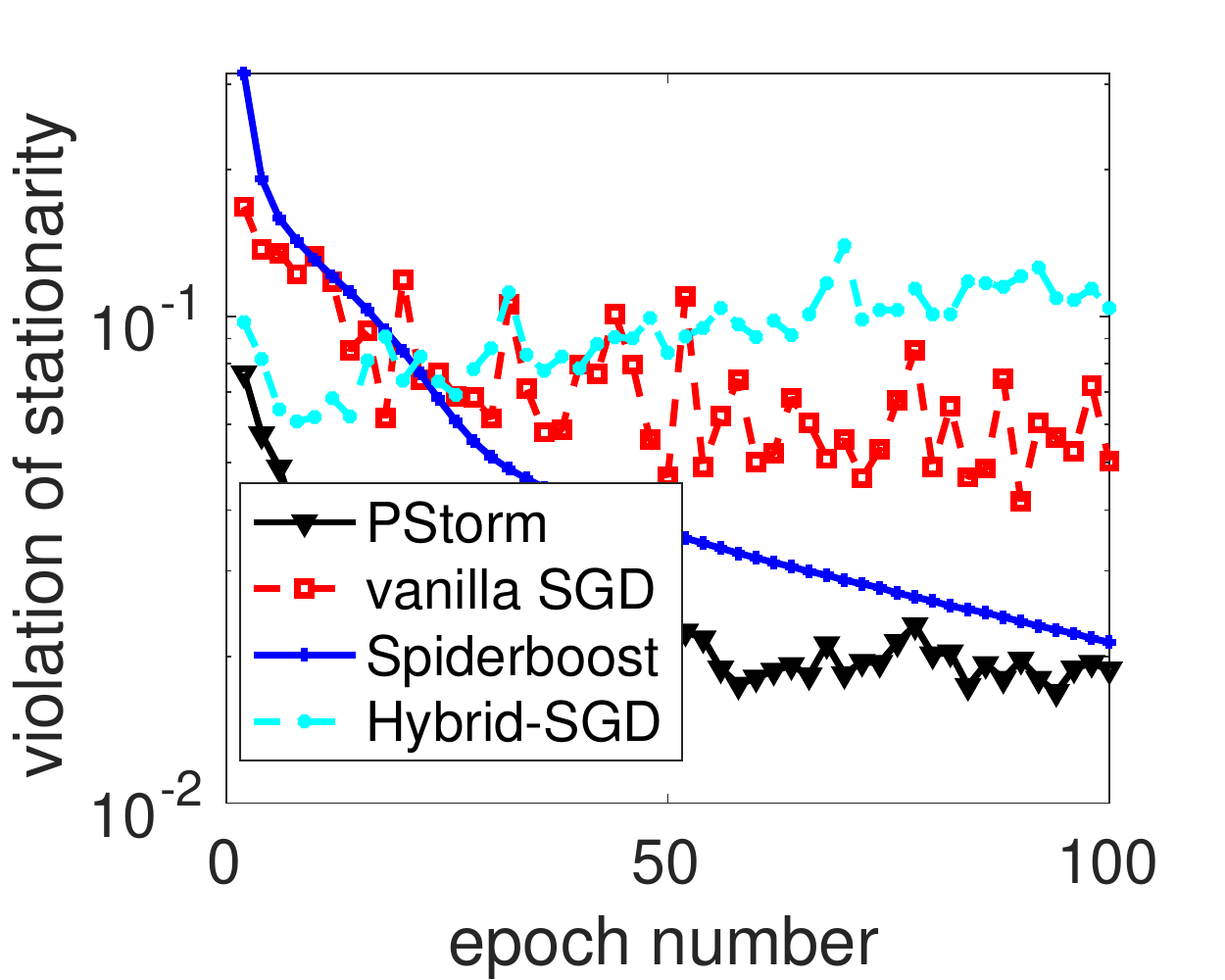} \\
\includegraphics[width=0.25\textwidth]{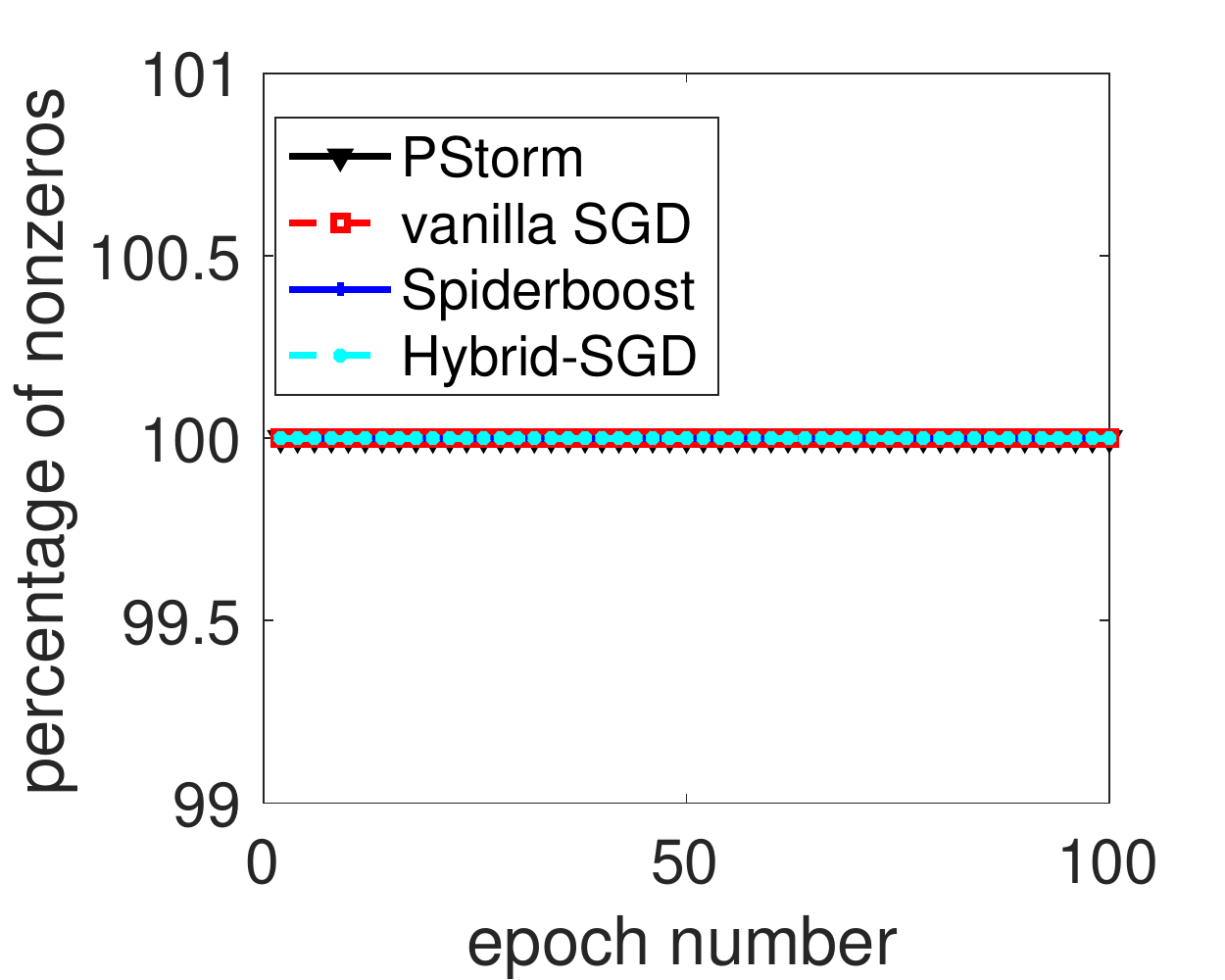}& 
\includegraphics[width=0.25\textwidth]{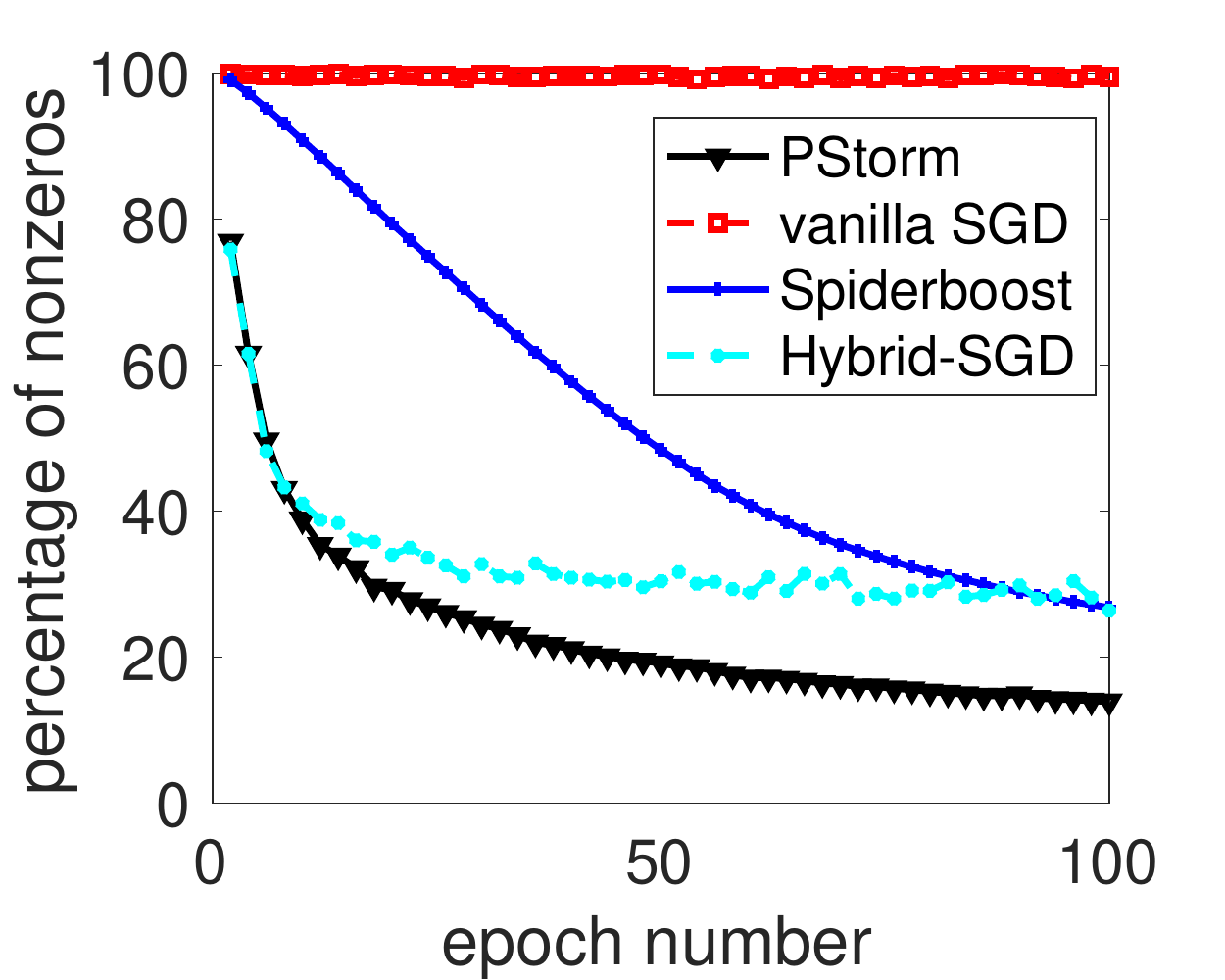} &
\includegraphics[width=0.25\textwidth]{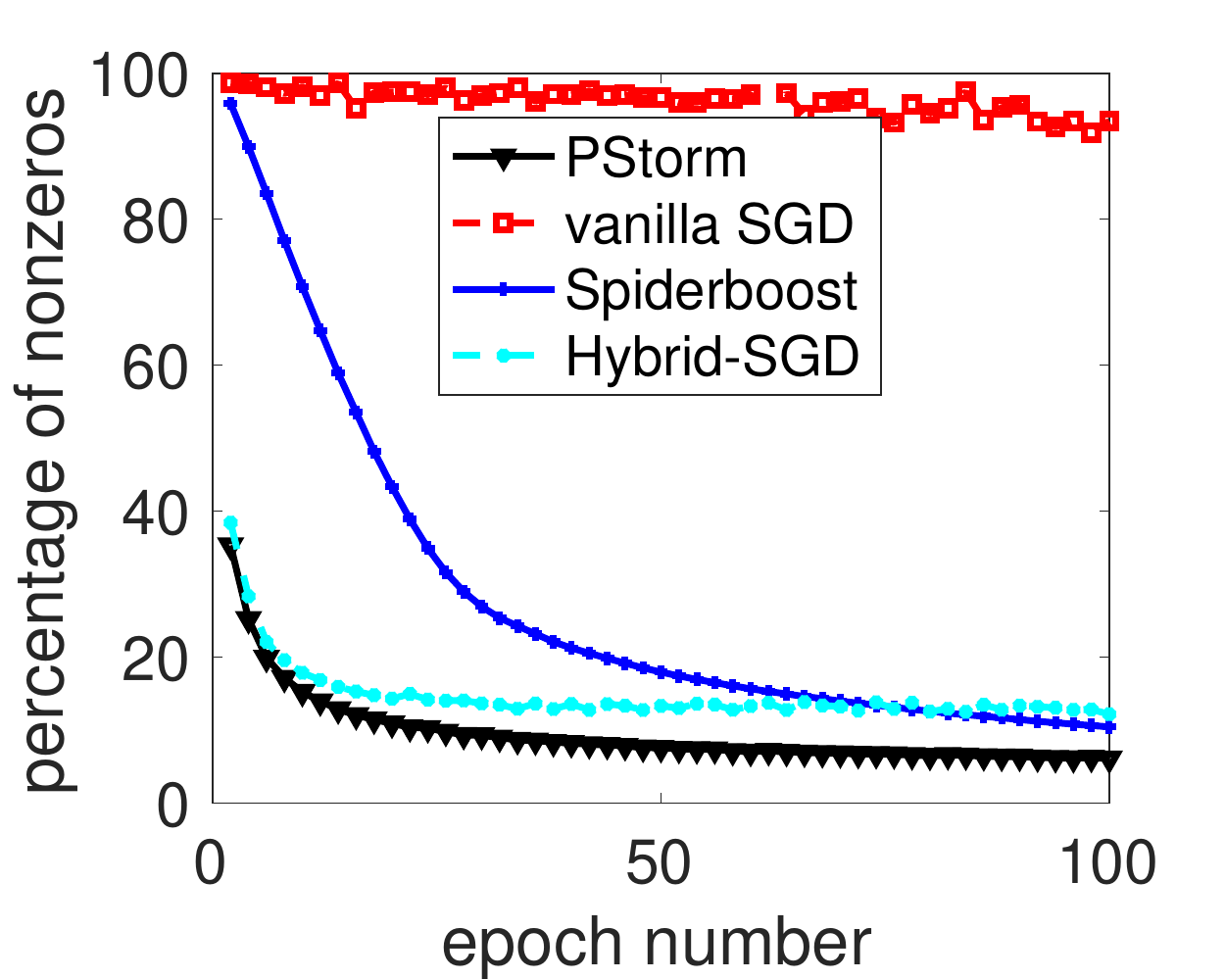} 
\end{tabular}
\end{center}
\caption{Results in terms of epoch by the proposed method PStorm, the vanilla SGD, Hybrid-SGD, and Spiderboost on training the model \eqref{eq:sparseDNN}. The first three methods use mini-batch $m = 32$.}\label{fig:sp-DNN-32}

\end{figure}

\subsection{Regularized Convolutional Neural Network}
In this subsection, we compare different methods on solving an $\ell_1$-regularized convolutional neural network, formulated as
\begin{equation}\label{eq:sparseAllCNN}
\min_{\vtheta} \frac{1}{N}\sum_{i=1}^N \ell\Big(\log \big(\mathrm{softmax}(\phi(\vx_i))\big), y_i\Big) + \lambda\|\vtheta\|_1.
\end{equation}
Similar to \eqref{eq:sparseDNN}, $\{(\vx_i,y_i)\}_{i=1}^N$ is a $c$-class training data set with $y_i\in \{1,\ldots,c\}$ for each $i$, $\vtheta$ contains all parameters of the neural network, $\ell$ denotes a loss function, the $\log$ function takes component-wise logarithm, $\phi$ represents the nonlinear transformation by the neural network, and $\lambda\ge0$ is a regularization parameter to trade off the loss and sparsity. In the test, we used the Cifar10 dataset \cite{krizhevsky2009learning} that has 50,000 training images and 10,000 testing images. In addition, we set $\ell$ to the cross entropy loss and $\phi$ to the all convolutional neural network (AllCNN) in \cite{springenberg2014striving} without data augmentation. The AllCNN has 9 convolutional layers.

We ran each method to 200 epochs. Mini-batch size was set to 100 for PStorm, the vanilla SGD, and Hybrid-SGD. The stepsizes of PStorm and the vanilla proximal SGD were tuned in the same way as in section~\ref{sec:npca}. For Spiderboost, we set $q=\lceil \sqrt{50000}\rceil =224 $ in \eqref{eq:vk-spider}, and its learning rate $\eta$ in \eqref{eq:spiderboost} was tuned by picking the best one from $\{0.01, 0.1, 0.5\}$. For Hybrid-SGD, we set its parameters in a way similar to that in section~\ref{sec:fnn} but chose the best pair of $(L, m_0)$ from $\{1, 10, 100\}\times \{10^2, 10^3, 10^4\}$. Results produced by the four methods are shown in Table~\ref{fig:sp-CNN} and Figure~\ref{fig:allcnn}. Again, each result in the table is the average of those at the last five epochs. From the results, we see that PStorm and Hybrid-SGD give similar training loss and testing accuracies. PStorm is slightly better than Hybrid-SGD, and the advantage of the former is more significant when $\lambda = 5\times 10^{-4}$. Spiderboost can give small violation of stationarity, but it tended to have significantly higher loss and lower accuracies. This is possibly because Spiderboost used larger batch size.

\begin{table}[h]
\begin{center}
\resizebox{0.99\textwidth}{!}{
\begin{tabular}{|c||cccc||cccc||cccc||cccc|}
\hline
Method & \multicolumn{4}{|c||}{PStorm} & \multicolumn{4}{|c||}{vanilla SGD} & \multicolumn{4}{|c||}{Spiderboost} & \multicolumn{4}{|c|}{Hybrid-SGD}\\\hline\hline
$\lambda$ & train & test & grad & density &  train & test & grad & density & train &  test & grad & density & train & test & grad & density \\\hline 
0.0 & \textbf{2.30e-2} & \textbf{89.74} & \textbf{0.10} & 100 
& 2.45e-1 & 85.61 & 0.76 & 100 
& 1.86 & 36.63 & 0.12 & 100 
& 5.26e-2 & 88.17 & 0.19 & 100 \\
2e-4 & \textbf{7.61e-1} & \textbf{89.40} & 4.17 & \textbf{44.91} 
& 9.42e-1 & 88.76 & 2.78 & 89.64 
& 2.93 & 20.43 & \textbf{0.89} & 53.79 
& 8.15e-1 & 88.03 & 1.68 & 72.78 \\
5e-4 & \textbf{1.15} & \textbf{88.53} & 5.94 & \textbf{19.87} 
& 2.15 & 86.62 & 5.55 & 40.64 
& 4.69 & 18.62 & \textbf{0.81} & 32.21 
& 1.75 & 86.71 & 6.09 & 60.78\\\hline
\end{tabular}
}
\end{center}
\caption{Results in terms of epoch by the proposed method PStorm, the vanilla SGD, Hybrid-SGD, and Spiderboost on training the model \eqref{eq:sparseAllCNN}. The first three methods use mini-batch $m = 100$. Each method runs to 200 epochs. ``train'' is for training loss; ``test'' is for testing accuracy;  ``grad'' is for the violation of stationarity; ``density'' is for the percentage of nonzeros in the solution. The best results for ``test'', ``grad'', and ``density'' are highlighted in \textbf{bold}.}\label{fig:sp-CNN}
\end{table}

\begin{figure}[h]
\begin{center}
\begin{tabular}{ccc}
$\lambda = 0$ & $\lambda = 2\times 10^{-4}$ & $\lambda = 5\times 10^{-4}$ \\
\includegraphics[width=0.24\textwidth]{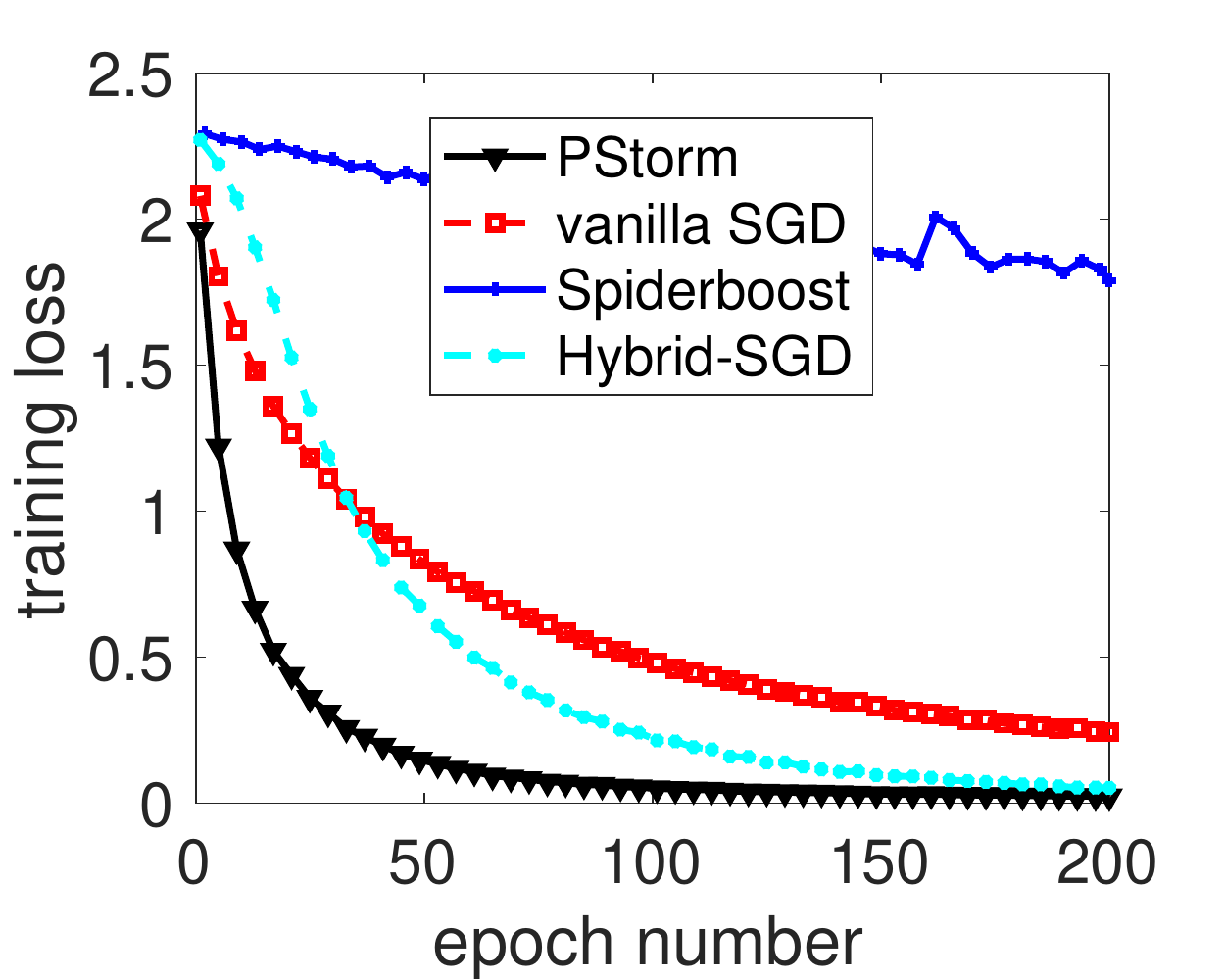} &
\includegraphics[width=0.24\textwidth]{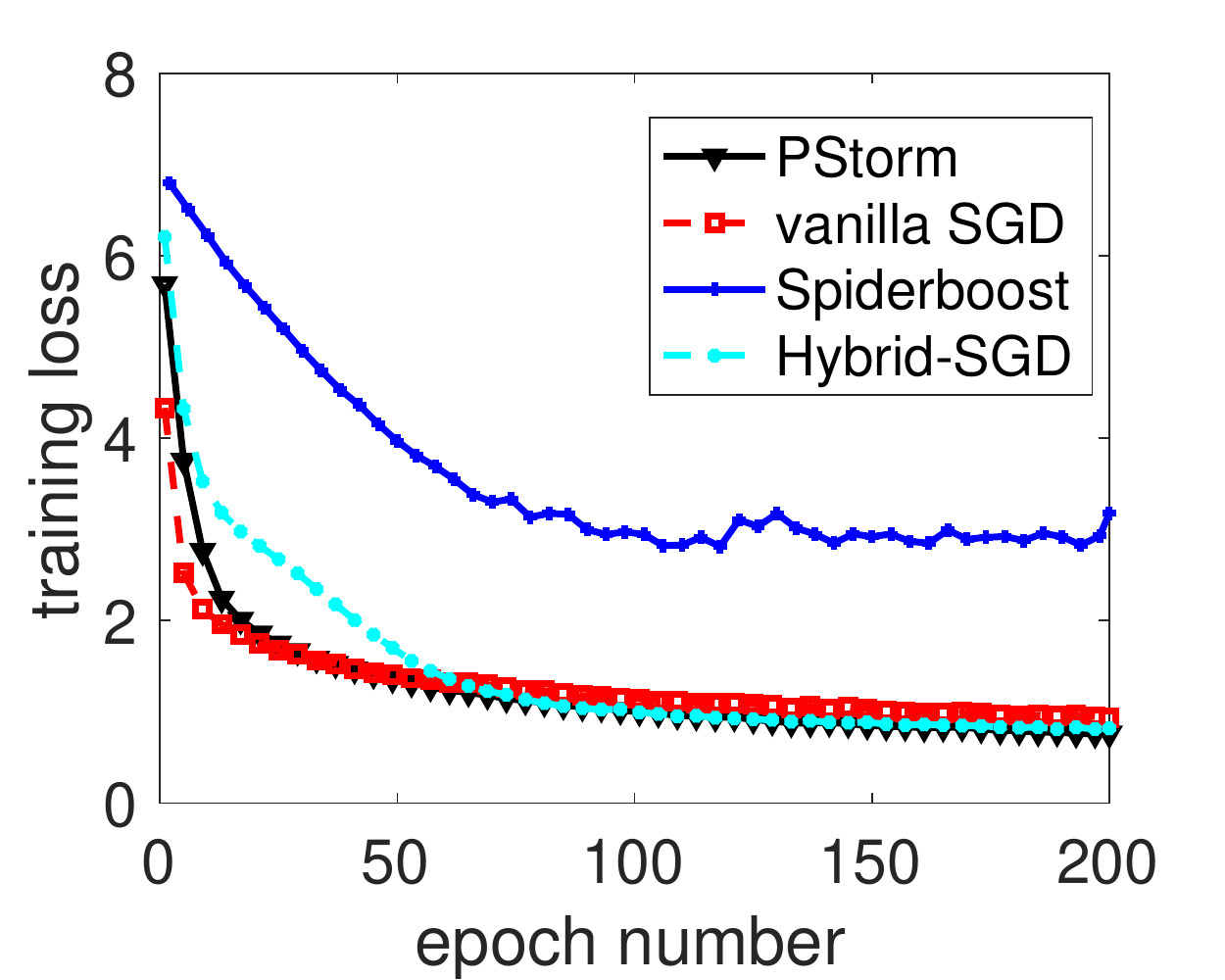}  &
\includegraphics[width=0.24\textwidth]{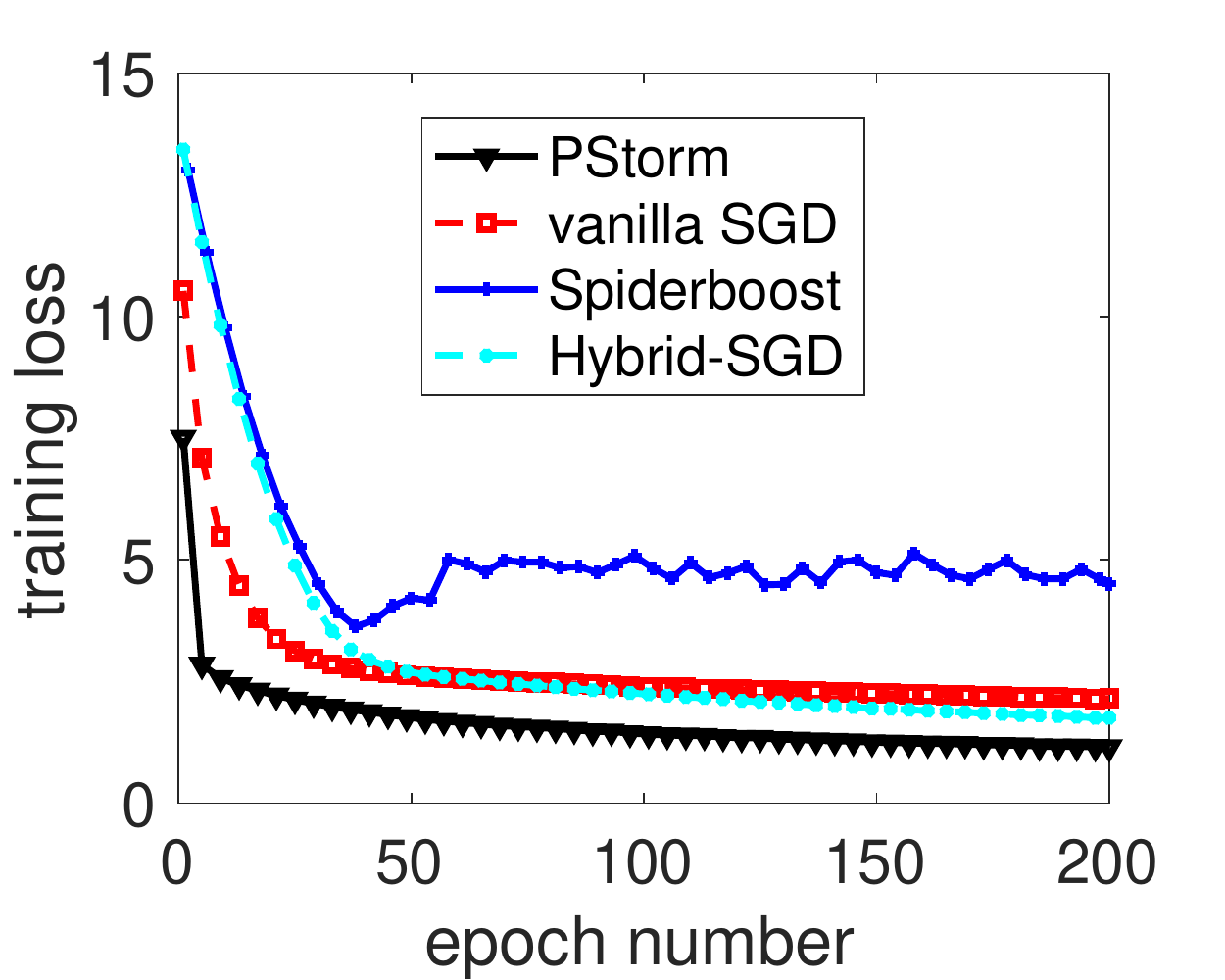}\\ 
\includegraphics[width=0.24\textwidth]{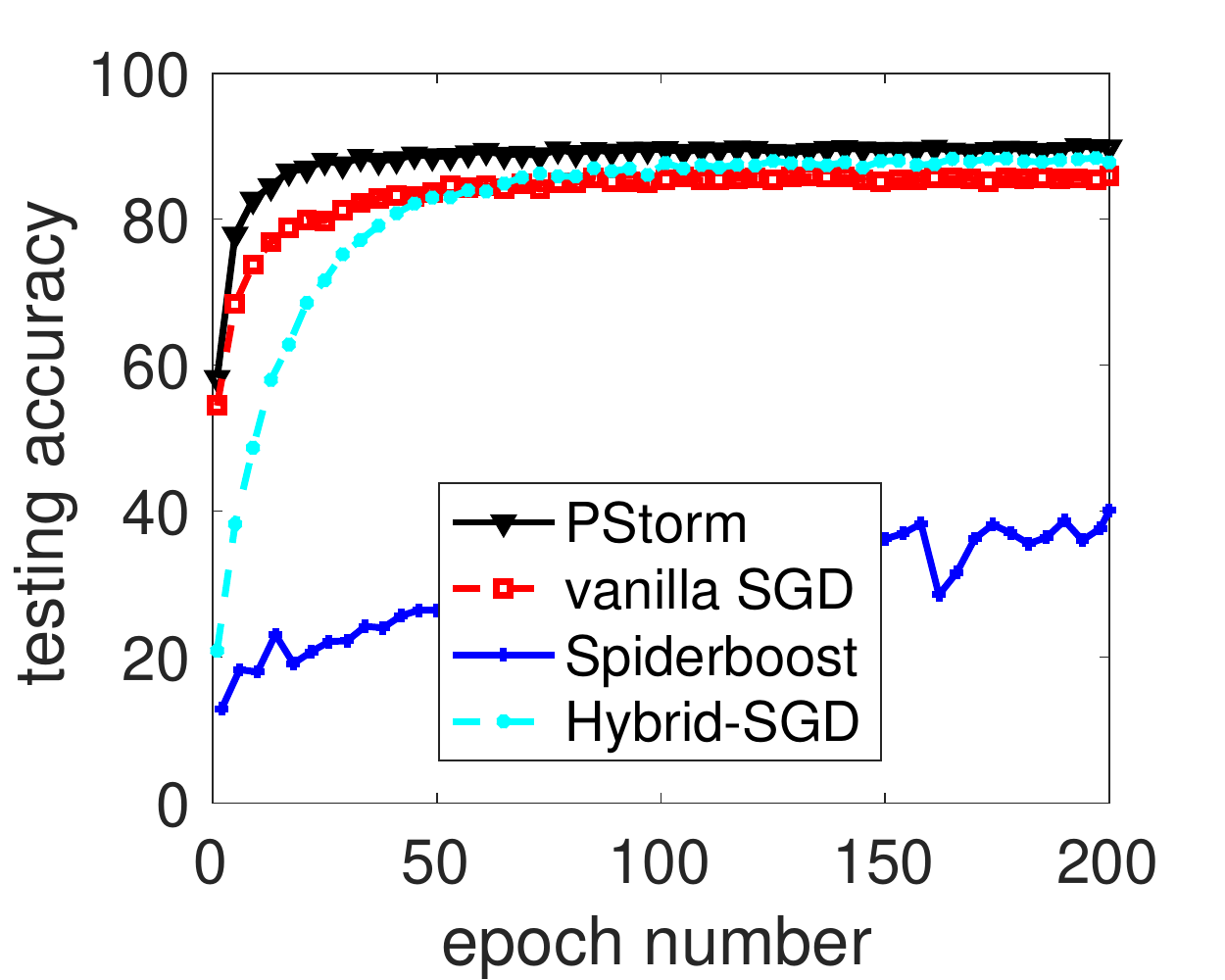} &
\includegraphics[width=0.24\textwidth]{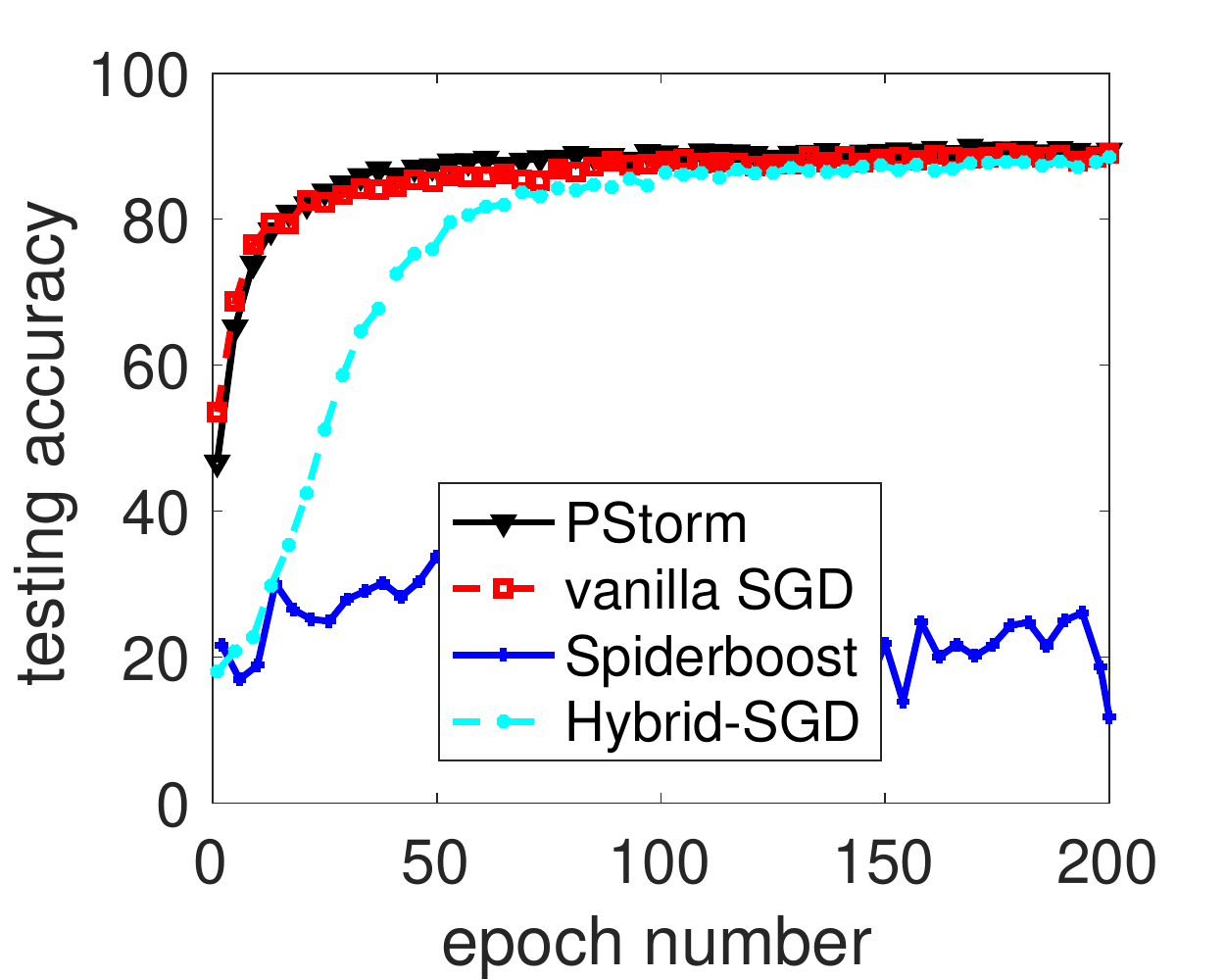} &
\includegraphics[width=0.24\textwidth]{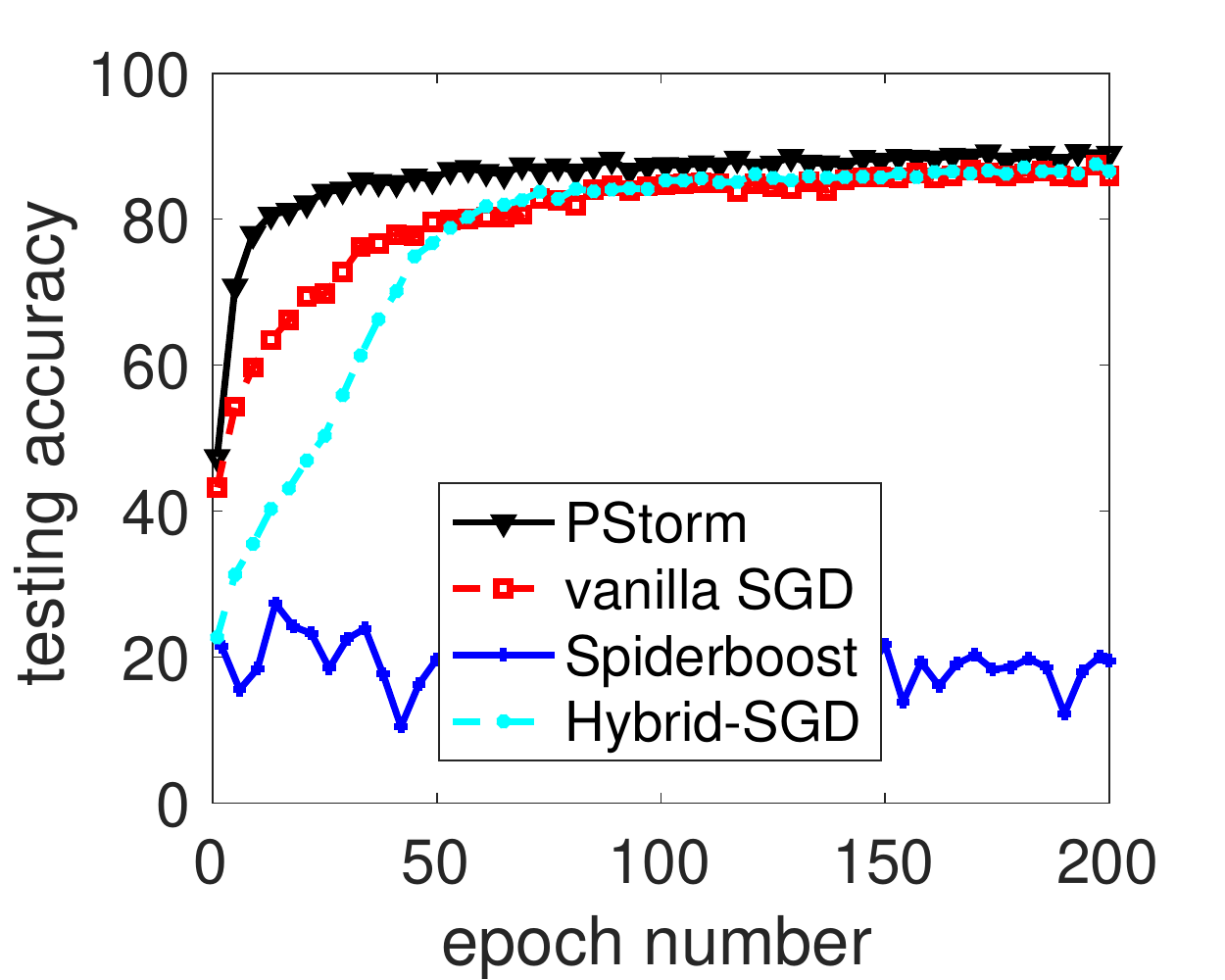} \\
\includegraphics[width=0.24\textwidth]{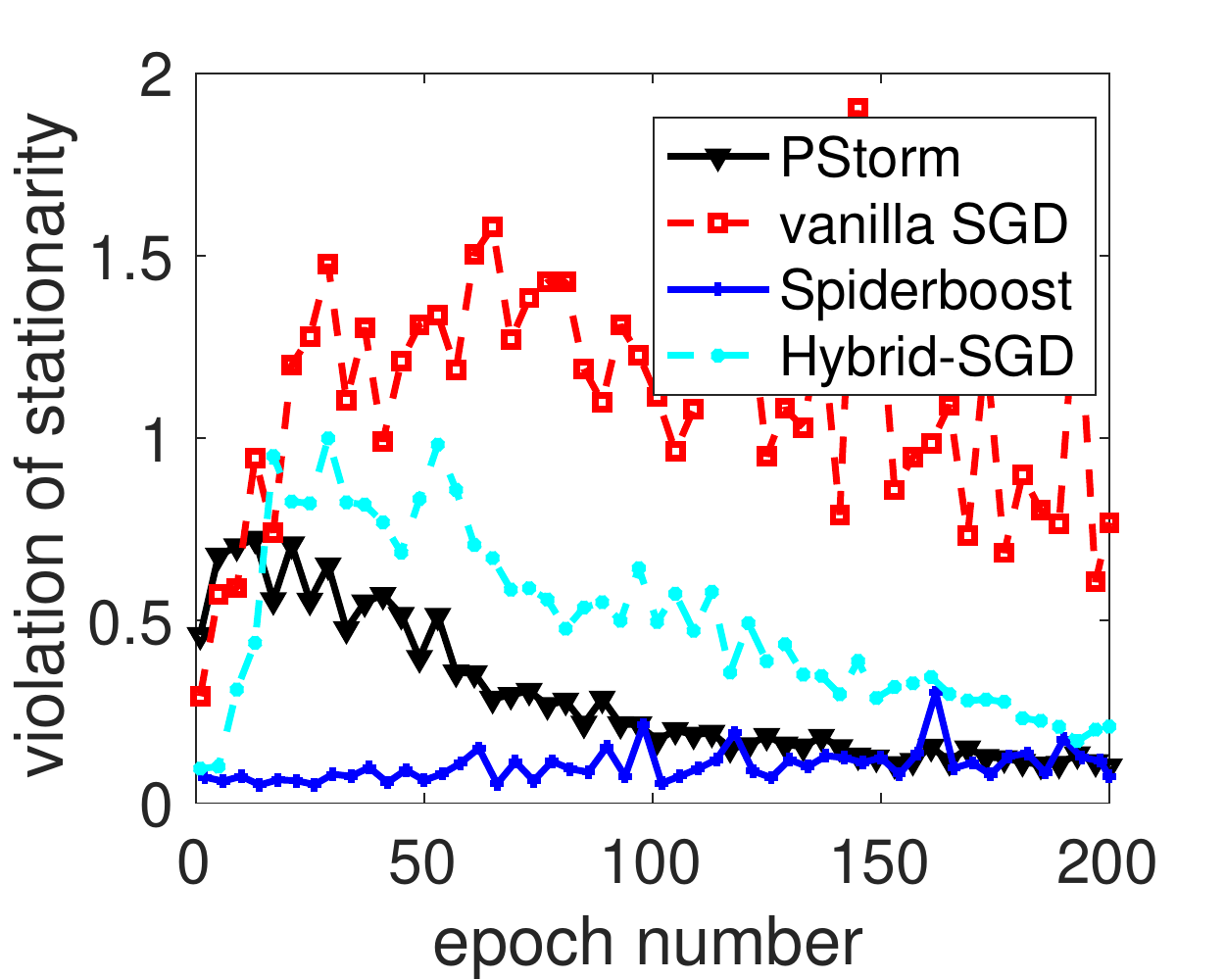} &
\includegraphics[width=0.24\textwidth]{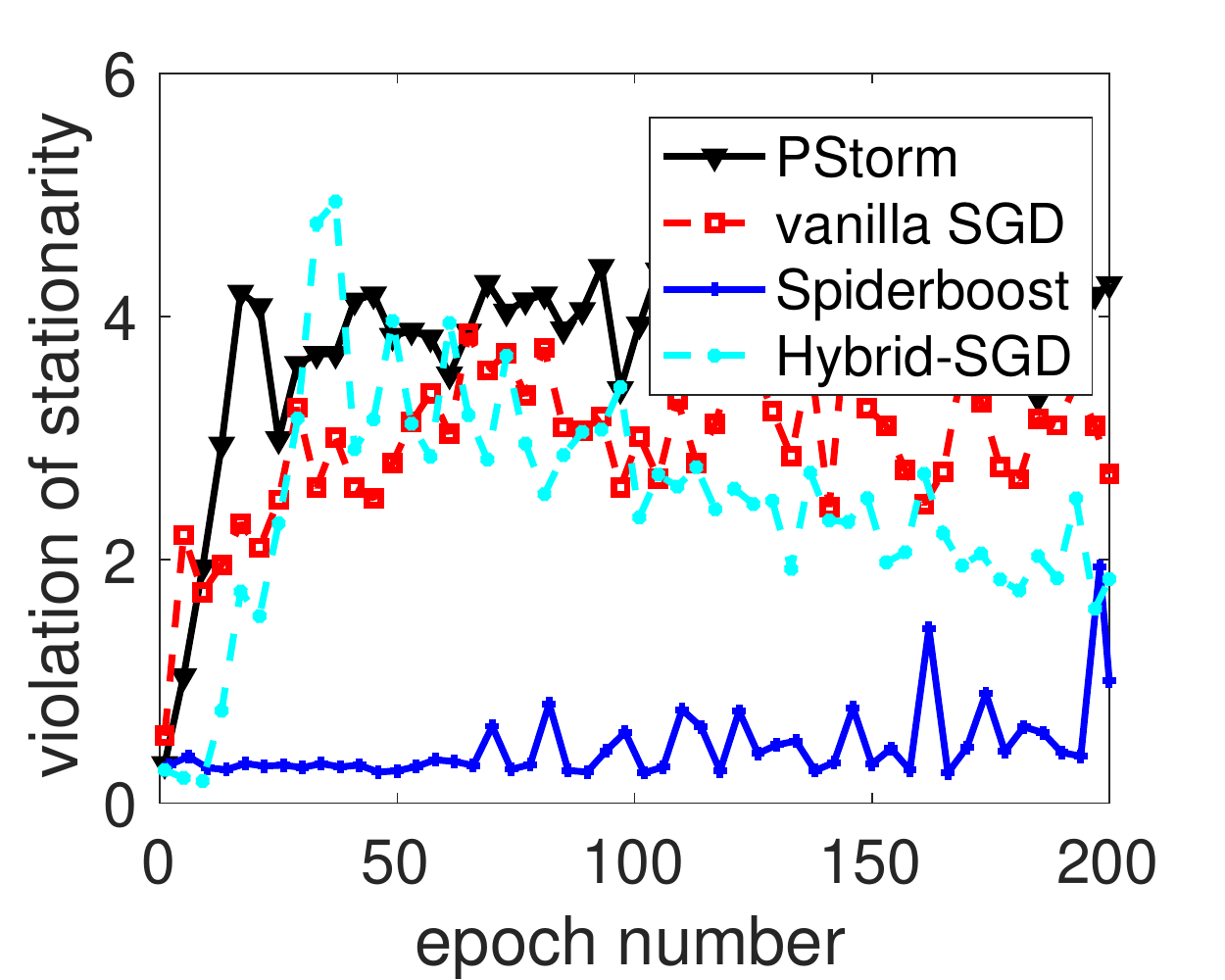} &
\includegraphics[width=0.24\textwidth]{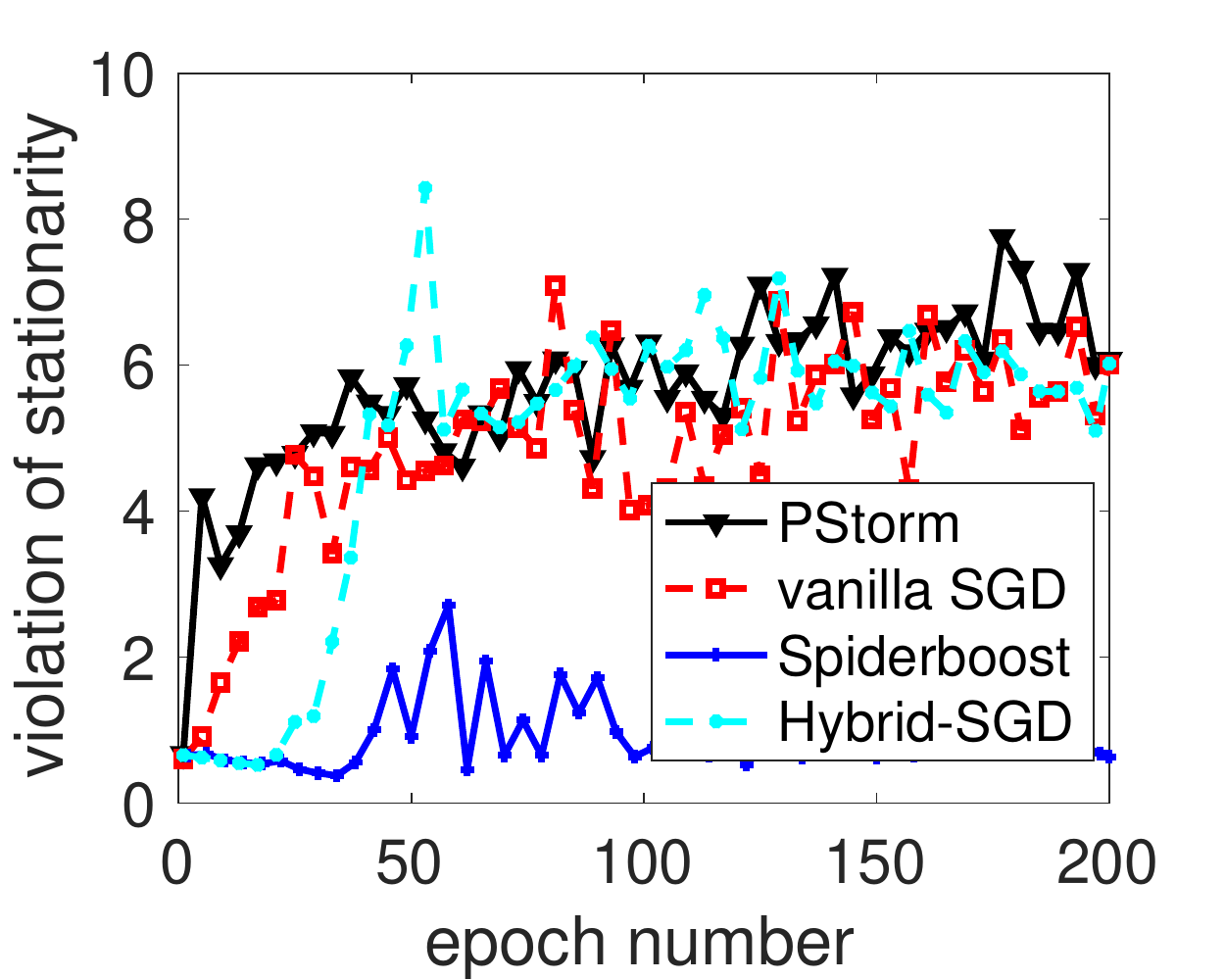} \\
\includegraphics[width=0.24\textwidth]{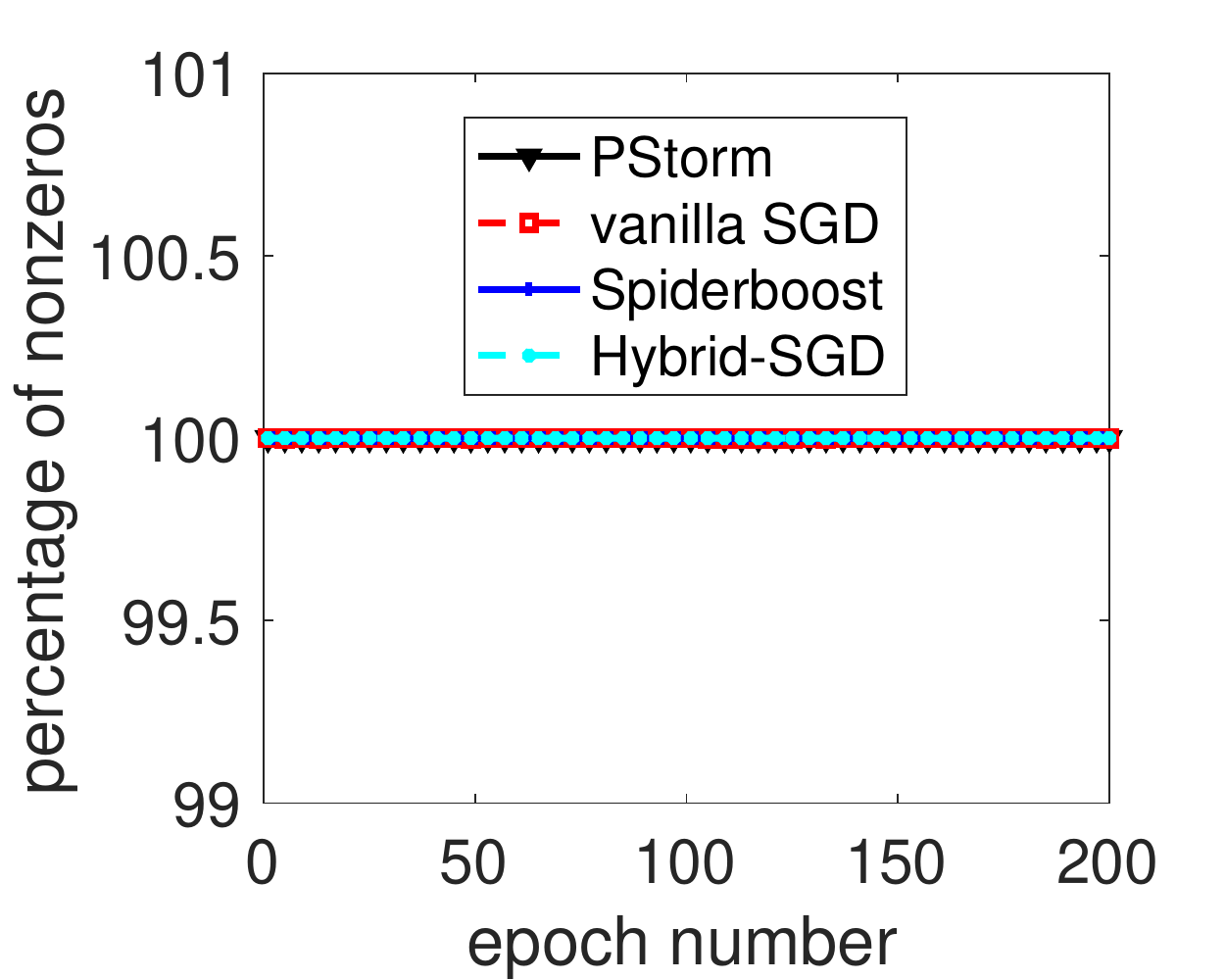}&
\includegraphics[width=0.24\textwidth]{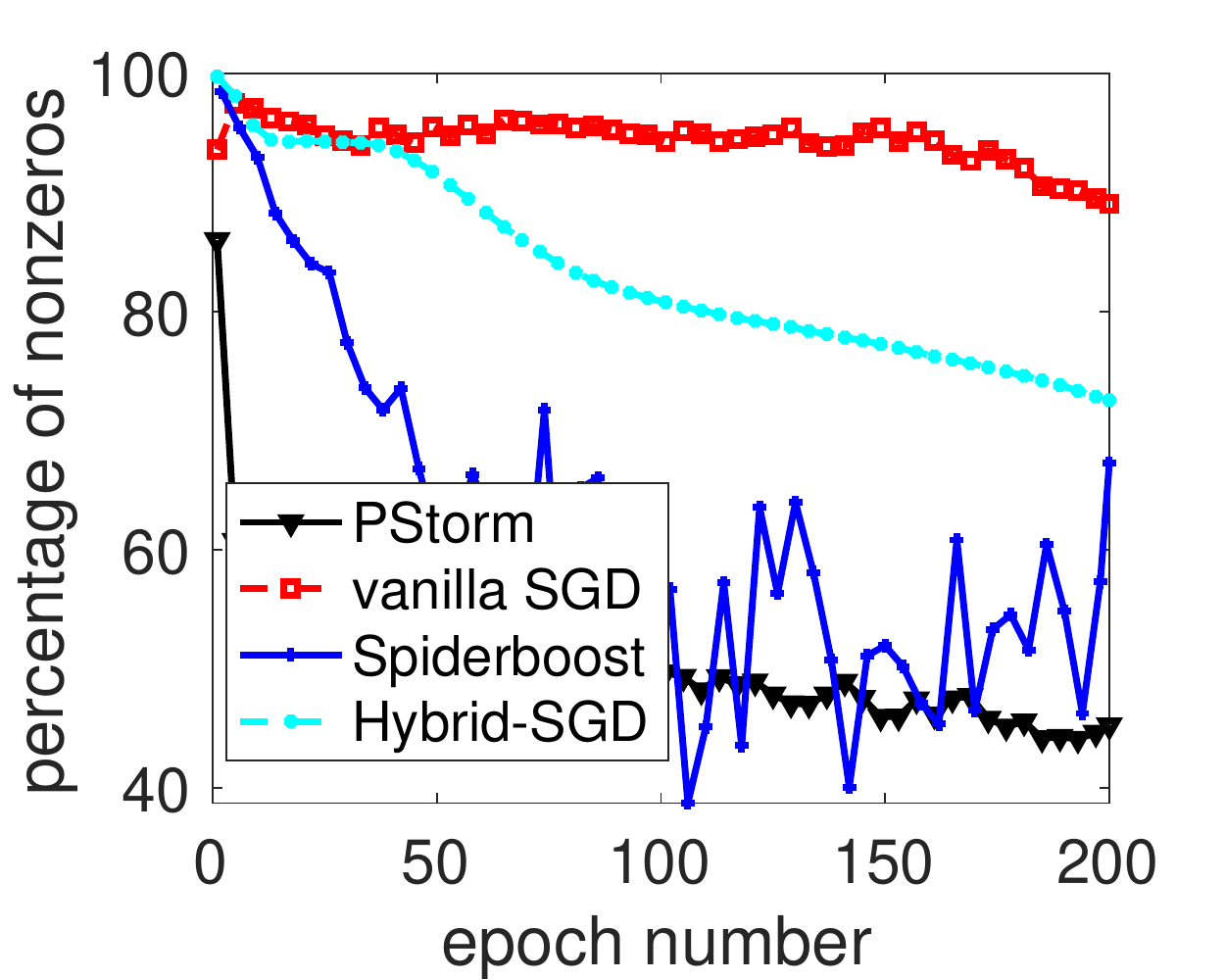} &
\includegraphics[width=0.24\textwidth]{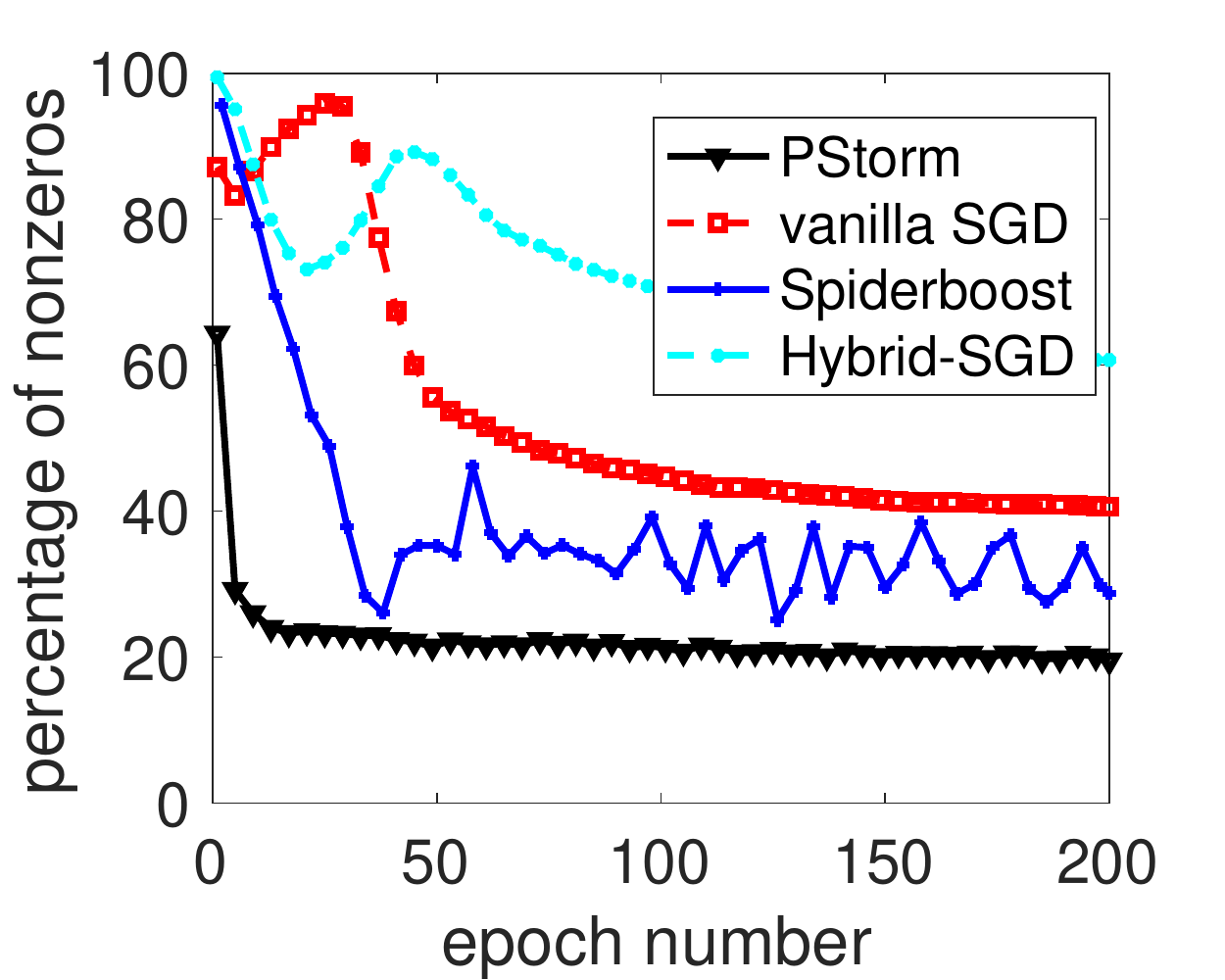}
\end{tabular}
\end{center}
\caption{Results in terms of epoch by the proposed method PStorm, the vanilla SGD, Hybrid-SGD, and Spiderboost on training the model \eqref{eq:sparseAllCNN}. The first three methods use mini-batch $m = 100$.}\label{fig:allcnn}
\end{figure}

\section{Conclusions}\label{sec:conclusion}
We have presented a momentum-based variance-reduced mirror-prox stochastic gradient method for solving nonconvex nonsmooth problems, where the nonsmooth term is assumed to be closed convex. The method, named PStorm, requires only one data sample for each update. It is the first $O(1)$-sample-based method that achieves the optimal complexity result $O(\vareps^{-3})$ under a mean-squared smoothness condition for solving nonconvex nonsmooth problems. The $O(1)$-sample update is important in machine learning because small-batch training can lead to good generalization. On training sparse regularized neural networks, PStorm can perform better than two other optimal stochastic methods and consistently better than the vanilla stochastic gradient method.

\bibliographystyle{abbrv}
\bibliography{optim,myPub}

\end{document}